\documentclass[12pt]{article}
\usepackage{fullpage}
\usepackage{amsmath,amsfonts,mathtools,enumerate}
\usepackage{amsthm}
\usepackage{amssymb}
\usepackage{amscd}
\usepackage{xypic}
\usepackage{verbatim}
\usepackage[plainpages=false,colorlinks,hyperindex,pdfpagemode=None,bookmarksopen,linkcolor=red,citecolor=blue,urlcolor=blue]{hyperref}
\usepackage{pdflscape}
\usepackage{stmaryrd}
\usepackage{mathabx}
\usepackage{multirow}
\usepackage{appendix}
\usepackage{graphicx}

\usepackage{color}

\theoremstyle{plain}
 
\renewcommand{\marginpar}[1] {  }
\renewcommand{\comment}[1] {  }

\usepackage[all]{xy}
\catcode`\é=\active\def é{\'e}
\catcode`\è=\active\def è{\`e}
\catcode`\à=\active\def à{\`a}
\catcode`\ù=\active\def ù{\`u}

\catcode`\ê=\active\def ê{\^{e}}
\catcode`\î=\active\def î{\^{i}}
\catcode`\ô=\active\def ô{\^{o}}
\catcode`\û=\active\def û{\^{u}}
\catcode`\ç=\active\def ç{\c{c}}

\newtheorem{theo}{Theorem}[section]

\newtheorem*{theo-nn}{Theorem}
\newtheorem{theoa}{Theorem}

\newtheorem{lem}[theo]{Lemma}
\newtheorem{prop}[theo]{Proposition}
\newtheorem{cor}[theo]{Corollary}
\newtheorem{defi}[theo]{Definition}
\newtheorem{example}[theo]{Example}

\theoremstyle{remark}
\newtheorem{rem}[theo]{Remark}
\newtheorem{rema}[theoa]{Remark}
\numberwithin{equation}{section}

\def \ga{\mathfrak{a}}
\def \gb{\mathfrak{b}}
\def\gg{\mathfrak{g}}
\def\gk{\mathfrak{k}}
\def\gh{\mathfrak{h}}
\def\gu{\mathfrak{u}}
\def\gl{\mathfrak{l}}
\def\gj{\mathfrak{j}}
\def\gt{\mathfrak{t}}
\def\gm{\mathfrak{m}}
\def\gn{\mathfrak{n}}
\def\gs{\mathfrak{s}}
\def\gr{\mathfrak{r}}
\def\gq{\mathfrak{q}}

\def\gZ{\mathfrak{Z}}

\def\la{\lambda}

\def\Si{\Sigma}

\def\tbeta{\tilde{\beta}}

\def\Id{{\operatorname{Id}}}
\def\ad{\operatorname{ad}}
\def\Ad{\operatorname{Ad}}
\def\gr{\operatorname{gr}}
\def\conv{\operatorname{conv}}

\def\Gr{\operatorname{Gr}}

\def\cA{{\mathcal A}}

\def\cC{{\mathcal C}}

\def\cF{{\mathcal F}}

\def\cH{{\mathcal H}}
\def\cI{{\mathcal I}}
\def\cJ{{\mathcal J}}

\def\cL{{\mathcal L}}
\def\cM{{\mathcal M}}
\def\cO{{\mathcal O}}

\def\cQ{{\mathcal Q}}

\def\cU{{\mathcal U}}

\def\cW{{\mathcal W}}

\def\cZ{{\mathcal Z}}

\DeclareFontFamily{OT1}{rsfs}{}
\DeclareFontShape{OT1}{rsfs}{n}{it}{<-> rsfs10}{}
\DeclareMathAlphabet{\mathscr}{OT1}{rsfs}{n}{it}

\def\hat{\widehat}
\def\C{\mathbb C}
\def\D{\mathbb D}
\def\Q{\mathbb Q}
\def\R{\mathbb R}

\def\N{\mathbb N}
\def\Q{\mathbb Q}
\def\R{\mathbb R}
\def\X{\mathbb X}
\def\Y{\mathbb Y}

\def\Z{\mathbb Z}

\def\bD{\mathbf D}

\def\bW{\mathbf W}

\def\bv{\mathbf{v}}
\def\bw{\mathbf{w}}
\def\bs{\mathbf s}

\def\sA{\underline{ A}}

\def\sG{\underline{ G}}
\def\sGL{\underline{ GL}}
\def\SL{\operatorname{SL}}

\def\sH{\underline{ H}}

\def\sJ{\underline{ J}}
\def\sK{\underline{ K}}
\def\sL{\underline{ L}}
\def\sM{\underline{ M}}
\def\sN{\underline{ N}}
\def\sO{\underline{ O}}
\def\sP{\underline{ P}}

\def\sQ{\underline{ Q}}
\def\sR{\underline{ R}}

\def\sU{\underline{ U}}

\def\sZ{\underline{ Z}}

\def\Reel{\operatorname{Re}\,}

\def\Lie{\mathrm{Lie}\,}
\def\Ker{\operatorname{Ker}\,}
\def\dim{\operatorname{dim}\,}
\def\Spec{\operatorname{Spec}}
\def\Pol{\operatorname{Pol}}

\def\la{\langle}
\def\ra{\rangle}
\DeclareMathOperator{\Symm}{\mathrm{Symm}}

 \def\beq{\begin{equation}}
\def\eeq{\end{equation}}

\newenvironment{res}
               {\begin{equation}\begin{minipage}{0.85\textwidth}}
               {\end{minipage}\end{equation}}
\def\ber{\begin{res}}
\def\eer{\end{res}}
\newenvironment{res-nn}
               {$$\begin{minipage}{0.85\textwidth}}
               {\end{minipage}$$}

\def\qed{{\null\hfill\ \raise3pt\hbox{\framebox[0.1in]{}}\break\null}}

\newcommand{\Hom}{\operatorname{Hom}}
\newcommand{\End}{\operatorname{End}}

\title{The constant term of tempered functions on a real spherical space}

\author{Patrick {Delorme}\thanks{The first author was supported by a grant of Agence Nationale de la Recherche 
with reference ANR-13-BS01-0012 FERPLAY.}  \and Bernhard Kr\"otz \and Sofiane {Souaifi} \and (with an Appendix by Rapha\"el Beuzart-Plessis)}

\begin{document}

\maketitle

\begin{abstract}
Let $Z$ be a unimodular real spherical space. 
We develop a theory of constant terms for tempered functions on $Z$ which parallels the work of Harish-Chandra. 
The constant terms $f_I$ of an eigenfunction $f$ are parametrized by subsets $I$ of the set $S$ of 
spherical roots which determine the fine geometry of $Z$ at infinity. 
Constant terms are transitive i.e., $(f_J)_I=f_I$ for $I\subset J$, and our main result is a quantitative bound of the difference 
$f-f_I$, which is uniform in the parameter of the eigenfunction.
\end{abstract}

\addcontentsline{toc}{section}{Introduction}
\tableofcontents

\section*{Introduction}

Real spherical spaces are the natural algebraic homogeneous structures $Z=G/H$ attached to 
a real reductive group $G$. Formally, one defines {\em real spherical}  by the existence of a minimal parabolic subgroup $P\subset G$ with $PH$ open in $G$.  On a more informal level, 
one could define real spherical spaces as the class of algebraic homogeneous spaces $Z=G/H$ which allow a uniform treatment of spectral theory, i.e., admit explicit Fourier analysis for $L^2(Z)$.  

\par  Real spherical spaces provide a wide class of algebraic homogeneous spaces. 
Important examples are the group $G$ itself, viewed as a homogeneous space under its both sided symmetries $G\simeq G\times G/ \operatorname{diag} (G)$, and, more generally, all symmetric spaces.
In case $H$ is reductive, a classification of all infinitesimal real spherical pairs $(\Lie G, \Lie H)$ was recently obtained 
and we refer to \cite{kkps, kkps2} for the tables.

\par Harmonic analysis on spherical spaces was initiated by Sakellaridis and Venkatesh in the context of $p$-adic 
absolutely spherical varieties of wavefront type \cite{sak-venk}. In particular, they developed 
a theory of asymptotics for smooth functions generalizing Harish-Chandra's concept of constant term for real 
reductive groups.

\par Harish-Chandra's approach to the Plancherel formula for $L^2(G)$, a cornerstone of 20th century mathematics (cf.~\cite{hc3}),  was based on his theory of the constant term \cite{hc1,hc2}
and his epochal work on the determination of the discrete series  \cite{hcd1,hcd2}. In more precision, 
constant terms were introduced in \cite{hc1} and then made uniform in the representation parameter 
in \cite{hc2} by using the strong results on the discrete series in \cite{hcd1,hcd2}.  Also in Harish-Chandra's approach 
towards the Plancherel formula for  $p$-adic reductive groups, the constant term concept played an important 
role and we refer to Waldspurger's work \cite{wald} for a complete account (the constant term in \cite{wald} is called weak constant term). Likewise, the Plancherel theory 
of Sakellaridis and Venkatesh for $p$-adic spherical spaces is founded on their more general theory of asymptotics. 

\par Carmona introduced a theory of constant term for real symmetric spaces \cite{carmona} parallel 
to \cite{hc1, hc2}, with the uniformity in the representation parameter relying on the description 
of the discrete series by Oshima-Matsuki \cite{OM}.
This concept of constant term then crucially entered the proofs
of  Delorme \cite{Delorme} and van den Ban--Schlichtkrull \cite{vdBS} of the Plancherel formula for real symmetric spaces.

\par Motivated by \cite{sak-venk}, we develop in this paper a complete theory of constant term for real spherical spaces generalizing the works 
of Harish-Chandra \cite{hc1,hc2} and Carmona  \cite{carmona}.  Let us point out that our results hold in 
full generality for all real spherical spaces, i.e., in contrast to \cite{sak-venk}, we are not required 
to make any limiting geometric assumptions on $Z$ such as absolutely spherical or of wavefront type. Further, we 
do not need to make any assumptions on the discrete spectrum as in \cite{sak-venk}. This is because of 
the recently obtained  spectral gap theorem for the discrete series on a real spherical space \cite{kkos},
which then implies the uniformity of the constant term approximation in the representation parameter. The results of this paper then will be applied in the forthcoming paper 
\cite{planch-sph}, where we derive the Bernstein decomposition of $L^2(Z)$, which is a major step towards the Plancherel formula for $Z$. 

\par Let us describe the results more precisely. In this introduction, $G$ is the group of real points of 
a connected reductive algebraic group $\sG$ defined over $\R$, and $H=\sH(\R)$ 
for an algebraic subgroup $\sH$ of $\sG$ defined over $\R$.
Furthermore, we assume that $Z$ is unimodular, i.e., $Z$ carries a positive $G$-invariant Radon measure.
We will say that $A$ is a split torus of $G$ if $A=\sA(\R)$, where $\sA$ is a split 
$\R$-torus of $\sG$. 

\par Central to the geometric theory of real spherical spaces $Z=G/H$ 
is the local structure theorem (cf.~\cite[Theorem~2.3]{kks0} and Subsection 
\ref{lst-sect}), which associates a parabolic subgroup {$Q\supset P$}, 
said $Z$-adapted to $P$, with Levi decomposition $Q=LU$.

\par Let $A$ be a maximal split torus of $G$ which we choose to be contained in $L$ and set $A_H:=A \cap H$. 
We define $A_Z$ to be the identity component of $A/A_H$ and recall the spherical roots $S$ as defined in e.g., \cite[Section~3.2]{kks} or Subsection \ref{sect-1.2} below. Suitably normalized the spherical roots are the simple roots
for a certain root system on $\ga_Z=\Lie A_Z$ and give rise to the co-simplicial compression cone 
$\ga_Z^-:=\{ X \in \ga_Z\mid \alpha(X)\leq 0,\alpha\in S \}$, see \cite{kk}. Set $A_Z^-:=\exp (\ga_Z^-) \subset A_Z$. 

\par We move on to boundary degenerations $\gh_I$ of $\gh$, which are parametrized by subsets 
$I\subset S$. These geometric objects show up naturally in the compactification 
theory of $Z$ (see \cite{kkss}, \cite{kk} and Section \ref{sect-2}), which in turn is closely related 
to the polar decomposition (see \eqref{polar-intro} below). In more detail, let
$\ga_I=\bigcap_{\alpha\in I}\Ker \alpha\subset \ga_Z$  and pick $X\in\ga_I^{--}=\{X\in\ga_I\mid\alpha(X)<0,\alpha\in S\setminus I\}$. Let $H_I$ be the analytic subgroup of $G$ with Lie algebra 
$$\Lie H_I=\lim_{t\to +\infty}e^{t\,\mathrm{ad}\,X}\Lie H\,,$$
where the limit is taken in the Grassmannian $\operatorname{Gr}(\gg)$ of $\gg=\Lie G$ and does not depend on $X$. 
If we denote by $z_0=H$ the standard base point of $Z$, then one can view $H_I$ 
(up to components) as the invariance group of the asymptotic directions
$\gamma_X(t):=\exp(tX)\cdot z_0$ for $t\to \infty$ and $X\in \ga_I^{--}$. Phrased  more geometrically, $Z_I:=G/H_I$ is (up to cover) asymptotically tangent to $Z$ in direction of the curves $\gamma_X$, $X\in \ga_I^{--}$.

\par As a deformation of $Z$, the space $Z_I$ is real spherical. Further, one has 
$A_{Z_I}=A_{Z}$ naturally, but the compression cone 
$\ga_{Z_I}^-=\{ X \in \ga_Z\mid \alpha(X)\leq 0,\,\alpha\in I\}$ becomes larger. 
In particular, $\ga_I$ is the edge of the cone $\ga_{Z_I}^-$, which translates into the fact that $A_I=\exp(\ga_I)$
acts on $Z_I$ from the right, commuting with the left action of $G$.

\par The general concept of ``constant term'' is to approximate functions $f$ on $Z$ in directions
$\gamma_X$, $X\in \ga_I^{--}$, by functions $f_I$, called constant terms,  on $Z_I$.  The notion ``constant'' then refers to the fact 
that $f_I$ should transform finitely under the right action of $A_I$. 

\par The  appropriate class of functions for which this works are tempered eigenfunctions on $Z$. 
In order to define them, we need to recall the polar decomposition 
which asserts
\begin{equation}\label{polar-intro}
Z=\Omega A_Z^- \cW\cdot z_0\, ,\tag{1}
\end{equation}  
for a compact subset $\Omega \subset G$ 
and a certain finite subset $\cW$ of $G$ which parametrizes 
the open $P$-orbits  in $Z$ (see Lemma~\ref{lem-decomp-cW} and Remark \ref{rem-cc} below for more explicit expressions of the elements of $\cW$).

\par
Let $\rho_Q$ be the half sum of the roots of $\ga$ in $\Lie U$. Actually, as $Z$ is unimodular, $\rho_Q\in\ga_Z^*$.

\par For $f\in C^\infty(Z)$ and $N\in \N$, we set 
$$q_N(f)=\sup_{g\in\Omega,w\in\cW,a\in A_Z^-}a^{-\rho_Q}(1+\Vert\log a\Vert)^{-N}\vert f(g aw\cdot z_0)\vert$$
and define $C_{temp,N}^\infty(Z)$ as the space of all $f\in C^\infty(Z)$ 
such that, for all $u$ in the enveloping algebra $\cU(\gg)$ of the complexification $\gg_\C$ of $\gg$,
$$q_{N,u}(f):=q_N(L_uf)$$
is finite. The semi-norms $q_{N,u}$  induce a Fr\'echet structure on $C_{temp,N}^\infty(Z)$ for which the $G$-action is smooth and of moderate growth  (in \cite{bk}, these are called $SF$-representations). 
We define the space of {\em tempered functions} $C_{temp}^\infty(Z)=\bigcup_{N\in\N} C_{temp,N}^\infty(Z)$ and endow it with the inductive 
limit topology.

\par  We denote by $\cZ(\gg)$ the center of $\cU(\gg)$ and define $\cA_{temp}(Z)$, resp. 
$\cA_{temp,N}(Z)$, 
as the subspace of $C_{temp}^\infty(Z)$, resp. $C_{temp,N}^\infty(Z)$, consisting of $\cZ(\gg)$-finite functions. 

\begin{rema} Functions $f\in \cA_{temp}(Z)$ can be described suitably in terms of representation theory. 
A variant of  Frobenius reciprocity implies that elements $f\in \cA_{temp}(Z)$ can be expressed as generalized matrix coefficients 
$$f(gH) =  m_{\eta, v}(gH):=\eta (\pi(g)^{-1}v)\,,$$
for $v\in V^\infty$, where $(\pi, V^\infty)$ is a $\cZ(\gg)$-finite $SF$-representation of $G$ and $\eta: V^\infty\to \C$ 
a $H$-invariant continuous functional. If $V^{-\infty}$ denotes the continuous dual of $V^\infty$, then 
an element $\eta\in (V^{-\infty})^H$  is called {\em $Z$-tempered} provided $m_{\eta, v}\in\cA_{temp}(Z)$
for all $v\in V^\infty$. We denote the corresponding subspace by 
$(V^{-\infty})^H_{temp}$.
\end{rema}

\par For $I\subset S$, we choose a set $\cW_I\subset G$ parametrizing the open $P$-orbits on $Z_I$.
Then it is the content of Section \ref {sect-2} that there is a map ${\bf m}: \cW_I \to \cW$ 
obtained from a natural matching of open $P$-orbits on $Z_I$ with open $P$-orbits on $Z$. 
As $Z_I$ is a real spherical space, we can define $C_{temp,N}^\infty(Z_I)$ and $\cA_{temp,N}(Z_I)$ as before.

\par For $\cJ$ a finite codimensional ideal of the center $\cZ(\gg)$ of $\cU(\gg)$ we denote 
by  $\cA_{temp,N}(Z:\cJ)$ the space of elements of $\cA_{temp,N}(Z)$ which are annihilated by $\cJ$. Note that, by definition, there exists for each $f\in \cA_{temp,N}(Z)$ a co-finite ideal $\cJ$ such that $\cA_{temp,N}(Z:\cJ)$. 
\par The main result of this paper is the following (cf.~Remark~\ref{rem-exppolyfI} and Theorem~\ref{theo-constterm} for 
(i) - (iii) and Theorem \ref{theo-uniform} for (iv)).  

\begin{theoa}\label{theo-intro} Let $\cJ\subset \cZ(\gg)$ be an ideal of finite co-dimension. Let $I\subset S$. 
Then there exists 
an $N_\cJ\in\N$ such that for all $N\in \N$ there exists a continuous $G$-equivariant linear map

$$ \operatorname{CT}_{I,N}: \cA_{temp, N}(Z: \cJ) \to \cA_{temp, N+N_\cJ}(Z_I: \cJ), \ \ f\mapsto f_I$$
with the following properties for all $g\in G$ and $X_I\in \ga_I^{--}$:

\begin{enumerate}
\item[(i)] If we interpret 
$f$, resp.~$f_I$, as functions on $G$ which are right invariant under $H$, resp.~$H_I$, then 
$$\lim_{t\to +\infty} e^{-t\rho_Q(X_I)}\left(f(g\exp(tX_I))-f_I(g\exp(tX_I))\right)=0\, .$$
\item[(ii)] The assignment 
$$\R\ni t\mapsto e^{-t\rho_Q(X_I)}f_I(g\exp(tX_I))$$
defines an exponential polynomial with unitary characters, 
i.e., it is of the form $\sum_{j=1}^np_j(t)e^{i\nu_j t}$, where the $p_j$'s are polynomials and the $\nu_j$'s are real numbers. 
\end{enumerate} 
Conditions (i) and (ii) determine the constant term morphism $\operatorname{CT}_{I,N}$ uniquely. Moreover, 
\begin{enumerate} 
\item[(iii)] For  each $w_I\in\cW_I$ with  $w={\bf m}(w_I)\in \cW$, and any compact subsets $\cC_I\subset \ga_I^{--}$, 
$\Omega\subset G$, there exists $\varepsilon>0$ 
and a continuous semi-norm $p$ on $\cA_{temp,N}(Z)$ such that for all $f\in \cA_{temp, N}(Z: \cJ)$ one has: 
$$
\begin{array}{l}
\vert (a_Z\exp(tX_I))^{-\rho_Q}\left( f(g a_Z\exp(tX_I)w\cdot z_0)-f_I(g a_Z\exp(tX_I)w_I\cdot z_{0,I})\right)\vert\\
\leq e^{-\varepsilon t}(1+\Vert\log a_Z\Vert)^Np(f),\quad a_Z\in A_Z^-,X_I\in\cC_I, g\in\Omega, t\geq 0\,.
\end{array}
$$
\item[(iv)] The bound in (iii) is uniform for all $\cJ$ of codimension 1, i.e., $\cJ = \ker \chi$ for a character $\chi$ of $\cZ(\gg)$. 
\end{enumerate}
\end{theoa}

Given $f\in \cA_{temp}(Z)$ and $I\subset S$, we call $f_I\in \cA_{temp}(Z_I)$ the {\em constant term} of $f$ 
associated to $I$. Note that properties (i) and (ii) in Theorem \ref{theo-intro} determine $f_I$ uniquely as a smooth function on $Z_I$. 
Furthermore, we may interpret Theorem \ref{theo-intro}(iii) in such a way that $f_I$ controls the normal 
asymptotics of $f$ in direction of $\ga_I^{--}$ emanating from the base points $w\cdot z_0$ for certain $w\in\cW$.

\begin{rema} (a) 
Theorem \ref{theo-intro} can be phrased differently in the language of representation theory and it is worthwhile to mention this reformulation.  Let $V$ be a Harish-Chandra module and $V^\infty$ its unique $SF$-completion. 
The subgroups $H, H_I$ being real spherical implies that the invariant spaces $(V^{-\infty})^H$ and 
$(V^{-\infty})^{H_I}$ are both finite dimensional (cf.~\cite{ks}). Inside, we find the subspaces of tempered functionals
$(V^{-\infty})^H_{temp}$ and $(V^{-\infty})^{H_I}_{temp}$.  Then Theorem \ref{theo-intro} defines a linear map 
$$ (V^{-\infty})^H_{temp} \to (V^{-\infty})^{H_I}_{temp} , \ \ \eta\mapsto \eta_I$$
such that, for all $v\in V^\infty$, the matrix coefficient $f=m_{\eta,v}$ is approximated by $f_I=m_{\eta_I, v}$ 
in the sense of Theorem \ref{theo-intro}(iii).  As $A_I$ normalizes $H_I$, we obtain an action of $\ga_I$ 
on the finite dimensional space $(V^{-\infty})^{H_I}_{temp}$. It is easy to see that temperedness implies that 
$$\Spec_{\ga_I} (V^{-\infty})^{H_I}_{temp}\subset \rho_Q\big|_{\ga_I} + i\ga_I^*\,,$$
which in turn translates into the behavior of $f_I$ as exponential polynomial as recorded in Theorem \ref{theo-intro}(ii). 
\par (b) It is possible to develop a constant term theory for all $\cZ(\gg)$-finite eigenfunctions on $Z$, i.e. the assumption of temperedness
is not really necessary. This was carried out in \cite[Th. 7.1]{geoc} in case where $Z$ is wavefront.

\end{rema}

Parts (i) - (iii)  of Theorem \ref{theo-intro} generalize the work of Harish-Chandra in the group case (see \cite[Sections~21 to 25]{hc1}, also 
the work of Wallach \cite[Chapter~12]{wallach2}, where the constant term is called leading term) and the one of Carmona for symmetric spaces (see \cite{carmona}). The uniformity in (iv) generalizes  the uniform results  of Harish-Chandra in the group case 
(cf.~\cite[Section~10]{hc2}) and Carmona for symmetric spaces (cf.~\cite[Section~5]{carmona}). 

\par As a corollary of Theorem \ref{theo-intro}, we obtain a characterization of tempered eigenfunctions $f$
in the discrete series by the vanishing of their constant terms $f_I$, $I\subsetneq S$ (see Theorem~\ref{theo-ds} below). 
Again it is analogous to a result of Harish-Chandra. For this, we use in a crucial manner some results on discrete series from \cite[Section~8]{kks}.

\par The proof of Theorem \ref{theo-intro} is inspired by the work of Harish-Chandra for real reductive groups $G$, \cite{hc1,hc2},
who associates to a tempered eigenfunction $f$ on $G$ certain systems of linear differential equations. 
The main technical difficulty here is to set up the correct first order 
system \eqref{eq-phiPhi} of differential equations on $A_I$ associated to a function $f\in \cA_{temp}(Z)$. 
This is based on novel insights on the algebra of invariant differential operators on $Z$. 
With the solution formula for the first order differential system \eqref{eq-phiPhi}, one then obtains, as in \cite{hc1},
for each $f\in \cA_{temp}(Z)$, a unique smooth function $f_I\in C^\infty(Z_I)$ 
with properties (i), (ii) in Theorem \ref{theo-intro} and also (iii) for $w_I=w=1$. A main difficulty in this paper 
was to show that $f_I$ is in fact tempered, which translates into the assertions in (iii) for all $w_I\in \cW_I$. This, 
we deduce from Proposition \ref{prop-converg-asbs-1-us} on geometric asymptotics related to the natural matching map 
${\bf m}: \cW_I\to \cW$.  
Let us point out further that our treatment in Section \ref{section uniform} of the uniformity in Theorem \ref{theo-intro}(iv) constitutes a simplification to the so far existing state of the art in \cite[Chapter 12]{wallach2}.

\bigskip Earlier versions of this article needed the assumption that $Z$ is of wavefront type. This was mainly due to the lack of a better understanding of the algebra $\D(Z)$ of $G$-invariant differential operators on $Z$ and their behavior under boundary degenerations, i.e., overlooking that there is a natural map $\D(Z)\to \D(Z_I)$ originating from Knop's work \cite{knop-hc}. This was observed 
by Rapha\"el Beuzart-Plessis and is now recorded in Appendix \ref{app-RBP}. With this insight, we could remove the wavefront 
assumption and make the paper valid in the full generality of real spherical spaces.

\bigskip
{\it Acknowledgement}: We thank Rapha\"el Beuzart-Plessis and Friedrich Knop for their generous help on certain
technicalities of the paper. We appreciate comments of Yiannis Sakellaridis regarding the exposition of an earlier version of this article. 
Finally we thank a professional team of three referees for their careful and complete screening of 
our work.

\section{Notation}\label{sect-not}

In this paper, we will denote (real) Lie groups by upper case Latin letters and their Lie algebras by lower case German letters. 
If $R$ is a real Lie group, then $R_0$ will denote its identity component.
Furthermore, if $\sZ$ is an algebraic variety defined over $\R$ and $k$ is any field containing $\R$, then we denote 
by $\sZ(k)$ the $k$-points of $\sZ$. 

\par 
Let $\sG$ be a connected reductive algebraic group defined over $\R$ and let $G:=\sG(\R)$ be its group of real points. 

\begin{rem} More generally, we could define $G$ as an open subgroup of the real Lie group $\sG(\R)$.
The main analytic result of this paper (i.e., Theorem \ref{theo-intro}) is not affected by this more general assumption 
but  we do not supply a complete proof here. 
\end{rem}

For an $\R$-algebraic subgroup $\sR$ of $\sG$, we set $R:=\sR(\R)$ and note that $R\subset G$ is a closed subgroup.

\par Let now $\sH\subset \sG$ be an $\R$-algebraic subgroup. 
Having $\sG$ and $\sH$, we form the homogeneous 
variety $\sZ= \sG/\sH$. We note that $\sZ(\C)=\sG(\C)/\sH(\C)$ and denote  by $z_0=\sH(\C)$ the 
standard base point of  $\sZ(\C)$. Set $Z=G/H$ and record the $G$-equivariant embedding 
$$Z  \to \sZ(\C), \ \ gH\mapsto g\cdot z_0\, .$$ 
In the sequel, we consider $Z$ as a submanifold of $\sZ(\C)$ and, in particular, identify $z_0$ with the standard base point $H$ of $Z=G/H$ as well.
\begin{rem} 
Note that $Z$ is typically strictly smaller than $\sZ(\R)$, which is a finite union 
of $G$-orbits. An instructive example is the space of invertible symmetric matrices 
$\sZ= \sGL_n/ \sO_n$, which features 
$\sZ(\R)= \bigcup_{p+q=n}  \operatorname{GL}(n,\R)/\operatorname{O}(p,q)$. In particular, $\sZ(\R)\supsetneq Z=G/H
={\operatorname{GL}(n,\R)/\operatorname{O}(n)}$. 
\end{rem}

As a further piece of notation, we use, for an algebraic subgroup $\sR\subset \sG$ defined over $\R$, the 
notation $\sR_{\sH} := \sR \cap \sH$ and, likewise, $R_H:= R \cap H$.
In the sequel, we use the letter $\sP$ to denote a minimal $\R$-parabolic subgroup of $\sG$. The unipotent 
radical of $\sP$ is denoted by $\sN$.

\subsection{The local structure theorem}\label{lst-sect}
From now on, we assume that $Z$ is real spherical, that is, there is a choice of $\sP$ such that 
$P\cdot z_0$ is open in $Z$. 
\begin{rem} Notice that $\sP(\C) \sH(\C)$ is Zariski open and hence dense in $\sG(\C)$ as $\sG$ was assumed to 
be connected. Thus, any other choice $\sP'$ of $\sP$, with $P'\cdot z_0$ open, is  conjugate to $\sP$ under 
$\sH$. 
\end{rem}
We now recall the local structure theorem for real spherical varieties (cf.~\cite[Theorem~2.3]{kks0} or \cite[Corollary~4.12]{kk}; 
see also \cite{brion-luna,knop} for preceding versions in the complex case), which asserts 
that there is a unique parabolic subgroup $\sQ \supset \sP$ endowed with a Levi decomposition 
$\sQ=\sL \ltimes \sU$, defined over $\R$,  such that 

\begin{subequations}\label{lst}
\begin{align}\sQ\sH &= \sP\sH,\label{lst1} \\
\sQ_{\sH}&=\sL_{\sH}, \label{lst2}\\
\sL_{\sH} &\supset  \sL_{\rm n}, \label{lst3}
\end{align}
\end{subequations}
where $\sL_{\rm n}$ denotes the connected normal subgroup of $\sL$ generated by all unipotent elements defined over $\R$.

\begin{rem}\label{rem LST} (a) The Lie algebra $\gl_{\rm n}$ is the sum 
of  all  non-compact simple non-abelian ideals of $\gl$. \par 
(b) As mentioned  above, $\sQ$ is the unique parabolic subgroup containing $\sP$ 
with properties  (\ref{lst1}) - (\ref{lst3}).   Slightly differently, we could have defined $\sQ$ via 
\cite[Lemma 3.7]{kkss} which asserts 
$$ \sQ(\C)=\{g\in \sG(\C)\mid g \sP(\C)\cdot z_0 =\sP(\C)\cdot z_0\}\, .$$
The group $\sL_{\sH}$ is uniquely determined 
by $\sQ$ and we recall from \cite[Lemma 13.5]{kk} that  $\sL_{\sH}$ is an invariant 
of $\sZ$, i.e., its $\sH$-conjugacy class is defined over $\R$.  In contrast to 
$L_H$, the Levi subgroup $L$  is only unique up to conjugation with elements from $U$ 
which stabilize $L_H$. In this regard, we note that it is quite frequent that $L_H$ 
is trivial and then $L$ could be an arbitrary Levi of $Q$. 
For later purposes of compactifications, we will only use those choices of $L$ which are obtained 
from the constructive proof of the local structure theorem (cf.~\cite[Sect. ~2.1]{kks0} or \cite[Sect. 4]{kk}).  In case 
$\sZ$ is quasi-affine, this means that $\gl$ is defined as the centralizer of a generic hyperbolic element of $\gg^*$ which is contained in $(\gh +\gn)^\perp$ (see  the  constructive proof of the local structure theorem in \cite{kks0} or \cite{kk}).
\end{rem}

\begin{example}\label{triple 1} {\rm (The triple space)} Let $\sG_o:=\SL(2)$ with $G_o=\SL(2,\R)$ and form 
$\sG:= \sG_o\times \sG_o\times \sG_o$.  With $\sH:=\operatorname{diag} \sG_o$ we obtain a real spherical space
$Z=G/H$, the so-called triple space. It  features $Z=\sZ(\R)$ as one deduces from  the standard 
identification 
$$\sZ=\sG/\sH\to \sG_o\times \sG_o, \ \ (g_1, g_2, g_3)H\mapsto (g_1g_3^{-1}, g_2 g_3^{-1})$$
which a $\sG$-isomorphism of varieties where $\sG$ acts  on $\sG_o\times \sG_o$ as 
$$ (g_1, g_2, g_3)\cdot (x,y)= (g_1 xg_3^{-1}, g_2 yg_3^{-1}) \qquad (g_1, g_2, g_3, x, y\in \sG_o)\, .$$
Let $\gg_o=\gk_o +\gs_o$ be the standard Cartan decomposition with $\gk_o=\mathfrak{so}(2,\R)$. 
Let further $X_1, X_2\in \gs_o$ be linearly independent elements, set $X_3:= - (X_1 +X_2)$ and define 
$$ \ga_1:= \R X_1, \quad \ga_2 = \R X_2, \quad \ga_3=\R X_3\, .$$
Then $\ga:= \ga_1 \times \ga_2 \times \ga_3$ defines a Cartan subspace of $\gg=\gg_o\times \gg_o\times \gg_o$. 
We set $A_i=\exp(\ga_i)$ and then $A=A_1 \times A_2 \times A_3$. 
Observe that 
$$\gh^\perp:= \{ (Y_1, Y_2, Y_3)\in \gg \mid Y_1 + Y_2 +Y_3 =0\}$$ and thus 
$$X_0:= (X_1, X_2, X_3) \in \ga \cap \gh^\perp\, .$$
Let now $P_i\subset G_o$ be any minimal parabolic subgroup of $G_o$ above $A_i$. With 
$P=P_1 \times P_2 \times P_3$ we then obtain a minimal parabolic of $G$ such that 
$PH\subset G$ is open.  For later purpose we record $P_1 \cap P_2 \cap P_3=\{\pm 1\}$ equals the center $Z(H)$ 
of $H$. This 
entails
that $\sP \sH / \sH\simeq \sP/Z(\sH)$ and thus $\sQ=\sP$. 

If we write $P_i=M_iA_i N_i$, then $M_i=Z(G_o)$ with $N= N_1\times N_2 \times N_3$ and therefore $X_0\in (\gh +\gn)^\perp$ is in accordance with the constructive proof of the local structure theorem. 
\end{example}

A parabolic subgroup $Q$ as above in \eqref{lst} will be called $Z$-adapted to $P$. 

\par Let $\sA_{\sL}$ be the maximal split torus of the center of $\sL$ and 
$\sA$ be a maximal split torus of $\sP\cap \sL$. Note that $\sA_{\sL}\subset \sA$. 
Define $\sA_{\sZ}:= \sA/ \sA_{\sH}$ and let (by slight abuse of notation) $A_Z:=(A/A_H)_0\simeq A_0/(A_H)_0$. 
From the fact that $\sL_{\rm n}\subset \sL_{\sH}$ and $\sA= \sA_{\sL} (\sA\cap \sL_{\rm n})$, we obtain $A_Z\simeq (A_L)_0 /  (A_L)_0\cap H$ with 
$\ga_Z\simeq \ga_L/\ga_L\cap\gh$.  
\begin{res}\label{eq-tildea}
We choose a section $\bs:A_Z\to (A_L)_0$ of the projection $(A_L)_0 \to  A_Z$, which is a morphism of Lie groups. 
We will often use $\tilde{a}$ instead of $\bs(a)$, $\tilde{\ga}_Z$ instead of $\bs(\ga_Z)$ etc.
\end{res}
%$L=K_LAN_L$ is an Iwasawa decomposition. 
%\par   
Note that $Z_{\sG}(\sA)= \sM  \sA$ where $\sM \subset\sP$ is a maximal  anisotropic subgroup.
Moreover, 
$\sM\sA$, as a Levi 
of $\sP$, is connected (recall that Levi subgroups of connected algebraic groups are connected). 
Notice that $\sM$ commutes with $\sA$ and $P=MAN$. Observe that $M\cap A$ equals the $2$-torsion points $A_2$ of $A$.

\par From \eqref{lst1} - \eqref{lst2}, we obtain $\sP\sH/ \sH=\sQ \sH/ \sH\simeq \sU \times \sL/ \sL_{\sH}$, and, taking real points, 
we get 
$$[\sP\cdot z_0](\R)\simeq U \times (\sL/\sL_{\sH})(\R)\, .$$

Next, we collect some elementary facts about $(\sL/\sL_{\sH})(\R)$.
To begin with, we define 
$$\hat M_H:=\{m \in M \mid  m \cdot z_0 \subset \sA_{\sZ}(\R)\}$$
and note that $M_H$ is a cofinite normal $2$-subgroup of $\hat M_H$, see Proposition \ref{prop-B}. 
We denote by $F_M:=\hat M_H/ M_H$ this finite $2$-group. Since the action of  the $\sP$-Levi $\sM  \sA \subset 
\sL$ on 
$\sL/\sL_{\sH}$ is transitive, we obtain for the real points, by Proposition \ref{prop-B},
\begin{equation}\label{LST0} (\sL/ \sL_{\sH})(\R) =  [M/ M_H]  \times^{F_M}  \sA_{\sZ}(\R)\, .\end{equation}

From that, we derive the local structure theorem in the form 
\begin{equation}\label{LST1}  
[\sP\cdot z_0] (\R)=  U\times  \big[[M/M_H]  \times^{F_M}  \sA_{\sZ}(\R)\big]\,, \end{equation}
which we will use later.  Let us denote by $\sA_{\sZ}(\R)_2 \simeq\{-1,1\}^r$ the $2$-torsion elements in $\sA_{\sZ}(\R)
\simeq (\R^\times)^r $ and note that 
 $\sA_{\sZ}(\R)_2$ naturally parametrizes the connected components of $\sA_{\sZ}(\R)$, that is, the $A_Z$-orbits in 
 $\sA_{\sZ}(\R)$. In particular, we record the natural isomorphism of Lie groups
$$\sA_{\sZ}(\R)\simeq A_Z \times \sA_{\sZ}(\R)_2\, .$$

Observe that $F_M$ naturally acts on $\sA_{\sZ}(\R)_2$.  Hence, if we denote by 
$(P\backslash \sZ(\R))_{\rm open}$ the set of open $P$-orbits in $\sZ(\R)$, then we obtain from \eqref{LST1}
and Corollary \ref{cor-B} that: 

\begin{lem} \label{open orbit 1}
The map 
$$\sA_{\sZ}(\R)_2/ F_M \to (P\backslash \sZ(\R))_{\rm open}, \ \ F_M a_Z\mapsto   Pa_Z$$
is a bijection. 
\end{lem}

If we intersect \eqref{LST1}  with $Z$, we obtain 
\begin{equation}\label{LST1a}  
[\sP\cdot z_0] (\R)\cap Z =  U\times  \big[[M/ M_H]  \times^{F_M}  A_{Z,\R}\big]\ 
\end{equation}
with $A_{Z,\R}:= \sA_{\sZ}(\R)\cap Z$. Observe that $A_{Z,\R}$ might not be a group and is in general only a
$A_Z$-set. With $A_{Z,2}:=\sA_{\sZ}(\R)_2\cap A_{Z,\R}$, we then obtain 
$$A_{Z,\R}= A_Z A_{Z,2} \simeq A_Z\times A_{Z,2}\, .$$
Note that $F_M$ acts on $A_{Z,2}$ and thus we obtain, in analogy to Lemma \ref{open orbit 1}, that the map 
$$A_{Z,2}/ F_M \to (P\backslash Z)_{\rm open}, \ \  F_M a_Z\mapsto P a_Z$$
is a bijection. 
Next, we wish to find suitable representatives of the open $P$-orbits of $Z$ in $G$, i.e., find, for each  $F_M a_Z$ with $a_Z\in A_{Z,2}$, an element $w\in G$ such that 
$Pa_Z= Pw\cdot z_0$.
For that, we consider the exact sequence
$$ 1 \to \sA_{\sH} \to \sA\to \sA_{\sZ}= \sA/\sA_{\sH}\to 1\, .$$
Now, note that this sequence stays, in general, only left exact when taking real points
$$ 1 \to \sA_{\sH}(\R)  \to \sA(\R)\to \sA_{\sZ}(\R)\, . $$
In particular, we typically do not find a preimage of a torsion element $t\in A_{Z,2}\subset \sA_{\sZ}(\R)$ in
$A=\sA(\R)$.  However, if we set $T:=\exp(i\ga)\subset \sA(\C)$ and $T_Z:=\exp(i\ga_Z)\subset 
\sA_{\sZ}(\C)$, then $T\to T_Z$ is surjective. In particular, each $t\in A_{Z,2}$ has a lift 
$\tilde t\in T$, which can even be chosen in $\exp(i\tilde \ga_Z)\subset T$.   In this way we extend the 
lift $\bs: A_Z\to (A_L)_0$ from  \eqref{eq-tildea} to a map $\sA_Z(\R)\to A_L \exp(i\tilde \ga_Z)$, also denoted by 
$\bs$ in the sequel. Thus, we have shown that 

\begin{lem}\label{lem-decomp-cW} 
There exists a set $\cW\subset G$ of representatives 
of $(P\backslash Z)_{\rm open}$ such that 
any element 
$w\in\cW$ has a factorization in $\sG(\C)$  of the form
\begin{res-nn}%\label{eq-decomp-cW}
\begin{center}
$w=\tilde{t}h$, where $\tilde{t}\in\exp(i\tilde{\ga}_Z)$ and $h\in \sH(\C)$ such that $t:=\tilde{t}\cdot z_0\in  A_{Z,2}$.
\end{center}
\end{res-nn}
In particular, if $a\in A_H$, $aw\cdot z_0=w\cdot z_0$.
\end{lem}

In the sequel, $\cW\subset G$ is a choice of representatives of $(P\backslash Z)_{\rm open}$ as in Lemma \ref{lem-decomp-cW}, 
assumed to contain $1$ as a representative of $P\cdot z_0$.
\begin{example} (a) {\rm (Group case)} Let $G=G'\times G'$. In the group case, i.e. 
$$Z=G'\times G'/ \operatorname{diag} G'\simeq G'\, ,$$
one has 
only one open $P=P'\times P'$-orbit in $Z$ by the Bruhat decomposition of $G'$. Hence $\cW=\{1\}$ in this case. 
\par (b) {\rm (Triple case)}  We recall the triple space $Z = G_o\times G_o\times G_o/ \operatorname{diag}G_o$ from Example \ref{triple 1}.  We recall that $\sP\cdot z_0 \simeq \sP/ Z(\sH)$ with $Z(\sH)=\{\pm 1\}$ the center of $\sH$. 
Further we have $F_M=M/M_H$ in this case so that $\sA_Z(\R)_2/ F_M$ consists in fact of two elements. Hence 
$\cW=\{1, w\}$ has two elements in this case. If we were to consider disconnected $\operatorname{PGL}(2,\R)$ instead of $G_o=\SL(2,\R)$, then $\cW=\{1\}$ would be trivial. 
\end{example}

\subsection{Spherical roots and polar decomposition}\label{sect-1.2}

Let $K\subset G$ be a maximal compact subgroup associated to a Cartan involution 
$\theta$ of  $\gg$ with $\theta(X)=-X$ for all $X\in \ga$. Furthermore, let 
$\kappa$ be an $\Ad G$ and $\theta$-invariant bilinear form on $\gg$ such that the quadratic form 
$X\mapsto\Vert X\Vert^2=-\kappa(X,\theta X)$ is positive definite. We will denote by $(\,\cdot\,,\,\cdot\,)$ the corresponding scalar 
product on $\gg$. 
It defines a quotient scalar product and a quotient norm on $\ga_Z$ that we still denote by $\Vert\,\cdot\,\Vert$.\par
For later reference, we record  that $K$ is algebraic, i.e., $K=\sK(\R)$, and further, $M\subset K$ 
as we requested $\theta\big|_{\ga}=-\operatorname{id}_\ga$.

\par
Let $\Sigma$ be the set of roots of $\ga$ in $\gg$. 
If $\alpha\in\Sigma$, let $\gg^\alpha$ be the corresponding weight space for $\ga$. 
We write $\Sigma_\gu$ (resp.~$\Sigma_\gn$) $\subset\Sigma$ for the set of $\ga$-roots in $\gu$ (resp.~$\gn$) 
and set $\gu^-=\sum_{\alpha\in\Sigma_\gu}\gg^{-\alpha}$, i.e., the nilradical of the parabolic subalgebra $\gq^-$ opposite 
to $\gq$ with respect to $\ga$.

\par
Let $(\gl\cap \gh)^{\perp_\gl}$ be the orthogonal complement of $\gl\cap\gh$ in $\gl$ with respect to the scalar product $(\,\cdot\,,\,\cdot\,)$. One has:
\begin{equation}\label{eq-decomph}
\gg=\gh\oplus (\gl\cap \gh)^{\perp_\gl}\oplus \gu\,.
\end{equation}
Let $T$ be the restriction to $\gu^-$ of minus the projection from $\gg$ onto $(\gl\cap \gh)^{\perp_\gl}\oplus \gu$ parallel to $\gh$.
Let $\alpha\in\Si_\gu$ and $X_{-\alpha}\in\gg^{-\alpha}$. Then (cf.~\cite[equation~{(3.3)}]{kks})
\begin{equation}\label{eq-defXab}
T(X_{-\alpha})=\sum_{\beta\in\Si_\gu\cup\{0\}}X_{\alpha,\beta}\,,
\end{equation}
with $X_{\alpha,\beta}\in\gg^\beta\subset \gu$ if $\beta\in\Si_\gu$ and $X_{\alpha,0}\in(\gl\cap \gh)^{\perp_\gl}$.

\par
Let $\cM\subset\N_0[\Sigma_\gu]$ be the monoid generated by:
\begin{equation}\label{eq-cM}
\{\alpha+\beta\mid \alpha\in\Si_\gu,\beta\in\Si_\gu\cup\{0\}\textrm{ such that there exists } 
X_{-\alpha}\in\gg^{-\alpha}\textrm{ with }X_{\alpha,\beta}\neq 0\}.
\end{equation}
The elements of $\cM$ vanish on $\ga_H$ so $\cM$ is identified with a subset of $\ga_Z^*$. We define 
$$
\begin{array}{rcl}
& \ga_Z^{--}=\{X\in\ga_Z\mid \alpha(X)< 0,\,\alpha\in\cM\}\\
\textrm{and} & \ga_Z^{-}=\{X\in\ga_Z\mid \alpha(X)\leq 0,\,\alpha\in\cM\}.
\end{array}
$$

Following e.g., {\cite[Section~3.2]{kks}}, we recall that $\ga_Z^-$ is a co-simplicial cone, and our choice of spherical roots $S$ consists 
of the irreducible elements of $\cM$ which are extremal in $\R_{\geq 0}\cM$. Here, an element of $\cM$ is called 
irreducible if it cannot be expressed as a sum of two nonzero elements in $\cM$. 

\begin{example}{\rm  (Triple case continued)} In the triple case of Example \ref{triple 1} we had 
$\ga=\ga_1 \times \ga_2 \times \ga_3$ and accordingly $\Sigma_\ga= \Sigma_1\amalg\Sigma_2\amalg \Sigma_3$
with $\Sigma_i=\{ \pm \alpha_i\}$ and  $\alpha_i$ corresponding to $N_i\subset P_i \subset G_o$. 
Let now $0\neq Y\in \gg_o^{-\alpha_1}$ and expand it  as $Y=Y_0 + Y^+  +Y^-$ for $Y_0\in \ga_i$ and $Y^\pm \in \gg_o^{\pm \alpha_i}$ for $i=2,3$. Then a simple computation shows that $Y_0\neq 0$ and thus $\alpha_1\in \cM$. 
Likewise we obtain $\alpha_2, \alpha_3\in\cM$ and thus $S=\{\alpha_1, \alpha_2, \alpha_3\}$. 
\end{example}

Later, we will also need the edge of $\ga_Z$  
$$\ga_{Z,E}:=\ga_Z^- \cap (-\ga_Z^-)=\{X\in\ga_Z \mid \alpha(X)=0,\, \alpha\in S\}\, .$$
Note that $\ga_{Z,E}$ (more precisely ${\bf s} (\ga_{Z,E})$) normalizes $\gh$ and, likewise, $A_{Z,E}:=\exp(\ga_{Z,E})\subset A_Z$.

\par
We turn to the polar decomposition for $Z$.  Set $A_Z^-:=\exp(\ga_Z^-)$ and 
$A_{Z,\R}^-=A_{Z,2} A_Z^{-}\subset A_{Z,\R}$. By the definition of $\cW$, we then record that 
$$ A_{Z,\R}^-=  A_Z^-\cW\cdot z_0\, .$$ 

\begin{lem}[Polar decomposition]\label{lem-polar}
There exists a compact subset $\Omega \subset G$ such that 
\begin{equation}\label{polar} Z=\Omega A_{Z,\R}^-\,.\end{equation}
\end{lem} 

\begin{proof} Recall the group of $2$-torsion points $\sA_{\sZ}(\R)_2$ of $\sA_{\sZ}(\R)$. 
According to \cite[Theorem~13.2 with {Remark~13.3(ii)}]{kk} (building up on 
the earlier work \cite[Theorem~5.13]{kkss}), we have $\sZ(\R)= \Omega\cdot \sA_{\sZ}(\R)_2 A_Z^-$, for some compact subset $\Omega$ of $G$. 
Note that $A_Z^- \sA_{\sZ}(\R)_2 \cap Z= A_{Z,\R}^-$ and the assertion follows. 
\end{proof}

\begin{rem}[Passage to $H$ connected]\label{rem-cc} 
An analytically more general setup would be to work with connected $H$, i.e., with $Z_0=G/H_0$ instead of $Z=G/H$. For that, only some 
adjustments are needed. In detail, by right-enlarging $\cW$ with a set $F_H$  of representatives for $H/H_0 (H\cap M)$, we obtain with $\cW_0:=\cW F_H$ a set which is in bijection with the set of the open $P\times H_0$-double cosets in $G$. Similarly, one obtains a polar decomposition for $Z_0$ as  
$Z_0 = \Omega A_Z^- \cW_0\cdot z_0$ with $z_0=H_0$, now denoting the base point of $Z_0$. 

\par 
In order not to introduce further notation and maintain readability, the main text is kept in the algebraic framework. 
At various places, we will comment on the necessary adjustments needed for $H$ connected.
\end{rem}

The polar decomposition is closely related to compactification theory of $Z$, which we summarize in the next section. 

\begin{example} In some cases it is possible to have a very specific choice of $\Omega$, for example 
in the group case $Z= G'\times G'/ \operatorname{diag}G'\simeq G'$ one can take $\Omega = K'\times K'$. 
Interesting is also the triple case. Here one has in fact $Z=KA_Z$ for $K=K_o\times K_o\times K_o$ by 
\cite[Th. 3.2]{dks}, but 
$ Z\supsetneq K A_{Z,\R}^-$ as a consequence of \cite[Th. 4.1]{dks}.

\end{example}

%%%%%%%%%%%%%

\section{Boundary degenerations and quantitative geometry at infinity}\label{sect-2}

For a real spherical subalgebra $\gh\subset \gg$ and any subset $I\subset S$, there is natural deformation of $\gh_I$ of $\gh$, see \eqref{eq-Gr-limit} 
below for the straightforward definition. 
We define $H_I=\langle \exp(\gh_I)\rangle$ as the analytic subgroup of $G$ with Lie algebra $\gh_I$ and define 
the boundary degenerations of $Z$ as $Z_I: = G/ H_I$. Let us mention 
that $Z_I$ identifies (up to cover) with an open cone-subbundle in the normal bundle of a certain 
$G$-boundary orbit in a smooth compactification of $Z$. This more elaborate point of view
will be taken in the forthcoming work \cite{planch-sph}, but is not the topic of this paper.

The compactification theory is reviewed here shortly in Subsection \ref{subsection smooth compactification},
but only as a tool to give a short proof of  Proposition \ref{prop-converg-asbs-1-us}, which is 
the main result of this section. In more detail, let $\cW_I$ be the set of open $P$-orbits 
in the deformed space $Z_I$. We introduce a natural matching map ${\bf m}: \cW_I\to \cW$ for 
open $P$-orbits. The definition of ${\bf m}$ involves certain sequences  and the contents 
of  Proposition \ref{prop-converg-asbs-1-us} is about the rapid (i.e., exponentially fast) convergence of these
sequences.

\subsection{Boundary degenerations of $Z$}\label{sect-degenZ}

Let $I$ be a subset of $S$ and set:
$$\begin{array}{llllll}
\ga_I & = & \{X\in\ga_Z\mid \alpha(X)=0,\,\alpha\in I\},& A_I & = & \exp(\ga_I) \subset A_Z\,,\\
\ga_I^{--} & = & \{X\in\ga_I\mid \alpha(X)<0,\,\alpha\in S\setminus I\},& A_I^{--} & = & \exp (\ga_I^{--})\,.\\
\end{array}
$$
Then there exists an algebraic Lie subalgebra $\gh_I$ of $\gg$ such that, for all $X\in\ga_I^{--}$, one has: 
\begin{equation}\label{eq-Gr-limit} 
\gh_I=\lim_{t\to +\infty}e^{\mathrm{ad}\,tX}\gh
\end{equation} 
in the Grassmannian $\Gr_d(\gg)$ of $\gg$, where $d:=\dim(\gh)$ (cf.~{\cite[equation~(3.9)]{kks}}).

\par Notice that $\tilde\ga_I$ normalizes $\gh_I$, and hence
$$\hat \gh_I := \gh_I +\tilde\ga_I$$
defines a subalgebra of $\gg$ that does not depend on the section $\bs$.

\par
Let $H_I$ be the analytic subgroup of $G$ with Lie algebra $\gh_I$ and set $Z_I=G/H_I$. Then $Z_I$ is a real spherical space for which:
\begin{enumerate}
\item[(i)] $P H_I$ is open,
\item[(ii)] $Q$ is $Z_I$-adapted to $P$,
\item[(iii)] $\ga_{Z_I}=\ga_Z$ and $\ga_{Z_I}^-=\{X\in\ga_Z\mid \alpha(X)\leq 0,\,\alpha\in I\}$ contains $\ga_Z^-$
\end{enumerate}
(cf.~\cite[Proposition~3.2]{kks}). 
Similarly to \eqref{eq-decomph}, one has: 
$$\gg = \gh_I\oplus(\gl\cap\gh)^{\perp_\gl}\oplus\gu\,.$$
Let $T_I: \gu^- \to (\gl\cap\gh)^{\perp_\gl}\oplus\gu$ 
be the restriction to $\gu^-$ of minus the projection of $\gg$ onto $(\gl\cap\gh)^{\perp_\gl}\oplus\gu$ parallel to $\gh_I$.
Furthermore, let $\langle I\rangle\subset \N_0[S]$ be the monoid generated by $I$. 
Let $X_{\alpha,\beta}^I=X_{\alpha,\beta}$ if $\alpha +\beta\in \langle I\rangle$ and zero otherwise. It follows from \cite[equation~{(3.12)}]{kks} that $X_{-\alpha}+\sum_{\beta\in\Sigma_\gu\cup\{0\}}X_{\alpha,\beta}^I\in\gh_I$. This implies that, for $\alpha\in\Sigma_\gu$,
$$T_I(X_{-\alpha})=\sum_{\beta\in\Sigma_{\gu}\cup\{0\}}X_{\alpha,\beta}^I\,.$$
Let $A_{Z_I}^-=\exp\ga_{Z_I}^-$. Similarly to $Z$, the real spherical space $Z_I$ has a polar decomposition:
\begin{equation}\label{eq-polarZI}
Z_I=\Omega_I A_{Z_I}^-\cW_I\cdot z_{0,I}\,,
\end{equation}
where $z_{0,I}=H_I$, $\Omega_I\subset G$ compact and $\cW_I\subset G$ finite 
(cf.~Lemma \ref{lem-polar} and Remark \ref{rem-cc} for the choice of $\cW_I$ as $H_I$ is defined to be connected).
In more detail, the Lie algebra $\gh_I$ is algebraic and we let $\sH_I$ be the corresponding connected algebraic subgroup of $\sG$.  Using Lemma \ref{lem-decomp-cW} applied to the real spherical space 
$\sG(\R)/ \sH_I(\R)$,  we can make, using Remark \ref{rem-cc}, a choice for $\cW_I$  such that elements $w_I\in\cW_I$ are of the form 
\begin{equation}\label{eq-decompwI}
w_I=\tilde{t}_Ih_I,\quad\textrm{for some }\tilde{t}_I\in\exp(i\tilde{\ga}_Z)\textrm{ and } h_I\in \sH_I(\C)\, .
\end{equation}
We fix such a choice in the following, requesting in addition that $1\in\cW_I$.

\subsection{Quantitative escape to infinity}\label{subsection escape}

Let $I\subset S$. Let us pick $X_I\in \ga_I^{--}$, i.e., $X_I\in\ga_I$ and $\alpha(X_I)<0$ for all $\alpha\in S\setminus I$. 
For $s\in\R$, let
\begin{equation}\label{eq-as}
a_{s}:=\exp(sX_I)\,.
\end{equation}
Fix $w_I={\tilde t}_I h_I \in\cW_I$.
According to \cite[Lemma 3.9]{kks}, there exists $w\in\cW$ (uniquely determined 
by $X_I$) and $s_0>0$ with
\begin{equation}\label{eq-wIw}
P w_I\tilde{a}_{s}H=P wH,\quad s\geq s_0\,.
\end{equation}
Note that (cf.~Lemma~\ref{lem-decomp-cW}):
\begin{equation}\label{eq-decompw}
w=\tilde{t}h,\quad\textrm{for some }\tilde{t}\in\exp(i\tilde{\ga}_Z)\textrm{ and }h\in \sH(\C) \, .
\end{equation}
A priori, $w$ might depend on $X_I$, say $w(X_I)$. On the other hand, the limit \eqref{eq-Gr-limit} is locally uniform in compact subsets of $\ga_I^{--}$. In particular, the set of $Y\in\ga_I^{--}$ such that $w(Y)=w(X_I)$ is open and closed. Hence, $w$ is independent of $X_I$. 
Given $w_I\in \cW_I$ and $w\in \cW$ such that \eqref{eq-wIw} holds, we say that  $w$ {\em corresponds to} $w_I$
and note that this correspondence sets up a natural map ${\bf m}: \cW_I \to \cW$.

\par
According to \cite[Lemma 3.10]{kks}, there exist elements 
 $u_s\in U$, $b_s\in A_Z$ and $m_s\in M$ such that: 

\begin{equation}\label{eq-decompwIan}
\begin{array}{rcl}
w_I\tilde{a}_{s}\cdot z_0 & = & u_sm_s\tilde{b}_sw\cdot z_0 \quad s\geq s_0\,,\\
\displaystyle{\lim_{s\to +\infty}({a}_{s}{b}_s^{-1})} & = & 1,\\
\displaystyle{\lim_{s\to +\infty}u_s} & = & 1\,,\\
\displaystyle{\lim_{s\to +\infty}m_s} & = & m_{w_I},\textrm{ for some }m_{w_I}\in M\, .
\end{array}
\end{equation}
Notice that in case $w_I=1$ we have $w=1$ by \eqref{eq-wIw} and our request that $1\in \cW$; also 
note $m_{w_I}=1$ for $w_I=1$. 

\par
The goal of this section is to give a quantitative version of the convergence in \eqref{eq-decompwIan}. 
For that, we first refer to Appendix \ref{sect-appA} for the definition and basic properties of rapid convergence.

\par
Recall the finite $2$-group $F_M=\hat M_H/ M_H$ defined before \eqref{LST1} and fix with $\widetilde{F_M}\subset M$ 
a set of its representatives containing $1$. Then we have the following result: 

\begin{prop}\label{prop-converg-asbs-1-us}
The families $(a_sb_s^{-1})$ and $(u_s)$ converge rapidly to $1$ and 
one can choose the family $(m_s)$ such that $(m_s)$ converges rapidly to $m_{w_I}\in \widetilde{F_M}$.
\end{prop}

\begin{rem} \label{rem-I-renormal} (a) Proposition \ref{prop-converg-asbs-1-us} allows us to change the representatives 
$w_I$ to $m_{w_I}^{-1} w_I$  without losing the special form $w_I = \tilde{t}_I h_I$ with $\tilde t_I \in \exp\>  i \tilde{\ga}_Z$.  This
is because of $F_M A_Z\subset A_{Z,\R} \subset \exp(i\tilde{\ga}_Z)A\cdot z_0$ 
Hence, we may and will assume in the sequel that $m_{w_I}=1$ for all $w_I \in \cW_I$. 
\par 
(b) For $H$ replaced by connected $H_0$, Proposition \ref{prop-converg-asbs-1-us} stays valid with 
$\widetilde{F_M}$ right-enlarged by representatives of the component group $M_H/  (M\cap H_0)$. However, this causes
that we possibly cannot take $m_{w_I}=1$ as in (a). 
\par (c) In order to give a shorter proof of Proposition \ref{prop-converg-asbs-1-us}, we use the compactification theory of $\sZ(\R)$, which we review in the next paragraph. In particular, it yields the framework to consider 
$\hat z_{0,I}:=\lim_{s\to \infty} \tilde{a}_s\cdot z_0$ as 
an appropriate rapid limit in a suitable smooth compactification of $Z$. 
\par Geometrically, compactification theory provides (up to cover) a first order approximation 
of $Z_I$ to $Z$ at the vertex $\hat z_{0,I}$ at infinity.  This first order approximation then yields readily $u_s\to 1$ rapidly 
and $m_s\to m_{w_I}  \in \widetilde{F_M}$ rapidly.  However, first order approximation can only give
$a_sb_s^{-1}\to 1$ and to show that $a_s b_s^{-1}\to 1$ indeed rapidly, we need to use finer tools from finite dimensional 
representation theory. 
\end{rem}

\subsection{Smooth equivariant compactifications}\label{subsection smooth compactification}

By an equivariant compactification of $\sZ(\R)$, we understand here a $\sG$-variety $\hat \sZ$, defined over 
$\R$, such that $\hat \sZ(\R)$ is compact and contains $\sZ(\R)$ as an open dense subset.  In this context, we denote by $\partial Z$ the boundary of $Z$ in $\hat \sZ(\R)$. 

\par
Suitable (i.e., smooth and equivariant) compactifications exist:

\begin{prop}\label{compactification}  
Let $\sZ=\sG/\sH$ be an algebraic real spherical space. Then there 
exists a smooth equivariant compactification $\hat \sZ(\R)$ of $\sZ(\R)$ with the following property:  for all $I\subset S$ 
and $X \in \ga_I^{--}$, the limit $z_X:=\lim_{s\to \infty} (\exp(sX) \cdot z_0)$ exists in $\partial Z$ and the convergence 
is rapid.  If  $\gh_X$ is the stabilizer Lie subalgebra of $z_X$ in $\gg$, then $ \gh_I \subset \gh_X \subset \hat\gh_I$. 
\end{prop} 

The proof of this result is implicit in the proof of \cite[Theorem 13.2]{kk}.  Since the constructive proof is of relevance for us, we allow ourselves 
to repeat the fairly short proof.  
\begin{proof} The starting point is the local structure theorem for the open $\sP$-orbit on $\sZ$ as in \eqref{LST1} 
\begin{equation}\label{LST2} 
(\sP \cdot z_0)(\R) = U  \times  \left[[M/M_H]  \times^{F_M}  \sA_{\sZ}(\R)\right] 
\, .
\end{equation} 

One of the main results in \cite{kk}, see loc.cit., Theorem~7.1,  was that the compactification theory of $\sZ$ can be reduced, via the local structure theorem, to the partial toric compactification theory of $\sA_{\sZ}$.  Let us be more precise
and denote by $\Xi$ the character group of $\sA_{\sZ}$.  Note that $\Xi \simeq \Z^n$ with $n = \dim \sA_{\sZ}$. 
If we denote by ${\mathcal N}$ the co-character group of $\sA_{\sZ}$, then there is a natural identification 
of $\ga_Z$ with ${\mathcal N}\otimes_\Z \R$. Further, the compression cone $\ga_Z^-$ identifies as a co-simplicial cone (in \cite{kk}, one uses the rational valuation cone, denoted ${\mathcal Z}_k(X)$:  take $k=\R$ and $X=\sZ$. Then $\ga_Z^-= \R \otimes_\Q \cZ_k(X)$). The set of spherical roots $S\subset \Xi$ are then the primitive (in $\Xi$) extremal elements, co-spanning 
$\ga_Z^-$.  Best possible compactifications (a.k.a. wonderful compactifications) exist when 
$\#S = \dim \ga_Z$ and $S$ is a basis of the lattice $\Xi$. In general, this is not satisfied and 
we proceed as follows: 
we choose a complete fan $\cF\subset \ga_Z$, supported in $\ga_Z^-$, which is generated by 
simple simplicial cones $C_1, \ldots, C_N$, i.e., 
\begin{itemize} 
\item  $\bigcup C_i = \ga_Z^-$,
\item  $C_i \cap C_j$ is a face of both $C_i$ and $C_j$ for all $1\leq i,j\leq N$, 
\item $C_i=\{  X\in \ga_Z\mid  {d}\psi_{ij}(X)\leq 0 , 1\leq j \leq n\}$  for 
$(\psi_{ij})_{1\leq j\leq n}$ a basis of the lattice $\Xi$.   
\end{itemize}

For the existence of such a subdivision, we refer to \cite[Th. 11.1.9]{cls}. 
Now, attached to the fan $\cF$, we construct the toric variety $\sA_{\sZ}(\cF)$ expanding $\sA_{\sZ}$ along $\cF$. Note that the toric variety  $\sA_{\sZ}(\cF)$ is smooth, as the fan consists of simple cones (third bulleted property). Thus, we obtain a smooth variety 
\begin{equation}\label{LST3} 
\sZ_0(\cF):=\sU \times \big[[\sM/ \sM_{\sH}] \times^{F_{\sM}} \sA_{\sZ}(\cF)\big]\,, \end{equation}
which can be enlarged to a smooth $\sG$-variety  $\sZ(\cF):= \sG\cdot  \sZ_0(\cF)$, containing $\sZ_0(\cF)$ as an open subset. 
This is the content of \cite[Theorem~7.1]{kk}.  Now, set $\hat \sZ(\R):= \sZ(\cF)(\R)$ and note that $\hat \sZ(\R)$ is compact
by \cite[Corollary~7.12]{kk} as $\cF$ was assumed to be complete.  

\par We now claim that the limit 
$\lim_{s\to \infty}( \exp(s X) \cdot z_0)$ exists in $\sA_{\sZ}(\cF)(\R)$ and that the convergence is rapid. 
For that, we pick a cone $C_i$ which contains $X$, and let $\cF_i$ be the complete fan supported in $C_i$, which is generated by $C_i$. 
Notice that $\sA_{\sZ}(\cF_i)(\C)\simeq \C^n$  
is open in $\sA_{\sZ}(\cF)(\C)$. More specifically, the embedding of 
$\sA_{\sZ}(\C) \hookrightarrow \C^n$ 
is obtained by 
\begin{equation} \label{eq-cc}
\sA_{\sZ}(\C) \ni a \mapsto  (\psi_{ij}(a))_{1\leq j\leq n}\in(\C^*)^n\subset \C^n\, .\end{equation}
Given the definition of $C_i$ as the negative dual cone to the $\psi_{ij}$'s, $j=1,\dots,n$, the claim 
now follows. 

\par Note that the stabilizer of $z_s:= \exp(sX)\cdot z_0$ in $G$ is given by $H_s:=\exp(sX) H \exp(-sX)$ with Lie algebra $\gh_s:= 
e^{s \ad X} \gh$.  Since $z_s\to z_X$ in the smooth manifold $\hat \sZ(\R)$, we obtain that the vector fields corresponding to elements of $\lim_{s\to \infty}\gh_s= \gh_I$ vanish at $z_X$.  This shows that $\gh_I\subset \gh_X$. Finally the property 
$\gh_X \subset \hat\gh_I$ is derived from \cite[Theorem~7.3]{kk}. 
\end{proof} 
\medskip 
We end this subsection with further remarks and explanations of the construction in the proof above.

\par \begin{rem} \label{rem-compact}
(a) It is quite instructive to consider the special case of $\sZ=\sG=\sA$.  Here $A_Z^-=A_Z=A=A_{Z,E}$ 
with $S=\emptyset$. Upon identifying $\ga_Z$ with $\R^n$ via the character lattice $\Xi$, there are two standard choices for the complete 
fan $\cF$ generated by the cones  $C_1, \ldots, C_N$.  The first one is for $N=2^n$ and the cones
given by the orthants: $C_\sigma=\sigma  (\R_{\geq 0})^n $ for $\sigma\in \{ -1, 1\}^n$. This fan leads to 
$\sA_{\sZ}(\cF)(\R)\simeq {\mathbb P}^1(\R)^n $, the $n$-fold copy of the projective line.  The other standard choice is obtained via the identification 
$\R^n \simeq \R^{n+1}/ \R {\bf e}$ with ${\bf e} = e_1 + \ldots+e_{n+1}$, where $(e_1,\dots,e_{n+1})$ is the canonical basis of $\R^{n+1}$, and has $N=n+1$ cones given by: 
$$C_i =[(\bigoplus_{j=1\textrm{ s.t. } j\neq i}^{n+1}\R_{\geq 0} e_j)   + \R {\bf e}]/ \R{\bf e}\ ,\qquad 1\leq i \leq n+1\, .$$
This fan leads to the projective space $\sA_{\sZ}(\cF)(\R)\simeq {\mathbb P}^n(\R)$. 
\par (b) In the previous example, we have seen that there are exactly $N$ fixed points for $G$ in 
the compactification $\hat \sZ (\R)$, paramatrized by the cones $C_i$ and explicitly given by 
limits $\hat z_{\emptyset, i}:=\lim_{t\to \infty} (\exp(tX)\cdot z_0)$, for some $X\in \operatorname{int} C_i$.  This feature is not limited to 
this specific example but general: the compactification $\hat Z(\R)$ features exactly $N$ closed $\sG(\R)$-orbits through the various $\hat z_{\emptyset, i}$'s. 
This is in contrast to wonderful compactifications \cite[Def. 11.4]{kk}, where one has exactly one closed orbit
\cite[Th. 11.1]{kk}. 
 For wonderful
compactifications, one has $\ga_{Z,E}=\{0\}$ and $S$ is a basis 
of the lattice $\Xi$. If one of these two conditions fails, one is in need of a further subdivision of $\ga_Z^-$ into simple simplicial cones $C_i$. 
\par (c)  Let $X\in \ga_I^{--}$ and $F\in \cF$ be the smallest face in the fan 
which contains $X$. Then $\operatorname{span}_\R F \subset \ga_I$ and $\gh_X =\gh_I +\operatorname{span}_\R F $. 
In particular, for $X\in \ga_I^{--}$ generic, we have $\gh_X = \hat \gh_I$. 
 \par (d) (cf.~\cite[Section~11]{kk}) In case $\sH=N_{\sG}(\sH)$ is self-normalizing, one obtains a wonderful compactification 
$\hat \sZ(\R)$ by closing  up $\sZ(\R)$ in the Grassmannian $\operatorname{Gr}_d(\gg)$ of $d:=\dim\gh$-dimensional 
subspaces of $\gg$.  The embedding is given by ${g\cdot z_0}\mapsto \Ad(g) \gh$ and, given the definition
of $\gh_I$ as a limit (cf.~\eqref{eq-Gr-limit}),  one derives easily that the stabilizer $\hat H_I$ of 
$\hat z_{0,I}$ in $G$ has Lie algebra $\hat \gh_I$.  
\end{rem}

\subsection{Proof of Proposition~\ref{prop-converg-asbs-1-us}}

We choose a smooth compactification $\hat \sZ(\R) =\sZ(\cF)(\R)$ as constructed in the previous section.
To begin with, we note that the limit 
\begin{equation}\label{eq vertex} 
\hat z_{0,I}:= \lim_{s\to \infty }\tilde{a}_s \cdot z_0
\end{equation}
exists. Moreover, $\hat z_{0,I}\in \sA_{\sZ}(\cF)(\R)$ and the convergence is rapid.  Further, we deduce from the fact that $\hat z_{0,I}$ is fixed by $\sH_I(\C)$ and $w_I=\tilde{t}_Ih_I$ that 
$\lim_{s\to \infty} w_I \tilde{a}_s\cdot z_0= \tilde{t}_I\cdot \hat z_{0,I} \in \sA_{\sZ}(\cF)(\R)$ is rapid. 
On the other hand, $w_I \tilde{a}_s\cdot z_0  = u_s m_s \tilde{b}_s w\cdot z_0 = u_s m_s \tilde{b}_s \tilde t\cdot z_0$ which, in 
local coordinates as given by \eqref{LST3}, translates into: 
\begin{equation} \label{lstwwi} 
w_I \tilde{a}_s  \cdot z_0 =  (u_s,  [ m_s, {\tilde t}\tilde{b}_s\cdot z_0]) \in U \times \big[[M/M_H] 
\times^{F_M} \sA_{\sZ}(\cF)(\R)
\big]\, .
\end{equation}
Since $\lim_{s\to \infty} w_I \tilde{a}_s\cdot z_0 = (1, [1, \tilde t\cdot \hat z_{0,I}])$ is rapid, we thus deduce that 
$\lim_{s\to \infty} u_s =1$ is rapid as well. Next, we use the smooth projection  $[M/M_H]
\times^{F_M} \sA_{\sZ}(\cF)(\R)\to M/M_H F_M$ and obtain 
that 
$m_s (M_H F_M) \xrightarrow[s\to\infty]{rapid} 1 (M_H F_M)\in M/ M_H F_M$.  
In particular, we may assume that $m_s\xrightarrow[s\to\infty]{rapid} m_{w_I}\in M_H F_M$. 
Notice that we are free to replace 
$m_s$ by elements of the form $m_s m_H$ with $m_H \in M_H$ as we have 
$$m_s m_H \tilde b_s w\cdot z_0 = m_s m_H \tilde b_s \tilde t  \
\cdot z_0= m_s \tilde b_s \tilde t\cdot z_0= m_s \tilde b_s w\cdot z_0\, .$$
Thus, we may  even assume that  $m:=m_{w_I} \in 
\widetilde{F_M}$ (which was defined just before Proposition \ref{prop-converg-asbs-1-us}).

\par We remain with showing $b_s a_s^{-1}\xrightarrow[s\to\infty]{rapid} 1$. 
Using the techniques from above, it is immediate that $d(a_s, b_s) \to 0$ rapidly for 
any Riemannian metric $d$ on $\hat \sZ(\R)$. However, the statement $a_s^{-1}b_s \to 1$ rapidly 
is a considerably finer assertion and difficult to obtain 
working with only one compactification.  Thus, we change the strategy of proof 
and work with (varying) finite dimensional spherical representations instead. The representations give us various morphisms of $\sZ$ into affine spaces. 

\par  We assume first that $\sZ$ is quasi-affine. The representations we work with are finite dimensional irreducible representations $(\pi, V)$ 
of $\sG(\C)$ featuring two properties: 
\begin{itemize}
\item The representation is $\sH(\C)$-spherical, that is, there exists a 
vector $v_H\neq 0$ such that $\pi(h) v_H = v_H$ for all 
$h\in \sH(\C)$. 
\item Each $\sN(\C)$-fixed vector is fixed by $\sM(\C)$. 
\end{itemize} 

We remark that the second property is equivalent to the representation being
$\sK(\C)$-spherical (Cartan--Helgason theorem). In particular, each of these representations
is self-dual, its highest weight $\lambda$ is an element of $\ga^*$ and its lowest weight is given by $-\lambda$. We write 
$\Lambda_Z$ for the set of highest weights of all $\sH(\C)$ and $\sK(\C)$-spherical irreducible representations. 

\par Given $\lambda\in \Lambda_Z$, we let $(\pi, V)$ be such an irreducible representation of $\sG(\C)$ of highest \mbox{weight $\lambda$}. Furthermore, we fix a 
highest weight vector $v^*$ in the dual representation $V^*$ of $V$. From the fact 
that $PH$ is open in $G$, we then deduce $v^*(v_H)\neq 0$ and, in particular, 
$V^H=\C v_H$ is one-dimensional. Moreover, it follows that $\Lambda_Z\subset \ga_Z^*$. 

\par We expand $v_H$ into $\ga$-weight vectors 
$$v_H =\sum_{\mu\in \Lambda_\pi}  v_{-\lambda +\mu}\,,$$
with $\Lambda_\pi:=\{\mu\in\ga^*\mid v_{-\lambda+\mu}\neq 0\}$.  
As $v_H$ is $\ga_H$-fixed, we have $\Lambda_\pi\subset \ga_Z^*$ and, by 
\cite[Lemma 5.3]{kkss}, we obtain:
\begin{equation} \label{eq-contr1}\mu\big|_{\ga_Z^{--}} <0\, , \qquad \mu \in \Lambda_\pi{\setminus\{0\}}\, .
\end{equation}

Set 
$$v_{H,s}:= a_s^\lambda \pi(\tilde a_s)v_H \qquad s\geq 0 $$
and note, as $v_H$ is $H$-invariant, that this expression is independent of the choice of the particular section $\bs$. 
From the definition, we then get
\begin{equation}  \label{eq-contr2} 
v_{H,s}= \sum_{\mu\in \Lambda_\pi} a_s^\mu  v_{-\lambda +\mu}\, .
\end{equation}

If we define 
$$v_{H,I}:= \sum_{\mu\in \Lambda_\pi\textrm{ s.t. } \mu(X_I)=0} v_{-\lambda +\mu}\, ,$$ 
then it is immediate from (\ref{eq-contr1}) and (\ref{eq-contr2}) that 
\begin{equation}\label{eq-ra1} 
v_{H,s} \to v_{H,I} \qquad \text{rapidly for $s\to \infty$}\, . 
\end{equation}  

Recall $v^*\in V^*$, a highest weight vector in the dual representation. 
Then we obtain from $w_I \tilde a_s \cdot z_0 = u_s m_s \tilde b_s \tilde t \cdot z_0$ 
with $t=\tilde t \cdot z_0\in A_{Z,2}$ (cf. Lemma \ref{lem-decomp-cW}) that: 
\begin{equation}\label{eq-ra2}  
v^*(\pi(w_I) v_{H,s})= a_s^{\lambda} \left(v^* 
(\pi(u_s m_s \tilde b_s\tilde t ) v_H)\right)=  (a_s b_s^{-1})^{\lambda} t^{-\lambda} \left(v^*(v_H)\right)\,.
\end{equation} 

By \eqref{eq-ra1}, we thus obtain from \eqref{eq-ra2} that: 
\begin{equation}\label{eq-ra3}
( a_s b_s^{-1})^{\lambda} = t^\lambda \frac{v^*(\pi(w_I) v_{H,s})}{v^*(v_H)}
\to   t^\lambda \frac{v^*(\pi(w_I) v_{H,I})}{v^*(v_H)}  \qquad \textrm{rapidly for $s\to \infty$} \, .
\end{equation}

We now employ \cite[Lemma~3.10]{kks} for the simple convergence 
$a_s b_s^{-1}\to 1$. Thus, (\ref{eq-ra3}) implies 
$t^\lambda \frac{v^*(\pi(w_I) v_{H,I})}{v^*(v_H)}=1$
with 
\begin{equation} \label{eq-ra4} 
(a_s b_s^{-1})^{\lambda} \to 1  \qquad \text{rapidly for }s\to \infty,\, 
{\lambda \in \Lambda_Z}\, .
\end{equation} 
Assume for the moment that $\sZ$ is quasi-affine. We claim that the set $\Lambda_Z$ spans $\ga_Z^*$. 
In fact, this is  as a consequence of \cite[Lemma~3.4 and (3.2)]{kks0} as in the notation of op. cit. 
each $f\in \C(Z)_\chi$ is a quotient $f=f_1/ f_2$ for some $f_i \in \C[Z]_{\chi_i}$ corresponding to characters
$\Lambda_Z$ in the notation of this article. From the claim and  \eqref{eq-ra4} we then get $a_s b_s^{-1}\to 1$ rapidly.

\par If $\sZ$ is not quasi-affine, then matters are reduced to the quasi-affine case via the so-called cone construction
from algebraic geometry: we extend $\sG(\C)$ to $\sG'(\C):= \sG(\C) \times \C^*$ 
and, for a character $\psi: \sH(\C) \to \C^*$ defined over $\R$, we set 
$\sH'(\C):=\{ (h, \psi(h))\mid h\in \sH(\C)\}\, .$

\par In this way, we obtain a real spherical space 
$\sZ':= \sG'/\sH'$ which projects $\sG'$-equivariantly onto $\sZ$.  According to 
\cite[Corollary~6.10]{kk}, there is compatibility of compression cones: 
\begin{equation} \label{eq-compcc}
\ga_Z'^- =\ga_Z^- \oplus \R\, .
\end{equation} 
Furthermore, according to Chevalley's quasiprojective embedding theorem for homogeneous 
spaces \cite[Sect. 11.2]{humphreys},  we find such a $\psi$ such that $\sZ'$ is quasi-affine and we complete 
the reduction to the quasi-affine case as follows: we lift the identity \eqref{eq-decompw} to $Z'$ and obtain 
$$w_I\tilde{a'_s}\cdot z_0'  =  u_sm_s\tilde{b'_s}w\cdot z_0' \quad s\geq s_0\,,$$
with $\tilde{a'_s} \in \tilde{a}_s ( 1 \times \R^\times) \in G'$ and likewise for 
$\tilde{b'_s} \in \tilde{b}_s ( 1 \times \R^\times) \in G'$.  Because of \eqref{eq-compcc}, we obtain the
rapid convergence $b_s' (a_s')^{-1}\to 1$ in the quasi-affine environment of $Z'$.  Projecting to $Z$ then completes this final reduction step.

\section{$Z$-tempered $H$-fixed continuous linear forms and the space $\cA_{temp}(Z)$}\label{sect-Z-temp}

In this section, we introduce the function space $\cA_{temp}(Z)$ of tempered 
$\cZ(\gg)$-eigenfunctions on $Z$. Via Frobenius reciprocity, these functions can naturally be interpreted
as matrix coefficients of smooth representations of $G$ which are of moderate growth ($SF$-representations for short). 
This section starts with a brief digression on $SF$-representations and then provides   the definition of $\cA_{temp}(Z)$.

\subsection{$SF$-representations of $G$}

Let us recall some definitions and results of \cite{bk}.

\par
A continuous representation $(\pi,E)$ of a Lie group $G$ on a locally convex complex topological vector space $E$ is a representation such that the map:
$$G\times E\to E,\, (g,v)\mapsto\pi(g)v,\textrm{ is continuous.}$$
If $R$ is a compact subgroup of $G$ and $v\in E$, we say that $v$ is $R$-finite if $\pi(R)v$ generates a finite dimensional 
subspace of $E$. 
Let $V_{(R)}$ denote the vector space of $R$-finite vectors in $E$. 
Let $\eta$ be a continuous linear form on $E$ and $v\in E$. Let us define the generalized matrix coefficient 
associated to $\eta$ and $v$ by:
\begin{equation} \label{def mc} m_{\eta,v}(g):=\la\eta,\pi(g^{-1})v\ra,\quad g\in G\,.
\end{equation}
Let $G$ be a real reductive group and $\Vert\cdot\Vert$ be a norm on $G$ (cf.~\cite[Section~2.A.2]{wallach1} or \cite[Section~2.1.2]{bk}). 
A continuous representation $(\pi,E)$ of $G$ is called a {\it Fr\'echet representation with moderate growth} if $E$ is a Fr\'echet space and if,  
for any continuous semi-norm $p$ on $E$, there exist a continuous semi-norm $q$ on $E$ and $N\in\N$ such that: 
\begin{equation}\label{eq-f-rep}
p(\pi(g)v)\leq q(v)\Vert g\Vert^N ,\quad v\in E,g\in G.
\end{equation}
This notion coincides with the notion of $F$-representations given in \cite[Definition~2.6]{bk} for the large scale 
structure corresponding to the norm $\Vert\cdot\Vert$. We will adopt the terminology of $F$-representations.

\par 
Let $(\pi,E)$ be an $F$-representation. A smooth vector in $E$ is a vector $v$ such that $g\mapsto \pi(g)v$ is smooth from $G$ to $E$. 
The space $E^\infty$ of smooth vectors in $E$ is endowed with the Sobolev semi-norms that we define now.
Fix a basis $X_1,\dots,X_n$ of $\gg$ and $k\in\N$. Let $p$ be a continuous semi-norm on $E$ and set 
\begin{equation}\label{eq-pk-semi-norm}
p_k(v)=\left(\sum_{m_1+\cdots+m_n\leq k}p(\pi(X_1^{m_1}\cdots X_n^{m_n})v)^2\right)^{1/2},\quad v\in E^\infty\,.
\end{equation}
We endow $E^\infty$ with the topology defined by the semi-norms $p_k$, $k\in\N$, when $p$ varies in the set of 
continuous semi-norms of $E$, and denote by $(\pi^\infty,E^\infty)$ the corresponding sub-representation of $(\pi,E)$.

\par
An $SF$-representation is an $F$-representation $(\pi,E)$ which is smooth, i.e., such that $E=E^\infty$ as topological vector spaces. 
Let us remark that if $(\pi,E)$ is an $F$-representation, then $(\pi^\infty,E^\infty)$ is an $SF$-representation (cf.~\cite[Corollary~2.16]{bk}).

\par 
Recall our fixed maximal compact subgroup $K\subset G$.

\par Following \cite{bk}, we call an $SF$-representation $E$ {\em admissible} provided that 
$E_{(K)}$ is a Harish-Chandra module with respect to the pair $(\gg, K)$, that is, a $(\gg, K)$-module with finite $K$-multiplicities which is finitely generated as a $\cU(\gg)$-module. 
  
\par
An admissible $SF$-representation will be called an $SAF$-representation of $G$. 

\par It is a fundamental theorem of Casselman--Wallach (cf.~\cite{cas}, \cite[Chapter~11]{wallach2} or \cite{bk}) 
that every Harish-Chandra module $V$ admits a unique $SF$-completion $V^\infty$, i.e., an $SF$-representation $V^\infty$ of $G$, unique up to isomorphism in the $SF$-category, such that: 
$$V^\infty_{(K)} \simeq_{(\gg, K)}  V\, .$$
In particular, all $SAF$-representations of $G$ are of the form $V^\infty$ for a Harish-Chandra \mbox{module $V$.} 

\subsection{The spaces $C_{temp,N}^\infty(Z)$ and $\cA_{temp,N}(Z)$}\label{sect-Atemp}

From now on and for the remainder of this paper, we will assume that $Z$ is unimodular. Let $\rho_{Q}$ be the half sum of the roots of $\ga $ in $\gu$. 
Let us show that
$$%\begin{equation}\label{eq-rhoQtrivial}
\rho_{Q}\textrm{ is trivial on }\ga_H\,.
$$%\end{equation}
As $\gl\cap\gh$-modules,
$$\gg/\gh=\gu\oplus(\gl/\gl\cap\gh)\,.$$
But the action of $\ga_H=\ga\cap\gh$ on $(\gl/\gl\cap\gh)$ is trivial. Since $Z$ is unimodular, the action of $\ga_H$ has 
to be unimodular. Our claim follows.

\par
Hence $\rho_Q$ can be defined as a linear form on $\ga_Z$.\\

We have the notion of weight functions on a homogeneous space $X$ of a locally compact group $G$ (cf.~\cite[Section~3.1]{bernstein-planch}). 
This is a function $w:X\to\R_{>0}$ such that, for every ball $B$ of $G$ (i.e., a compact symmetric neighborhood of $1$ in $G$), 
there exists a constant $c=c(w,B)$ such that
\begin{equation}\label{eq-weight}
w(g\cdot x)\leq cw(x),\quad g\in B,x\in X\,.
\end{equation}
One sees easily that, if $w$ is a weight function, then $1/w$ is also a weight function.

\par
Let $\Omega\subset G$ be a compact subset in accordance with the polar decomposition
in Lemma \ref {lem-polar}. Then weight function $\bv$ and $\bw$ on $Z$ are defined by
$$
%\begin{array}{rcl}
\bv(z):=\mathrm{vol}_Z(Bz)\quad \textrm{and}\quad
\bw(z):=1+\sup_{a\in A_Z^-\textrm{ s.t.}\,z\in\Omega a\cW\cdot z_0}\|\log(a)\|\,,
%\end{array}
$$
where $B$ is some ball of $G$ and $\Vert\cdot\Vert$ refers to the quotient norm on $\ga_Z=\ga/\ga_H$. 
It is then clear that $\bv$ is a weight function and $\bw$ is a weight function 
by \cite[Proposition~3.4]{kkss-volgrowth}. Recall that the equivalence class of $\bv$ does not depend on $B$ (see op. cit. Lemma~4.1 and beginning of Section 3 for the definition of the equivalence relation).

\par 
For any $N\in\N$, we define a norm  $p_N$ on 
$C_c(Z)$ by 
\begin{equation}\label{eq-pn-norm}
p_N(f):=\sup_{z\in Z}\left(\bw(z)^{-N}\bv(z)^{1/2}\vert f(z)\vert\right)\, .
\end{equation}
From the polar decomposition of $Z$ (cf.~\eqref{polar}), one has
$$
p_N(f)=\sup_{g\in\Omega,a\in A_Z^-,w\in\cW}\left(\bw(g aw\cdot z_0)^{-N}
\bv(g aw\cdot z_0)^{1/2}\vert f(g aw\cdot z_0)\vert\right)\,.
$$
From the fact that $\bv$ and $\bw$ are weight functions on $Z$ and from \cite[Propositions~3.4(2) and~4.3]{kkss-volgrowth}, 
one then sees that:
\begin{res}\label{eq-equiv-norms}
\begin{quote}
The norm $p_N$ is equivalent to the norm
$$f\mapsto q_N(f):=\sup_{g\in\Omega,a\in A_Z^-,w\in\cW}\left(a^{-\rho_Q}(1+\Vert\log a\Vert)^{-N}\vert f(g aw\cdot z_0)\vert\right)\,.$$
\end{quote}
\end{res}
Moreover, due to the fact that $\bv$ and $1/\bw$ are weight functions on $Z$, one gets that $G$ acts by left translations on $(C_c(Z), p_N)$ and, 
for any compact subset $C$ of $G$, by changing $z$ into $z'={g^{-1}}\cdot z$ in \eqref{eq-pn-norm}, one sees that:
\begin{res}\label{eq-EnGstable}
\begin{quote}
There exists $c>0$ such that 
$$p_N(L_gf)\leq cp_N(f),\quad g\in C,f\in C_c(Z)\,.$$
\end{quote}
\end{res}
This is in essence what is needed to identify  
\begin{equation}\label{eq-defCtempN}
C_{temp,N}^\infty(Z):=\{f\in C^\infty(Z)\mid p_{N,k}(f)<\infty,\, k\in \N \}
\end{equation}
as an $SF$-module for $G$.  Here, the  $p_{N,k}$, $k\in\N$, are as in \eqref{eq-pk-semi-norm}, with $p$ replaced by $p_N$ and $(\pi,E)$ by the $SF$-representation $(L,C_{temp,N}^\infty(Z))$.  
Further, we endow the increasing union 
$C_{temp}^\infty(Z):={\bigcup_{N\in\N}} C_{temp,N}^\infty(Z)$, with the inductive limit topology. 
We call $C_{temp}^\infty(Z)$ the space of {\em smooth tempered 
functions} on $Z$. 

\par Inside of $C_{temp}^\infty(Z)$, we define $\cA_{temp}(Z)$ as the subspace of $\cZ(\gg)$-finite functions. Likewise we define $\cA_{temp,N}(Z)$.

\subsection{$Z$-tempered functionals}

Let $(\pi, E)$ be an $SF$-representation and $E'$ its strong dual.
An element $\eta\in (E')^H$ will be called {\it $Z$-tempered} provided 

\begin{res}\label{eq-def-tempform}
\begin{quote}
There exists $N\in\N$ such that, for all $v\in E$, one has 
$m_{\eta,v}\in C_{temp,N}^\infty(Z)$.
\end{quote}
\end{res}

The $Z$-tempered functionals then define a subspace $(E')^H_{temp}$ of $(E')^H$. 
Frobenius reciprocity then asserts for an $SF$-representation $(\pi, E)$ the following isomorphism
of vector spaces: 
\begin{equation} 
\Hom (E, C_{temp}^\infty(Z)) \simeq (E')^H_{temp}\,,
\end{equation}
which can be established as in \cite[Lemma~6.5]{ks-sanya} via the Grothendieck factorization theorem
for topological vector spaces. 

\par In case $E=V^\infty$ is an $SAF$-representation, we adopt the more common 
terminology $V^{-\infty}:=(V^\infty)'$ and recall the finiteness result 
for real spherical spaces (cf.~\cite[Theorem~3.2]{ks}):
\begin{equation}\label{eq-Hfin} 
\dim (V^{-\infty})^H <\infty\, .
\end{equation} 

For a finite codimensional ideal $\cJ$ of $\cZ(\gg)$, let
\begin{equation}\label{eq-defAtemp}
\cA_{temp, N}(Z: \cJ):=\{ f \in  \cA_{temp, N}(Z)\mid f \ \text{is annihilated by $\cJ$}\}
\end{equation}
and denote by $\cA_{temp}(Z:\cJ)$ the subspace of $\cA_{temp}(Z)$ annihilated by $\cJ$.    

\begin{prop}\label{prop-SAF-rep} 
There exists an $N_0\in\N$ such that $\cA_{temp}(Z:\cJ)=\cA_{temp,N_0}(Z:\cJ)$. In particular, $\cA_{temp}(Z:\cJ)$ is an $SAF$-representation of $G$.
\end{prop}

The proof of Proposition \ref{prop-SAF-rep} is preceded by two lemmas.

\begin{lem}\label{lem-AtempSAF1}
There exists a Harish-Chandra module $V_{\cJ}$ annihilated by $\cJ$ such that any Harish-Chandra module annihilated by $\cJ$ is a quotient of a finite direct sum of copies of $V_{\cJ}$.
\end{lem}

\begin{proof} According to Harish-Chandra \cite[Thm. 7]{hc56}, there exist only finitely many isomorphism classes $V_1,\ldots,V_n$ of irreducible Harish-Chandra modules that are annihilated by $\cJ$. 
We can find a finite set $F\subset \widehat{K}$ of isomorphism classes of irreducible $K$-representations such that, for each $1\leqslant i\leqslant n$, the $(\gg,K)$-module $V_i$ is generated by its $\delta$-isotypic component for some $\delta\in F$. 
Then every Harish-Chandra module which is annihilated by $\cJ$ is generated by the sum of its $\delta$-isotypic components for every $\delta\in F$. Let $R(K)$ be the algebra (for convolution) of $K$-finite functions on $K$ and $I_F\subset R(K)$ the ideal of elements which acts by zero in $\delta$ for any $\delta\in F$. Let $R(\gg,K)$ be the ``Hecke algebra" of Knapp--Vogan \cite[Section I.4]{KV}, i.e., the algebra of $K$-finite distributions on $G$ which are supported in $K$. Then $R(\gg,K)$ is generated as a $\cU(\gg)$-module (either on the left or on the right) by $R(K)$ and moreover the category of $(\gg,K)$-module is naturally equivalent to the category of {\em non-degenerate} (also called {\em approximately unital} by Knapp--Vogan) $R(\gg,K)$-modules. Setting $V_{\cJ}=R(\gg,K)/(R(\gg,K)I_F+R(\gg,K)\cJ)$ we see that $V_{\cJ}$ is  a $(\gg,K)$-module which is generated by any supplement subspace of $I_F$ in $R(K)$ and annihilated by $\cJ$. Hence, by another result of Harish-Chandra, $V_{\cJ}$ is in fact a Harish-Chandra module, see 
\cite[Th. 4.3]{bk} for a short proof. Moreover, it is clear that any Harish-Chandra module annihilated by $\cJ$ is a quotient of a finite sum of copies of $V_{\cJ}$. 
\end{proof}

\begin{lem}\label{lem-AtempSAF2}  
Let $f\in \cA_{temp,N}(Z)$ be a $K$-finite element. Set 
$E^f:=\overline{\operatorname{span}_\C{L(G)f}}\,,$
with the closure taken in $C_{temp, N}^\infty(Z)$. 
Then $E^f$ is an $SAF$-representation, i.e., $E^f_{(K)}$ is a Harish-Chandra module. 
\end{lem}

\begin{proof} We consider the $(\gg, K)$-module $V^f:= \cU(\gg)f$. 
Since $f$  is $\cZ(\gg)$-finite, the same holds for $V^f$.  Now, 
as a finitely generated and $\cZ(\gg)$-finite module, $V^f$ is a Harish-Chandra module by a theorem of Harish-Chandra, see \cite[Th. 4.3]{bk} for a short proof.   Note that $E^f$ is an $F$-representation of $G$ containing 
the Harish-Chandra module $V^f$. Hence the closure $\overline{V^f}$ in $E^f$ is a 
continuous $G$-representation. On the other hand $E^f$ was generated by the $G$-translates 
of $f$. Hence $\overline{V^f}=E^f$ and thus 
 $V^f \simeq_{(\gg, K)} E^f_{(K)}$. 
\end{proof} 

\begin{proof}[Proof of Proposition \ref{prop-SAF-rep}] 
Let $E_{\cJ}=V_{\cJ}^\infty$ be the $SAF$-globalization of $V_{\cJ}$ where $V_{\cJ}$ is as in Lemma \ref{lem-AtempSAF1}. We will actually show that $\cA_{temp}(Z:\cJ)$ is precisely the image of
\begin{equation}\label{eq 1}
\displaystyle (E'_{\cJ})^H_{temp}\otimes E_{\cJ}\to \cA_{temp}(Z)
\end{equation}
$$\displaystyle \eta\otimes v\mapsto m_{\eta,v}\,.$$
Indeed, since $(E'_{\cJ})^H_{temp}$ is of finite dimension (cf.~\eqref{eq-Hfin}), the image of (\ref{eq 1}) is contained in $\cA_{temp,N_0}(Z: \cJ)$ for some $N_0\geqslant 0$ and, by unicity of the $SAF$-globalization, this image is closed in $\cA_{temp,N}(Z:\cJ)$ for every $N\geqslant N_0$. Hence, it suffices to show that it is also dense in $\cA_{temp,N}(Z:\cJ)$ for every $N\geqslant N_0$.  Let $f$ be any $K$-finite function $f\in \cA_{temp,N}(Z:\cJ)$. 
By Lemma \ref{lem-AtempSAF2} $E^f$ is an $SAF$-representation annihilated 
by $\cJ$ and as such a quotient of finitely many copies of $E_\cJ$ by Lemma \ref{lem-AtempSAF1}, i.e. 
there exists a surjective morphism $\bigoplus_{j=1}^n E_\cJ \to E^f$.  Hence 
$((E^f)')_{temp}^H\to \bigoplus_{j=1}^n (E_\cJ')_{temp}^H$ injects and thus every $K$-finite function $f\in \cA_{temp,N}(Z:\cJ)$ is in the image of (\ref{eq 1}). Since $K$-finite functions are dense in $\cA_{temp,N}(Z:\cJ)$, this completes the proof.
\end{proof}   

We conclude this section with an illustration of invariant functionals for our guiding example. 
\begin{example} {\rm (Triple space continued)} In this case $G=G_o\times G_o \times G_o$ and 
Harish-Chandra modules for $(\gg, K)$ are of the form $V=V_1 \otimes V_2 \otimes V_3$ with 
$V_i$ Harish-Chandra modules for $(\gg_o, K_o)$.  Likewise one has $V^\infty = V_1^\infty\hat \otimes
V_2^\infty\hat\otimes V_3^\infty$. We denote by $\widetilde V_3$ the Harish-Chandra module dual 
to $V_3$. Note that 
$$((V^\infty)')^H= \Hom_{G_o}(V_1^\infty\hat \otimes
V_2^\infty\hat\otimes V_3^\infty, \C)$$
and thus 
$$((V^\infty)')^H\simeq \Hom_{G_o}(V_1^\infty\hat \otimes
V_2^\infty, \widetilde V_3^\infty)$$
by a standard argument: First $\Hom_{G_o}(V_1^\infty\hat \otimes
V_2^\infty\hat\otimes V_3^\infty, \C)\simeq \Hom_{G_o}(V_1^\infty\hat \otimes
V_2^\infty, (V_3^\infty)')$ by Grothendieck's theory of tensor products for nuclear Fr\'echet spaces, and then, by the use of the 
Grothendieck factorization theorem, deduce that the image of some $T\in \Hom_{G_o}(V_1^\infty\hat \otimes
V_2^\infty, (V_3^\infty)')$ lies in fact in some Banach completion of $\widetilde V_3$.  
Thus we see that $H$-invariant functionals are related to branching problems of tensor product 
representation $V_1^\infty \otimes V_2^\infty$ for $G_o$. There is a vast literature on this subject, see for instance \cite{repka} or \cite{br}. 
\end{example}

\section{Ordinary differential equation for $\cZ(\gg)$-eigenfunctions on $Z$}\label{sect-4}

Let $f\in \cA_{temp}(Z)$. The goal of this section is to show that $f\big|_{A_I}$ gives a certain 
system of ordinary differential equations on $A_I$.  In more precision, $f$ is by definition annihilated by an ideal $\cJ\subset \cZ(\gg)$
of finite codimension.  We construct out of $f$ a certain vector valued function $\Phi_f$ on $A_I$ with values 
in a finite dimensional vector space $U_f$ with dimension bounded by $\dim \cZ(\gg)/ \cJ$.
The function $f\big|_{A_I}$ is then recovered by contracting $\Phi_f$ with a vector in $U_f$.  The function 
$\Phi_f$ in turn satisfies a first order linear differential system recorded in \eqref{eq-equadiffPhi}.

\par This section starts with a basic estimate for functions $f\in C_{temp,N}^\infty(Z)$ which will be crucial in the sequel: in a nutshell, we show that derivatives in direction of $\gh_I$ have decreasing decay in direction of $A_I^-$.  After that, we have a short algebraic subsection 
on invariant differential operators on $Z$, where we review in particular the contents of Appendix \ref{app-RBP}.  With these preparatory subsections, we then derive the differential equation \eqref{eq-equadiffPhi} for $\Phi_f$. From the solution formula for $\Phi_f$ in  Lemma \ref{lem-Phi}, we then derive 
a variety of basic growth estimates for $\Phi_f$.

\subsection{Differentiating tempered functions in direction of $\gh_I$}

Recall the basic notions about boundary degenerations related to subsets $I\subset S$ of spherical roots. Let us fix $I\subset S$ throughout this section. 
We define a piecewise linear functional on $\ga_I$ by  
\begin{equation}\label{eq-deftbetaI}
\tbeta_I(X)=\max_{\alpha\in S\setminus I}\alpha(X),  \qquad  X\in \ga_I\,,
\end{equation}
and note that $\tbeta_I(X)<0$ if $X\in \ga_I^{--}$. 
If $a\in A_I$ with $a=\exp X$, we set $a^{\tbeta_I}=e^{\tbeta_I(X)}$.

\par
We begin this section with a crucial estimate: 

\begin{lem}\label{lem-h-I}
Let $Y\in\gh_I$ and $N\in\N$. There exists a continuous semi-norm on $C_{temp,N}^\infty(Z)$, $p$, such that
$$\vert (L_Yf)(a)\vert\leq a^{\rho_Q+\tbeta_I}(1+\Vert\log a\Vert)^Np(f),\quad a\in A_I^{--},f\in C_{temp,N}^\infty(Z)\,.$$
\end{lem}

\begin{proof} 
On one hand, if $Y\in\gl\cap\gh$, 
$$(L_Yf)(a)=0,\quad a\in A_I\,.$$
Hence, the conclusion of the Lemma holds for $Y\in\gl\cap\gh$.

\par 
On the other hand, from the definition of $T_I$ (cf.~beginning of Section~\ref{sect-degenZ}), $\gl\cap\gh$ and the elements
$$Y_{-\alpha}=X_{-\alpha}+T_I(X_{-\alpha})\in\gh_I\,,$$ 
for $\alpha$ varying in $\Si_\gu$ and $X_{-\alpha}$ in $\gg^{-\alpha}$, generate $\gh_I$ as a vector space. 
By linearity, it then remains to get the result for $Y=Y_{-\alpha}$.

\par 
Let $a\in A_I$ and $\tilde{a}=\bs(a)$ (cf.~\eqref{eq-tildea} for the definition of $\bs$). Then let us show that
$$%\begin{equation}\label{eq-Yminusalphatilde}
\mathrm{Ad}(\tilde{a})Y_{-\alpha}=\tilde{a}^{-\alpha}Y_{-\alpha}\,.
$$%\end{equation}
One has $\mathrm{Ad}(\tilde{a})X_{-\alpha}=\tilde{a}^{-\alpha}X_{-\alpha}$ and 
$\mathrm{Ad}(\tilde{a})X_{\alpha,\beta}=\tilde{a}^{\beta}X_{\alpha,\beta}$. But $\alpha+\beta\in \langle I\rangle$. 
Hence, $\tilde{a}^{\alpha+\beta}=1$ as $a\in A_I$. Our claim follows.

\par
Let us get the statement for $(L_{Y_{-\alpha}}f)(a)$, $a\in A_I^{--}$ and $f\in C_{temp,N}(Z)$. One has: 
$$%\begin{array}{rcl}
(L_{Y_{-\alpha}}f)(a) =  (L_{\tilde{a}^{-1}}(L_{Y_{-\alpha}}f))(z_0)
 =  \tilde{a}^\alpha (L_{Y_{-\alpha}}L_{\tilde{a}^{-1}}f)(z_0)\,.
%\end{array}
$$
Recall that $\cM$ is the monoid in $\N_0[\Sigma_\gu]$ defined in \eqref{eq-cM} and $\langle I\rangle$ denotes the monoid in $\N_0[S]$ generated by $I$. Let us notice that: 
$$Y_{-\alpha}+\sum_{\beta\in\Si_\gu\cup\{0\}\textrm{ s.t.}\,\alpha+\beta\notin\langle I\rangle}X_{\alpha,\beta}\in\gh\,.$$
Hence one has:
$$\begin{array}{rcl}
(L_{Y_{-\alpha}}f)(a) & = & -\tilde{a}^\alpha\sum_{\beta\in\Si_\gu\cup\{0\}\textrm{ s.t.}\, \alpha+\beta\notin\langle I\rangle} 
(L_{X_{\alpha,\beta}}L_{\tilde{a}^{-1}}f)(z_0)\\
& = & -\sum_{\beta\in\Si_\gu\cup\{0\}\textrm{ s.t.}\, \alpha+\beta\notin\langle I\rangle} 
\tilde{a}^{\alpha+\beta}(L_{\tilde{a}^{-1}}L_{X_{\alpha,\beta}}f)(z_0)\,.
\end{array}$$
But $\tilde{a}^{\alpha+\beta}=a^{\alpha+\beta}$ as $a\in A_I\subset A_Z$ and $\alpha+\beta\in S$. 
Then, as $(L_{\tilde{a}^{-1}}L_{X_{\alpha,\beta}}f)(z_0)=L_{X_{\alpha,\beta}}f(a)$, one has:
\begin{equation}\label{eq-decomp-actY-alpha}
(L_{Y_{-\alpha}}f)(a) =  -\sum_{\beta\in\Si_\gu\cup\{0\}\textrm{ s.t.}\, \alpha+\beta\notin\langle I\rangle} 
a^{\alpha+\beta}(L_{X_{\alpha,\beta}}f)(a)\,.
\end{equation}
If $\alpha+\beta\notin \langle I\rangle$ as above and $L_{X_{\alpha,\beta}}f\neq 0$, 
one has $\alpha+\beta\in\cM\setminus\langle I\rangle$ and, from the definition of $\beta_I$ (cf.~\eqref{eq-deftbetaI}):
$$a^{\alpha+\beta}\leq a^{\tbeta_I},\quad a\in A_I^{--}\,.$$ 
Then 
$$\vert (L_{Y_{-\alpha}}f)(a)\vert \leq a^{\tbeta_I}\sum_{\beta\in\Si_\gu\cup\{0\}\textrm{ s.t.}\, 
\alpha+\beta\notin\langle I\rangle} \vert (L_{X_{\alpha,\beta}}f)(a)\vert\,.$$
Hence, we get the inequality of the Lemma for $Y=Y_{-\alpha}$ by taking
$$p=\sum_{\beta\in\Si_\gu\cup\{0\}\textrm{ s.t.}\, \alpha+\beta\notin\langle I\rangle}q_{N,X_{\alpha,\beta}}\,,$$
with $q_{N,X}(f):=q_N(L_Xf)$.
\end{proof}

\subsection{Algebraic preliminaries}\label{subsect-algpre}

For a real spherical space $Z=G/H$, we denote by $\D(Z)$ the algebra of $G$-invariant differential operators.
We recall the deformations $Z_I = G/H_I$ of $Z$ which were defined with $H_I$ to be connected. In particular, we
point out that $H_S=H_0$ and that $Z_S\to Z$ is possibly a proper covering.  However, we have $\D(Z)\subset \D(Z_S)$
naturally by Remark \ref{rem appD}.  Next we describe $\D(Z_I)$ as in Appendix \ref{app-RBP}.

\par Let $R$ denote the right regular representation of $G$ on $C^\infty(G)$.
Differentiating $R$ yields an algebra representation of the universal enveloping algebra $\cU(\gg)$ of $\gg_\C$:
$$R:\cU(\gg)\to\End(C^\infty(G)),\quad u\mapsto R(u)\,.$$
Set $\gb:= \ga + \gm +\gu$ and note that $\gb\subset \gg$ is a subalgebra with $\gg= \gb +\gh_I$ for all 
$I\subset S$. Note that $\gb\cap \gh_I= \ga_H +\gm_H$ for all $I\subset S$, where $\gm_H=\gm\cap \gh$. Let 
$\gb_H:=\ga_H +\gm_H$. 
With 
\begin{equation} \label{Diff0}\cU_I(\gb):= \{ u \in \cU(\gb)\mid   Xu \in \cU(\gg)\gh_I, \ \ X\in \gh_I\}\,,\end{equation} 
we obtain a subalgebra of $\cU(\gb)$ which features  $\cU(\gb)\gb_H$ as a two-sided ideal. 
Next, we explain the natural isomorphism 
\begin{equation} \label{Diff1}\D(Z_I) \simeq \cU_I(\gb)/ \cU(\gb)\gb_H\end{equation} 
from \eqref{diff1}. For that, we denote for $f_I\in C^\infty(Z_I)$ by $\tilde f_I\in C^\infty(G)$ 
its natural lift to a right $H_I$-invariant function in $G$.  Then, with regard to the quotient map 
$\pi: \cU(\gb)\to \cU(\gb)/ \cU(\gb)\gb_H$, we take $\tilde u \in \cU(\gb)$ to be any lift of $u \in \cU_I(\gb)/ \cU(\gb)\gb_H\subset 
 \cU(\gb)/ \cU(\gb)\gb_H$. 
Then we can define 
$$(R_I(u)f_I)(gH_I) :=  (R({\tilde u})\tilde f_I)(g), \qquad g\in G\,,$$
as the right hand side is independent of the particular choice of the lift $\tilde u$ of  $u$ and the representative 
$g$ of the coset $gH_I$. With this notion of $R_I$, the isomorphism in \eqref{Diff1} is 
implemented by the assignment 
\begin{equation} \label{morph UI} \cU_I(\gb)/ \cU(\gb)\gb_H\ni u \to R_I(u)\in \D(Z_I)\, . \end{equation}
For $f \in C^\infty(Z)\subset C^\infty(Z_S)$ and $u \in \D(Z)\subset \D(Z_S)$, we use the abbreviated 
notation $R(u) f$ without specifying any further index.  

\par
In the sequel, we consider $\D(Z_I)$ as a subspace of $\cU(\gb)/ \cU(\gb)\gb_H$ for any $I\subset S$. 
Notice that $\cU(\gb)/ \cU(\gb)\gb_H$ is naturally a module for $A_Z$ under the adjoint action, 
which yields us a notion of $\ga_Z$-weights of elements $u \in \cU(\gb)/ \cU(\gb)\gb_H$.

\par Recall the center $\ga_{Z,E}=\ga_S$ of $Z$, which has the property that $A_{Z,E}$ normalizes $H$ and as such acts 
on $Z$ from the right, commuting with the left $G$-action on $Z$. In particular, we obtain a natural 
embedding $S(\ga_{Z,E})\hookrightarrow \D(Z)$.  When applied to the real spherical space $Z_I=G/H_I$, $I\subset S$, we note that 
$\ga_I=\ga_{Z_I,E}$ and record the inclusion $S(\ga_I)\hookrightarrow \D(Z_I)$. 

\par We rephrase Theorem \ref{lem-defmu_0} from Appendix \ref{app-RBP} as: 

\begin{lem}\label{lem-defmu_I2}
For $I\subset S$, the following assertions hold:
\begin{enumerate}[(i)]
\item For any $u\in\D(Z)\subset \cU(\gb)/\cU(\gb)\gb_H$ and $X\in\ga_I^{--}$, the limit
$$\mu_I(u):=\lim_{t\to \infty} e^{t \ad X} u$$
exists in the vector space  $\cU(\gb)/\cU(\gb)\gb_H$, lies in the subspace 
$\cU(\gb)_I/ \cU(\gb)_I \gb_H$ and defines via $R_I$ (see \eqref{morph UI}) and defines an injective algebra morphism $\mu_I:\D(Z)\to\D(Z_I)$, which does not depend on $X$.
\item For any non-zero $u\in\D(Z)$, the $\ga_Z$-weights of $\mu_I(u)$ and $u$ are non-positive on $\ga_Z^-$ and the $\ga_Z$-weights of $\mu_I(u)-u$ are negative on $\ga_I^{--}$.
\end{enumerate}
\end{lem}
This Lemma shows that we can view $\D(Z_I)$ as a subalgebra of $\D(Z_\emptyset)$. 
Since $\gh_\emptyset= \gl\cap\gh +\bar \gu$ is of a particular simple shape, i.e., close to a parabolic, 
the algebra  $\D(Z_\emptyset)$ can be described easily. 
For that, let $M_H:=\exp(\gm_H)< M$ and keep in mind the standard isomorphim
\begin{equation} \label{DiffM}\D(M/M_H)\simeq\cU(\gm)^{M_H}/ (\cU(\gm)\gm_H\cap \cU(\gm)^{M_H})\, .\end{equation}

\begin{lem} The natural map 
$$ \Phi: S(\ga_Z) \otimes \left[\cU(\gm)^{M_H}/ (\cU(\gm)\gm_H\cap \cU(\gm)^{M_H}) \right]\to \cU_\emptyset(\gb)/ \cU(\gb)\gb_H,  \ \ u \otimes v \mapsto uv +  \cU(\gb)\gb_H$$
is an isomorphism.  In particular, via \eqref{Diff1} and \eqref{DiffM}, we obtain a natural isomorphim
of algebras
\begin{equation} \label{DZ-empty} 
\D(Z_\emptyset)\simeq S(\ga_Z)\otimes\D(M/M_H)\,.
\end{equation}
\end{lem}

\begin{proof} In the absolutely spherical case, this is  found in \cite[Section 6]{knop-hc} (what is called $X_h$, the horospherical 
deformation of a $G$-variety $X$, would correspond to our $Z_\emptyset$). The slightly more general case 
is an easy adaptation. In the following proof, we replace from \eqref{eq pol} onwards $H_\emptyset$ by its algebraic closure, which is legitimate by Remark \ref{rem appD}(b).
\par  Recall that 
$$\cU_\emptyset(\gb)=\{ u \in \cU(\gb)\mid  [X,u]\in \cU(\gg) \gh_\emptyset,  \ X\in \gh_\emptyset\}\, .$$
In particular,  $ \cU_\emptyset(\gb)$ is $\ad \ga$-invariant and we 
obtain a spectral decomposition 
$$ \cU_\emptyset(\gb)=\sum_{\lambda\in \ga^*}  \cU_\emptyset(\gb)_\lambda\, .$$
For $\lambda =0$, we further have 
$$\cU_\emptyset(\gb)_0 = \cU_\emptyset(\gb)\cap \cU(\ga+\gm)= \cU(\ga)(\cU(\gm)^{M_H}+\cU(\gm)\gm_H)\,,$$
from which we easily derive that $\Phi$ is injective. 

\par It remains to be seen that $\Phi$ is surjective. For that, it suffices to show that $[\cU_\emptyset(\gb)/ \cU(\gb)\gb_H]_\lambda\simeq \D(Z_\emptyset)_\lambda=0$ for $\lambda\neq 0$. To verify that, we pass to the graded level and first note that 
the symbol map gives an embedding 
\begin{equation}\label{eq pol}
\gr \D(Z_\emptyset) \hookrightarrow \Pol(T^*Z_\emptyset)^G\,,
\end{equation}
with $\Pol (T^*Z_\emptyset):=\C[T^*Z_\emptyset]$ the regular (polynomial) functions on the quasi-affine variety $T^*Z_\emptyset$.
We identify the cotangent bundle $T^*Z_\emptyset$ with $G\times^{H_\emptyset} (\gg/\gh_\emptyset)^*$ and 
obtain $\Pol(T^*Z_\emptyset)^G \simeq \Pol((\gg/\gh_\emptyset)^*)^{H_\emptyset}$. 
Thus we have  $\gr \D(Z_\emptyset) \subset \Pol((\gg/\gh_\emptyset)^*)^{H_\emptyset}$ naturally. 
Recall the invariant non-degenerate bilinear form $\kappa$ on $\gg$. This form yields a $G$-equivariant 
identification of $\gg$ with its dual $\gg^*$ and induces an $H_\emptyset$-equivariant identification 
of $(\gg/\gh_\emptyset)^*$ with $\gh_\emptyset^\perp:=\{X\in\gg\mid \kappa(X,Y)=0,Y\in\gh_\emptyset\}$. 
The proof of the Lemma  will then be completed by showing that the restriction map 
$$\Pol(\gh_\emptyset^\perp)^{H_\emptyset} \to \Pol(\gb_H^{\perp_{\gm+\ga}})$$
is injective, where $\gb_H^{\perp_{\gm+\ga}}:=\{X\in\gm+\ga\mid \kappa(X,Y)=0,Y\in\gb_H\}$.   
This is now fairly standard. Note that 
$\gh_\emptyset^{\perp}
=\gb_H^{\perp_{\gm+\ga}}+ \bar\gu$. Next let $X=X_\ga  + X_\gm \in \ga +\gm$ with $X_\ga\in \ga $ and  $X_\gm\in \gm$. 
Suppose further that  $\alpha(X_\ga)>0$ for $\alpha\in \Sigma(\ga,\gu)$.  Then, by a slight modification 
of  \cite[Lemma 2.5]{kks0},  we have 
\begin{equation} \label{bar U act} \Ad(\bar U) X = X +[X, \bar \gu]= X+ \bar \gu. \end{equation} 
 Now $\bar U\subset H_\emptyset$
and the fact that $Z$ (and hence $Z_\emptyset$) is unimodular implies that there exists 
an element $X_\ga$ as above which lies in $\ga_H^{\perp_\ga}$ (see Lemma \ref{lem unim-quasiaff} below). 
It follows then from \eqref{bar U act}  that any $f \in \Pol(\gh_\emptyset^\perp)^{H_\emptyset}$ is constant in 
the $\bar \gu$-variable of $\gh_\emptyset^\perp$, i.e., the restriction map above is injective. 
This completes 
 the proof of the Lemma.
\end{proof}

\begin{lem}\label{lem unim-quasiaff}
Let $Z$ be a unimodular real spherical space. Then the following assertions hold:
\begin{enumerate}[(i)]
\item $Z$ is quasi-affine, i.e., $\sZ(\C)=\sG(\C)/\sH(\C)$ is a quasi-affine algebraic variety.
\item There exists an $X\in \ga_H^{\perp_\ga}\simeq\ga_Z$ such that $\alpha(X)>0$ for all $\alpha\in\Sigma_\gu$.
\end{enumerate}
\end{lem}

\begin{proof}
\cite[Example 12.6 and Lemma 12.7]{planch-sph}.
\end{proof}

Let us denote by $\gZ(Z_\emptyset)$ the center of $\D(Z_\emptyset)$. We then obtain from \eqref{DZ-empty} that

\begin{equation} \label{DZ-emptyM} 
\gZ(Z_\emptyset)\simeq S(\ga_Z)\otimes\gZ(M/M_H)\,.
\end{equation}

 We wish to describe the image of 
the natural map $\cZ(\gg)\to \gZ(Z_\emptyset)\subset \D(Z_\emptyset)$ more closely, i.e., derive 
a slight extension of \cite[Lemma 6.4]{knop-hc}.  

\par In order to do so, we have to recall  first the construction of the Harish-Chandra homomorphism
and then relate it to the Knop homomorphism for $\gZ(M/M_H)$. 

\par We begin with a short  summary on the Harish-Chandra homomorphims. The natural inclusion $\cZ(\gg) \subset \cU(\ga)\cZ(\gm) \oplus \cU(\gg)\overline{\gn}$ yields,  via 
projection to the first summand, an injective algebra morphism 
$$\gamma_{0, \ga +\gm}: \cZ(\gg)\to \cU(\ga)\otimes\cZ(\gm)\, .$$
With $\gt\subset\gm$ a maximal torus (which will be specified more closely below), we obtain with $\gj:=\ga+ \gt$ a Cartan subalgebra 
of $\gg$.  
We choose a positive system of roots $\Sigma^+(\gj_\C)$ of the the root system of $\gg_\C$ with respect to 
$\gj_\C$ such that the nonzero restrictions to $\ga$ yield the root spaces of $\gn$. 
Note that all roots 
are real-valued on $\gj_\R:= \ga+ i \gt$ and we denote by $\rho_\gj\in  \gj_\R^*$ the corresponding half sum. 
Then, similar to what was just explained, we obtain, by projection along the negative $\gm_\C$-root spaces with respect to 
$\gt_\C$, an injective algebra morphism 
 $\gamma_{0,\gm}:\cZ(\gm)\to \cU(\gt)$. Putting matters together, we obtain with 
 $$\gamma_0:= (\Id_{S(\ga)} \otimes \gamma_{0,\gm}) \circ \gamma_{0,\ga +\gm}$$
an injective algebra morphism $\gamma_0: \cZ(\gg)\to \cU(\gj)$. If we identify 
 $\cU(\gj)=S(\gj)$ with the polynomials $\C[\gj_\C^*]$ on $\gj_\C^*$,  the Harish-Chandra
 isomorphism  $\gamma: \cZ(\gg)\to \cU(\gj)^{W_\gj}$ is then obtained by twisting $\gamma_0$ with the $\rho_\gj$-shift, i.e., 
 $\gamma (z)(\cdot)= \gamma_0(z) (\cdot +\rho_\gj)$ as polynomials on $\gj_\C^*$.  For our purpose we are in fact 
 more interested in the unnormalized Harish-Chandra morphism
 $\gamma_0: \cZ(\gg) \to S(\gj)\simeq \cU(\gj)$. 
 
 \par Next we recall the Knop homomorphism for $\gZ(M/M_H)$.  Set $\gt_H:=\gt\cap \gh$ and $\gt_Z:= \gt/ \gt_H$.  Note that $M/M_H$ is affine, i.e., the complexification $M_\C/ (M_H)_\C$ is an affine homogeneous space. 
 We will request, from our choice of $\gt$, that the complexification of $\gt_Z$ is a flat for $M_\C/ (M_H)_\C$, i.e., compatible with 
 the local structure theorem (cf.~\cite[Theorem 4.2]{kk} applied to $Y=X=M_\C/ (M_H)_\C$ and $k=\C$). Set  $\rho_\gm:=\rho_\gj \big|_{i\gt}$ and let $W_M$ be the little Weyl group 
 of the affine space $M_\C/ (M_H)_\C$.  Then \cite[Theorem in the Introduction part(a)]{knop-hc} yields the Knop isomorphism 
 $$k: \gZ(M/M_H)\to \C[\gt_{Z,\C}^*+\rho_\gm]^{W_M}\, .$$
 For our purpose it is easier to work with the unnormalized 
Knop homomorphism which yields us an algebra monomorphism: 
$$k_0:  \gZ(M/M_H)\to S(\gt)/ S(\gt)\gt_H$$
The important thing to notice here is that the Knop homomorphism $k_0$ is compatible  with 
the unnormalized  Harish-Chandra homomorphism $\gamma_{0,\gm}: \cZ(\gm) \to S(\gt)$ in the
sense that the diagram 
\begin{equation} \label{diagramk}
    \xymatrix{\cZ(\gm) \ar@{>}[d]_{\gamma_{0,\gm}} \ar@{>}[r] &\gZ(M/M_H) \ar@{>}[d]_{k_0} \\
      S(\gt)  \ar@{>}[r] &S(\gt)/S(\gt)\gt_H\, .}
        \end{equation}
is commutative, see \cite[Lemma 6.4]{knop-hc}. 
To summarize, we obtain from \eqref{DZ-empty}, the just explained construction of the Harish-Chandra homomorphism and 
\eqref{diagramk} an injective algebra morphism
\begin{equation} \label{algebra j} j_0: \gZ(Z_\emptyset)\to S(\gj)/ S(\gj)(\ga_H +\gt_H)\end{equation}
together with the following commutative diagram 
\begin{equation} \label{diagram1}
    \xymatrix{\cZ(\gg) \ar@{>}[d]_{\gamma_0} \ar@{>}[r] &\gZ(Z_\emptyset) \ar@{>}[d]_{j_0} \\
      S(\gj) \ar@{>}[r] &S(\gj)/S(\gj)(\ga_H +\gt_H)\, .}
        \end{equation}
In this diagram, the upper lower horizontal arrow is obtained from the natural $\cZ(\gg)$-module structure of 
$\gZ(Z_\emptyset)$ and the lower horizontal arrow is the natural projection $S(\gj)\to S(\gj)/S(\gj)(\ga_H + \gt_H)$. 

\begin{example} {\rm (Triple space continued)} For the triple space $Z=G_o\times G_o \times G_o/ \operatorname{diag}G_o$ we have $\gg=\gg_o\times \gg_o\times \gg_o$ and thus $\cZ(\gg)= \C[\cC_1, \cC_2, \cC_3]$ with 
$\cC_i$ the Casimir operator of the $i$-th $\gg_o$-factor in $\gg$.  Also we have $\gj=\ga$ and the Brion-Knop
little Weyl group $W_Z$ coincides with the Weyl group $W_\ga\simeq (\Z/2\Z)^3$. Thus, by the Knop isomomorphism, we have 
$\D(Z)\simeq \cZ(\gg)$. Now $\gZ(Z_\emptyset)= S(\ga) \simeq \C[z_1, z_2, z_3]$ with  $z_i$ the co-root coordinates and the above algebra inclusion 
$$\cZ(\gg)= \C[\cC_1, \cC_2, \cC_3]\hookrightarrow \gZ(Z_\emptyset)= \C[z_1, z_2, z_3]$$
is given by the assignment 
$$\cC_i\mapsto {1\over 4}z_i^2 -{1\over 2} z_i\, .$$
\end{example}

\par Note that $\D(Z_I)$ is naturally a module for $\cZ(\gg)$, the center of $\cU(\gg)$. 
We define by $\bD(Z_I)$ the commutative subalgebra of $\D(Z)$ which is generated 
by $S(\ga_I)$ and the image of $\cZ(\gg)$ in $\D(Z_I)$.

\begin{lem}\label{lem: fin gen}  
The $\cZ(\gg)$-module $\gZ(Z_\emptyset)$ is finitely generated.  In particular, $\bD(Z_I)$ is a finitely generated 
$\cZ(\gg)$-module for all $I\subset S$. 
\end{lem}

\begin{proof} Since $S(\gj)$ is a module of finite rank over $S(\gj)^{W_\gj}$ (Chevalley's theorem), we obtain from 
\eqref{diagram1} and $\operatorname{im} \gamma= S(\gj)^{W_\gj}$ 
that $\gZ(Z_\emptyset)$ is a  finitely generated $\cZ(\gg)$-module.  
Since $\bD(Z_I)$ is naturally 
a submodule of $\bD(Z_\emptyset)$ via the injective algebra morphism $\mu_I$ of Lemma \ref{lem-defmu_I2}, the second assertion follows from the fact that 
$\cZ(\gg)\simeq S(\gj)^{W_\gj}$ is a polynomial ring (again by Chevalley) and hence noetherian. 
\end{proof}

\par Let us denote by $\bD_0(Z)$ the image of $\cZ(\gg)$ in $\bD(Z_\emptyset)\subset \D(Z)$. As we will see later, some aspects 
become simpler if we work with the slightly smaller algebra $\bD_0(Z)$. It follows from Lemma \ref{lem: fin gen} that 
$\bD(Z_I)$ is a finitely generated $\mu_I(\bD_0(Z))$-module.

\par 
Fix now $I\subset S$. Since $\bD(Z_I)$ is finitely generated over $\mu_I(\bD_0(Z))$, there exists a finite dimensional vector subspace $U$ of $\bD(Z_I)$ containing $1$ such that the map
\begin{equation}\label{eq: defU}
\begin{array}{rcl}
\mu_I(\bD_0(Z))\otimes U & \longrightarrow & \bD(Z_I)\\
v\otimes u & \longmapsto & vu
\end{array}
\end{equation}
is a linear surjective map.
\\

Let $\cI$ be a finite codimensional ideal of $\bD_0(Z)$ and let $\cI':=\mu_I(\cI)$. 
Let $C=C(\cI)$ be a finite dimensional vector subspace of $\mu_I(\bD_0(Z))$ containing $1$ such that $\mu_I(\bD_0(Z))=C+\cI'$. Hence:

\begin{equation}\label{eq: bdZ1}  
\bD(Z_I)  =  (C+\cI')U =  CU+\cI' U\,,
\end{equation} 
where $\cI' U$ (resp.~$CU$) is the linear span of $\{vu\mid v\in\cI',u\in U\}$ (resp.~$\{vu\mid v\in C,u\in U\}$).

\par 
Since $\cI'$ is an ideal on $\bD_0(Z)$, we obtain that: 
\begin{equation} \label{eq: J-D} 
\cI' U=\cI'\mu_I(\bD_0(Z))U=\cI' \bD(Z_I)=\bD(Z_I)\cI'\,.
\end{equation}
Hence, \eqref{eq: bdZ1} implies that: 
\begin{equation}\label{eq: bdZ2} 
\bD(Z_I)=CU+\bD(Z_I)\cI' \, .
\end{equation}

In case $\cI$ is a one codimensional ideal of $\bD_0(Z)$, one may and will take $C=\C 1$, and then $CU=U$.

\par
In general, we choose a finite dimensional subspace $U_\cI\subset CU$, possibly depending on $\cI$,  such that the sum in \eqref{eq: bdZ2} becomes direct: 
\begin{equation}\label{eq: bdZ3} 
\bD(Z_I)= U_{\cI}\oplus\bD(Z_I)\cI' \, .
\end{equation}
Let $s_\cI$, resp.~$q_\cI$, be the linear map from $\bD(Z_I)$ to $U_\cI$, resp.~$\bD(Z_I)\cI'$, deduced from this  
direct sum decomposition. The algebra $\bD(Z_I)$ acts on $U_\cI$ by a representation $\rho_\cI$ defined by:
%\begin{equation}\label{eq-rhoJ}
\begin{equation}\label{eq-reprhoI}
\rho_\cI(v)u=s_\cI(vu),\quad v\in \bD(Z_I),u\in U_\cI\,.
\end{equation}
%\end{equation}
In fact: 
\begin{res-nn}
The representation $(\rho_\cI,U_\cI)$ is isomorphic to the natural representation of $\bD(Z_I)$ on $\bD(Z_I)/\bD(Z_I)\cI'$.
\end{res-nn}
We notice that, for $v\in\bD(Z_I)$ and $u\in U_\cI$, 
\begin{equation}\label{eq-decompuw}
vu=\rho_\cI(v)u+q_\cI(vu).
\end{equation}
If $(u_i)_{i=1,\dots,n}$ is a basis of $U$, then we obtain, from $\bD(Z_I)\cI'=\mu_I(\cI) U=U\mu_I(\cI)$ (see \eqref{eq: J-D}),
elements $z_i=z_i(v,u,\cI)\in\cI$, not necessarily unique, such that: 
\begin{equation}\label{eq-zi}
q_\cI(vu)=\sum_{i=1}^n u_i\mu_I(z_i)\, .
\end{equation}
Moreover, we record from Lemma \ref{lem-defmu_I2}(ii) that: 
\begin{res}\label{eq-barzij}
\center{$\mu_{I}(z_i)-z_i$ has $\ga_Z$-weights non-positive on $\ga_Z^-$ and negative on $\ga_I^{--}$.}
\end{res}

\par
In order to use it later, we denote by $\cF=\cF(\cI)$ the (finite) set of all these $\ga_Z$-weights which occur when $v$ describes $\ga_I\subset\bD(Z_I)$ and $u$ describes $U_\cI$. Let us define a piecewise linear functional on $\ga_Z$ by:
\begin{equation}\label{eq-defbetaI}
\beta_I(X):=\max_{\lambda\in\cF\cup (S\setminus I)}\lambda(X),\quad X\in \ga_Z\,.
\end{equation}
Note that $\beta_I\big|_{\ga_Z^-}\leq 0$ and $\beta_I\big|_{\ga_I^{--}}< 0$. 

\subsection{The function $\Phi_f$ on $A_Z$ and related differential equations}

%If $\cI$ is a finite codimensional ideal in $\cZ(\gg)$ and $N\in\N$, we denote by $\cA_{temp,N}(Z:\cI)$ 
%(resp.~$\cA_{temp}(Z:\cI)$) the space of functions $f\in\cA_{temp,N}(Z)$ (resp.~$\cA_{temp}(Z)$) annihilated by $\cI$.\par

Fix $N\in\N$ and $\cI$ a finite codimensional ideal in $\bD_0(Z)$. Recall the surjective morphism $\cZ(\gg)\to\bD_0(Z)$ and let $\cJ$ be the corresponding preimage of $\cI$. Set
$$\cA_{temp}(Z:\cI):=\cA_{temp}(Z:\cJ)\,,$$
with $\cA_{temp}(Z:\cJ)$ defined in \eqref{eq-defAtemp}.

\par Recall that we identified for any $I\subset S$ the algebra 
$\bD(Z_I)$ as a subspace of $\cU(\gb)/\cU(\gb)\gb_H$. Now given $f\in C^\infty(Z)$, 
we denote by $\tilde f\in C^\infty(G)$ its lift to a right $H$-invariant smooth function on $G$. 
For $u\in \cU(\gb)/\cU(\gb)\gb_H$ we let further $\tilde u \in \cU(\gb)$ be any lift. Then, for all 
$a_Z\in A_Z$  the notion 
$$(R_u f)(a_Z) :=  (R(\tilde u)\tilde f)(\tilde a_Z)$$
is defined, i.e., independent of the lift $\tilde u$ and the section ${\bf s}$. 

\par Recall that $(\rho_\cI,U_\cI)$ is the finite dimensional $\bD(Z_I)$-module defined in \eqref{eq-reprhoI}
and in particular $U_\cI\subset \bD(Z_I)\subset \cU(\gb)/\cU(\gb)\gb_H$.
For any $f\in\cA_{temp,N}(Z:\cI)$, let us define a function $\Phi_f:A_Z\to U_\cI^*$ by: 
\begin{equation}\label{eq-Phif-wI}
\la\Phi_f(a_Z),u\ra:=(R_{u}f)(a_Z),\quad u\in U_\cI,a_Z\in A_Z\,.
\end{equation}
Hence, for $X\in\ga_I\subset \bD(Z_I)$, 
\begin{equation}\label{eq-derivPhif}
\la(R_X\Phi_f)(a_Z),u\ra=(R_{Xu}f)(a_Z),\quad a_Z\in A_Z,u\in U_\cI\,.
\end{equation} 
Hence, by using \eqref{eq-decompuw} and \eqref{eq-zi} for $Xu$, one gets 
\begin{equation}\label{eq-phiPhi}
R_{X}\Phi_f={}^t\!\rho_\cI(X)\Phi_f+\Psi_{f,{X}},\quad X\in A_I\,,
\end{equation}
where $\Psi_{f,X}:A_Z\to U_\cI^*$ is given by: 
\begin{equation}\label{eq-defPsi}
\la\Psi_{f,X}(a_Z),u\ra:=\sum_{i}(R_{u_i\mu_{I}(z_i)}f)(a_Z),\quad a_Z\in A_Z,u\in U_\cI\,,
\end{equation}
with $z_i=z_i(X,u,\cI)$ given by \eqref{eq-zi}.

\par
Since $R_{z_i}f=0$ as $z_i\in\cI$ and $f$ is annihilated by $\cI$, one then has:  
\begin{equation}\label{eq-defPsi2}
\la\Psi_{f,X}(a_Z),u\ra=\sum_{i}(R_{u_i(\mu_{I}(z_i)-z_i)}f)(a_Z),\quad a_Z\in A_Z,u\in U_\cI\,.
\end{equation}

One sets: 
\begin{equation}\label{eq-defGamma}
\Gamma_\cI(X)={}^t\!\rho_\cI(X),\quad X\in\ga_I\,.
\end{equation}
Hence, we arrive at the fundamental first order ordinary differential equation: 
\begin{equation}\label{eq-equadiffPhi}
R_X\Phi_f=\Gamma_\cI(X)\Phi_f +\Psi_{f,X},\quad X\in\ga_I\,.
\end{equation}
Notice that $\Gamma_\cI$ is a representation of the abelian Lie algebra $\ga_I$ on $U_\cI^*$. 

\par For $\lambda\in\ga_{I,\C}^*$, 
one denotes by $U_{\cI,\lambda}^*$ the space of joint generalized eigenvectors of $U_{\cI}^*$ by 
the endomorphisms $\Gamma_\cI(X)$, $X\in\ga_I$, for the eigenvalue $\lambda$. Let $\cQ_\cI$ be the (finite) 
subset of $\lambda\in\ga_{I,\C}^*$ such that $U_{\cI,\lambda}^*\neq\{0\}$. One has:
\begin{equation}\label{eq: decompeigenvecUI}
U_\cI^*=\bigoplus_{\lambda\in\cQ_\cI}U_{\cI,\lambda}^*\,.
\end{equation}
If $\lambda\in\cQ_\cI$, let $E_\lambda$ be the projector of $U_\cI^*$ onto $U_{\cI,\lambda}^*$ parallel  to the sum of the other $U_{\cI,\mu}^*$'s. 

\par Define, for $\lambda\in \cQ_\cI$, 
%\begin{equation}\label{eq-projlamba}
$$
\Phi_{f,\lambda}:=E_\lambda\circ\Phi_f\,.
$$

\par We conclude this subsection with the solution formula  for the system \eqref{eq-equadiffPhi} (see the next  Lemma \ref{lem-Phi})
and with two elementary estimates for $\Phi_f$ and $\Psi_{f,X}$ in Lemma \ref{lem-Phi-Psi-cinfini} below.

\begin{lem}\label{lem-Phi}
Let $f\in\cA_{temp}(Z:\cI)$. One has, 
\begin{enumerate}
\item[(i)] for all $a_Z\in A_Z$, $t\in\R$, $X\in\ga_I$,
$$\Phi_f(a_Z\exp(tX))=e^{t\Gamma_\cI(X)}\Phi_f(a_Z)+\int_0^te^{(t-s)\Gamma_\cI(X)}\Psi_{f,X}(a_Z\exp(sX))\,ds\,,$$
\item[(ii)] for all $a_Z\in A_Z$, $t\in\R$, $X\in\ga_I$, $\lambda\in\cQ_\cI$, 
$$\Phi_{f,\lambda}(a_Z\exp(tX))=e^{t\Gamma_\cI(X)}\Phi_{f,\lambda}(a_Z)+\int_0^tE_\lambda 
e^{(t-s)\Gamma_\cI(X)}\Psi_{f,X}(a_Z\exp(sX))\,ds\,.$$
\end{enumerate}
\end{lem}
\begin{proof}
The equality~(i) is an immediate consequence of~\eqref{eq-equadiffPhi}. 
Indeed, we apply the elementary result on first order linear differential equation to the function 
$s\mapsto F(s)=\Phi_f(a_Z\exp(sX))$, whose derivative $F'(s)=(R_{X}\Phi_f)(a_Z\exp(sX))$ satisfies
$$F'(s)=\Gamma_\cI(X)F(s)+\Psi_{f,X}(a_Z\exp(sX))\,.$$
The equality~(ii) follows by applying $E_\lambda$ to 
both sides of the equality of~(i).
\end{proof}
We recall the definition of $\beta_I$ from \eqref{eq-defbetaI}.
 
\begin{lem}\label{lem-Phi-Psi-cinfini} 
Let $N\in\N$.
\begin{enumerate}
\item[(i)] There exists a continuous semi-norm on $C_{temp,N}^\infty(Z)$, $p$, such that
$$
\Vert L_v \Phi_f(a_Z)\Vert \leq  a_Z^{\rho_Q}(1+\Vert\log a_Z\Vert)^Np(L_vf)$$
for all $v\in \cU(\ga)$, $a_Z\in A_Z^-$ and  $f\in\cA_{temp,N}(Z:\cI)$. 
\item[(ii)] There exists a continuous semi-norm $q$ on $C_{temp,N}^\infty(Z)$ such that, for all compact subset $\Omega_A\subset A_Z$, there exists a constant $C=C(\Omega_A)>0$ with:
$$
\Vert L_v \Psi_{f,X}(a_Z)\Vert \leq  Ca_Z^{\rho_Q+\beta_I}(1+\Vert\log a_Z\Vert)^N\Vert X\Vert q(L_vf)$$
for $a_Z\in \Omega_AA_Z^-$, $X\in \ga_I$ and $f\in\cA_{temp,N}(Z:\cI)$.
\end{enumerate}
\end{lem}

\begin{proof} (i) We first consider the case of $v=1$. 
Let $u\mapsto u^t$ denote the principal anti-automorphism of $U(\gg)$. 

\par
Let $u\in\bD(Z_I)\subset\cU(\ga_Z+\gm_Z+\gu)$. One has: 
$$(R_uf)(a_Z)=(L_{(\Ad(a_Z)u)^t}f)(a_Z)\,.$$
Since $\Ad(\ga_Z^-)$ contracts the $\ga_Z$-weights of $u$ (see Lemma \ref{lem-defmu_I2}(ii)), the assertion for $v=1$ follows from the continuity of the left regular action of $\cU(\gg)$ on $C_{temp,N}^\infty(Z)$. The more general 
case is obtained by the fact that the assignment $f \mapsto \Phi_f$ is $A$-equivariant for the left regular 
representation of $A$ on functions on $Z$, resp. $A_Z$. 
 
\par
(ii) We recall from \eqref{eq-defPsi2} that: 
$$\la\Psi_{f,X}(a_Z),u\ra=\sum_{i}(R_{u_i(\mu_{I}(z_i)-z_i)}f)(a_Z),\quad a_Z\in A_Z,u\in U_\cI\,,$$
with $z_i=z_i(X,u,\cI)$. In particular, this identity readily reduces to the case of $v=1$ as left and right regular 
representation commute.  

\par 
Since the $\ga_Z$-weights of $u_i$ are non-positive on $\ga_Z^-$ (see Lemma \ref{lem-defmu_I2}(ii)), we obtain that $u_i(\mu_{I}(z_i)-z_i)$ decomposes into a finite sum over $\cF-(\ga_Z^-)^\star$ of $\ga_Z$-weight vectors:
$$u_i(\mu_{I}(z_i)-z_i)=\sum_{\lambda}v_{i,\lambda}\,.$$
Here, $(\ga_Z^-)^\star$ denotes the dual cone of $\ga_Z^-$. 
Then:
$$
\begin{array}{rcl}
\la\Psi_{f,X}(a_Z),u\ra&=&\sum_{i}\sum_{\lambda}(L_{(\Ad(a_Z) (v_{i,\lambda}))^t}f)(a_Z)\\
&&\\
&=& \sum_{i}\sum_{\lambda}a_Z^{\lambda} (L_{v_{i,\lambda}^t }f)(a_Z)\,.
\end{array}
$$
Let $k:=\max_{u\in U_\cI,X\in\ga_I}(\deg(v_{i,\lambda}))$. Assume first that $\Omega_A=\{1\}$ and $\Vert X\Vert=1$. Then it follows from the continuity of the left action of $\cU(\gg)$ on $C_{temp,N}^\infty(Z)$ and the definition of $\beta_I$ that there is an appropriate Sobolev norm $q=p_{N,k}$ such that the bound in (ii) holds for $C=1$. In general, if $u\in\cU(\gg)$, $a\in\Omega_A$ and $a_Z\in A_Z^-$, one has:
$$(L_uf)(aa_Z)=L_{a^{-1}}(L_{\Ad(a^{-1})u}f)(a_Z)$$
and the assertion follows from: 
$$q(L_{a^{-1}}f)\leq Cq(f),\quad f\in C_{temp,N}^\infty(Z),a\in \Omega_A\,.$$
 \end{proof}

\subsection{The decomposition of $\Phi_f$ into eigenspaces}

We recall the representation $\Gamma_\cI: \ga_I \to \End (U_\cI^*)$ of the abelian 
Lie algebra $\ga_I$ from \eqref{eq-defGamma} and $\cQ_\cI$ the set of its generalized $\ga_I$-eigenvalues.

\par We endow $U_{\cI}^*$ with a scalar product and, if $T\in\mathrm{End}(U_\cI^*)$, we denote by $\Vert T\Vert$ its 
Hilbert--Schmidt norm. 
It is clear that, for any $\lambda\in\cQ_\cI$, the projector $E_\lambda$ defined just after \eqref{eq: decompeigenvecUI} commutes with the operators $\Gamma_\cI(X)$, $X\in\ga_I$. 
For $\lambda\in\cQ_\cI$, we set 
$$E_\lambda(X):=e^{-\lambda(X)}\left( E_\lambda \circ  e^{\Gamma_\cI(X)}\right),\quad X\in\ga_I\,.$$
As $E_\lambda\circ [\Gamma_\cI(X)-\lambda(X)\Id_{U_\cI^*}]$ is nilpotent, one readily obtains that: 
\begin{lem}\label{lem-Elambda-cont}
Let $\lambda\in\cQ_\cI$. We can choose $c\geq 0$ such that:
$$\Vert E_\lambda(X)\Vert\leq c(1+\Vert X\Vert)^{{N_\cI}},\quad X\in\ga_I\,,$$
where {$N_\cI$} is the dimension of $U_\cI$.
\end{lem}

Next, we decompose $\cQ_\cI$ into three disjoints subsets $\cQ_\cI^+$, $\cQ_\cI^0$ and $\cQ_\cI^-$ as follows:
\begin{res-nn}
\begin{enumerate}
\item[(1)] $\lambda\in\cQ_\cI^+$ if $\mathrm{Re}\,\lambda(X_I)>\rho_Q(X_I)$ for some $X_I\in\ga_I^{--}$,
\item[(2)] $\lambda\in\cQ_\cI^0$ if $\mathrm{Re}\,\lambda(X_I)=\rho_Q(X_I)$ for all $X_I\in\ga_I^{--}$,
\item[(3)] $\lambda\in\cQ_\cI^-$ if $\lambda\notin\cQ_\cI^+\cup\cQ_\cI^0$, i.e., for all $X_I\in\ga_I^{--}$, $\mathrm{Re}\,\lambda(X_I)\leq \rho_Q(X_I)$ and there exists $X_I\in\ga_I^{--}$ such that $\mathrm{Re}\,\lambda(X_I)< \rho_Q(X_I)$.
\end{enumerate}
\end{res-nn}

The next two propositions will be central for the definition of the constant term in the next section. 
We first state the results and then provide the proofs in a sequence of lemmas. The proofs of these results follow closely the work of 
Harish-Chandra (cf.~\cite[Section~22]{hc1}): to see the analogy replace $M_1^+$ in 
\cite{hc1} by $A_Z^-$ and $M_1$ by $A_{Z_I}^-$.

\begin{prop}\label{prop-main1} 
Let $\lambda\in \cQ_\cI^0$ and $f\in\cA_{temp}(Z:\cI)$.  Then, for $X_I\in \ga_I^{--}$,
the following limit 
$$\lim_{t\to +\infty}e^{-t\Gamma_\cI(X_I)}\Phi_{f,\lambda}(a_Z\exp(tX_I)),
\quad a_Z\in A_Z\, , $$
exists and is independent of $X_I\in \ga_I^{--}$. 
\end{prop}

For $\lambda\in \cQ_\cI^0$ and $f\in\cA_{temp}(Z:\cI)$, we now set 
\begin{equation}\label{def111} 
\Phi_{f,\lambda,\infty}(a_Z):=\lim_{t\to +\infty}e^{-t\Gamma_\cI(X_I)}\Phi_{f,\lambda}(a_Z\exp(tX_I)),
\quad a_Z\in A_Z\, . 
\end{equation}

Further we define
\begin{equation}\label{eq-phiinfty+-}
\Phi_{f,\lambda,\infty}(a_Z):=0,\quad a_Z\in A_Z,\lambda\in\cQ_\cI^+\cup\cQ_\cI^-,f\in\cA_{temp}(Z:\cI)\,.
\end{equation}

\begin{prop} \label{prop-main2} 
Let $\lambda\in \cQ_\cI$ and $f\in\cA_{temp}(Z:\cI)$. Then there exists $\delta>0$ such that for all  $a_Z\in A_Z$, $X_I\in\ga_I^{--}$ and $t\geq 0$: 
$$
\begin{array}{cl}
&\Vert \Phi_{f,\lambda}(a_Z\exp(tX_I))-\Phi_{f,\lambda,\infty}(a_Z\exp(tX_I))\Vert\\
&\\
\leq &\displaystyle{e^{t(\rho_Q+\delta\beta_I)(X_I)}\Big(\Vert E_\lambda(tX_I)\Vert \Vert \Phi_f(a_Z)\Vert} \\
&\displaystyle{+\int_0^\infty e^{-s(\rho_Q+\beta_I/2)(X_I)}\Vert E_\lambda((t-s)X_I)\Vert\Vert \Psi_{f,X_I}(a_Z\exp(sX_I))\Vert\,ds\Big)\,.}
\end{array}
$$
\end{prop}

\subsubsection{Proof of Proposition \ref{prop-main1}}

We say that an integral depending on a parameter converges uniformly if the absolute value of the integrand is bounded by an integrable function independently of the parameter.
\begin{lem}\label{lem-int-general}
Let $\lambda\in\cQ_\cI$ and $X_I\in\ga_I^{--}$ be such that $\mathrm{Re}\,\lambda(X_I)>(\rho_Q+\beta_I)(X_I)$. Then 
\begin{enumerate}
\item[(i)] The integral
$$\int_0^\infty E_\lambda e^{-s\Gamma_\cI(X_I)}\Psi_{f,X_I}(a_Z\exp(sX_I))\,ds$$
converges uniformly on any compact subset of $A_Z$.
\item[(ii)] The assignment 
$$a_Z\mapsto \int_0^\infty E_\lambda e^{-s\Gamma_\cI(X_I)}\Psi_{f,X_I}(a_Z\exp(sX_I))\,ds$$
is a well-defined map on $A_Z$. Its derivative along $u\in S(\ga_Z)$ is given by derivation under the integral sign.
\end{enumerate}
\end{lem}
\begin{proof}
One has
$$E_\lambda e^{-s\Gamma_\cI(X_I)}= e^{-s\lambda(X_I)}E_\lambda e^{s(\lambda(X_I)-\Gamma_\cI(X_I))}=e^{-s\lambda(X_I)}E_\lambda(-sX_I)\,.$$
Hence, from Lemma~\ref{lem-Elambda-cont}, one has:
\begin{equation}\label{eq-major-Elambda}
\Vert E_\lambda e^{-s\Gamma_\cI(X_I)}\Vert\leq c(1+\Vert sX_I\Vert)^{N_\cI} e^{-s\mathrm{Re}\,\lambda(X_I)}\,.
\end{equation}
Using Lemma~\ref{lem-Phi-Psi-cinfini}(ii), \eqref{eq-major-Elambda} and the assumption $\mathrm{Re}\,\lambda(X_I)>(\rho_Q+\beta_I)(X_I)$, we obtain that the integral in~(i) converges uniformly on compact subsets of $A_Z$. 

\par The assertion from (ii) follows in the same way and using the theorem on derivatives of integral depending of a parameter.
\end{proof}

Fix $N\in\N$ such that $f\in\cA_{temp,N}(Z:\cI)$ and $\lambda\in\cQ_\cI$ and put, for $X_I$ as in Lemma~\ref{lem-int-general}, 
i.e., $X_I\in\ga_I^{--}$ such that $\Reel\lambda(X_I)>(\rho_Q+\beta_I)(X_I)$:
\begin{equation}\label{eq-defPhiinfty}
\Phi_{f,\lambda,\infty}(a_Z,X_I)  :=  \displaystyle{\lim_{t\to +\infty}e^{-t\Gamma_\cI(X_I)}\Phi_{f,\lambda}(a_Z\exp(tX_I))},
\quad a_Z\in A_Z\,.
\end{equation}
It follows from Lemmas~\ref{lem-Phi}(ii) and~\ref{lem-int-general} that this limit exists and is $C^\infty$ on $A_Z$. Moreover
\begin{equation}\label{eq-LuPhiinfty}
%\begin{array}{r}
\Phi_{f,\lambda,\infty}(a_Z,X_I)=\Phi_{f,\lambda}(a_Z)+\displaystyle{\int_0^\infty E_\lambda 
e^{-s\Gamma_\cI(X_I)}\Psi_{f,X_I}(a_Z\exp(sX_I))\,ds,}%\\
\quad a_Z\in A_Z\,.
%\end{array}
\end{equation}

\begin{lem}\label{lem-Phiinfty-constant}
Let $X_1,X_2\in\ga_I^{--}$ and suppose that 
$$\mathrm{Re}\,\lambda(X_i)>(\rho_Q+\beta_I)(X_i),\quad i=1,2\,.$$
Then 
$$\Phi_{f,\lambda,\infty}(a_Z,X_1)=\Phi_{f,\lambda,\infty}(a_Z,X_2),\quad a_Z\in A_Z\,.$$
\end{lem}
\begin{proof}
Same as the proof of \cite[Lemma~22.8]{hc1}. We give it for sake of completeness. 
Let $a_Z\in A_Z$. Applying Lemma~\ref{lem-Phi}(ii) to $a_Z\exp(t_1X_1)$ instead of $a_Z$, $X_2$ instead of $X$ and $t_2$ instead of $t$, one gets:
$$
\begin{array}{cl}
& e^{-\Gamma_\cI(t_1X_1+t_2X_2)}\Phi_{f,\lambda}(a_Z\exp(t_1X_1)\exp(t_2X_2))\\  
&\\
 = & e^{-t_1\Gamma_\cI(X_1)}\Phi_{f,\lambda}(a_Z\exp(t_1X_1))\\
& +\displaystyle{\int_0^{t_2}E_\lambda e^{-\Gamma_\cI(t_1X_1+s_2X_2)}\Psi_{f,X_2}(a_Z\exp(t_1X_1+ s_2X_2))\,ds_2\,,}
\end{array}
$$
for $t_1,t_2> 0$. From Lemmas~\ref{lem-Elambda-cont} and~\ref{lem-Phi-Psi-cinfini}(ii) applied to $X=t_1X_1+s_2X_2$ and $(X,a_Z)=(X_2,a_Z\exp(t_1X_1+ s_2X_2))$ respectively, one sees that:
$$
\int_0^\infty \Vert E_\lambda e^{-\Gamma_\cI(t_1X_1+s_2X_2)}\Vert \Vert\Psi_{f,X_2}(a_Z\exp(t_1X_1+ s_2X_2))\Vert\,ds_2
$$
tends to $0$ when $t_1\to +\infty$. Hence:
$$
\begin{array}{cl}
&\lim_{t_1,t_2\to +\infty} e^{-\Gamma_\cI(t_1X_1+t_2X_2)}\Phi_{f,\lambda}(a_Z\exp(t_1X_1+ t_2X_2))\\
=&\lim_{t_1\to +\infty} e^{-\Gamma_\cI(t_1X_1)}\Phi_{f,\lambda}(a_Z\exp(t_1X_1))\\
=&\Phi_{f,\lambda,\infty}(a_Z,X_1)\,.
\end{array}
$$
Since the first limit on the above equality is symmetrical in $X_1$ and $X_2$, one then deduces that:
$$\Phi_{f,\lambda,\infty}(a_Z,X_1)=\Phi_{f,\lambda,\infty}(a_Z,X_2).$$
\end{proof}

\begin{proof}[Proof of Proposition \ref{prop-main1}] 
If $\lambda\in\cQ_\cI^0$, the hypothesis of (\ref{eq-defPhiinfty}) is satisfied. 
Together with  the preceeding Lemma, it shows the proposition. 
\end{proof}

\subsubsection{Proof of Proposition \ref{prop-main2}}

\begin{lem}\label{lem-Phiinftynul}
For $X_I\in\ga_I^{--}$ such that $\mathrm{Re}\,\lambda(X_I)>\rho_Q(X_I)$, one has:
$$\Phi_{f,\lambda,\infty}(a_Z,X_I)=0,\quad a_Z\in A_Z\,.$$
\end{lem}
\begin{proof}
One has 
$$\Vert e^{-t\Gamma_\cI(X_I)}\Phi_{f,\lambda}(a_Z\exp(tX_I))\Vert \leq e^{-t\mathrm{Re}\,\lambda(X_I)}
\Vert E_\lambda(-tX_I)\Vert\Vert\Phi_f(a_Z\exp(tX_I))\Vert\,.$$
From Lemmas~\ref{lem-Elambda-cont} and~\ref{lem-Phi-Psi-cinfini}(i), one then has 
$$\Vert e^{-t\Gamma_\cI(X_I)}\Phi_{f,\lambda}(a_Z\exp(tX_I))\Vert \leq Ca_Z^{\rho_Q}(1+\Vert\log a_Z\Vert)^N(1+\Vert tX_I\Vert)^{{N+N_\cI}}e^{t(\rho_Q-\mathrm{Re}\,\lambda)(X_I)}\,.$$
The right hand side of the inequality tends to zero as 
$t\to +\infty$. Hence, the Lemma follows from the definition~\eqref{eq-defPhiinfty} of $\Phi_{f,\lambda,\infty}(a_Z,X_I)$. 
\end{proof}

\begin{lem}\label{lem-Phiinftynul2}
Assume $\lambda\in\cQ_\cI^+$ and $X_I\in\ga_I^{--}$ such that $\mathrm{Re}\,\lambda(X_I)>(\rho_Q+\beta_I)(X_I)$. Then, for any $a_Z\in A_Z$, 
$$\Phi_{f,\lambda,\infty}(a_Z,X_I)=0$$
and
$$\Phi_{f,\lambda}(a_Z\exp(tX_I))=-\int_t^\infty {E_\lambda} e^{(t-s)\Gamma_\cI(X_I)} 
\Psi_{f,X_I}(a_Z\exp(sX_I))\,ds\,,\quad {t\in\R}\,.$$
\end{lem}
\begin{proof}
Since $\lambda\in\cQ_\cI^+$, there exists $X_0\in\ga_I^{--}$ such that $\mathrm{Re}\,\lambda(X_0)>\rho_Q(X_0)$. 
Then, from Lemma~\ref{lem-Phiinftynul}, $\Phi_{f,\lambda,\infty}(a_Z,X_0)=0$, and, from 
Lemma~\ref{lem-Phiinfty-constant}, as $\mathrm{Re}\,\lambda(X_0)>\rho_Q(X_0)>(\rho_Q+\beta_I)(X_0)$, 
one has $\Phi_{f,\lambda,\infty}(a_Z,X_I)=\Phi_{f,\lambda,\infty}(a_Z,X_0)$ for any $X_I\in\ga_I^{--}$ such that $\mathrm{Re}\,\lambda(X_I)>(\rho_Q+\beta_I)(X_I)$. 
This proves the first part of the Lemma. 
The second part follows from~\eqref{eq-LuPhiinfty} by change of variables and when we replace $a_Z$ by $a_Z\exp(tX_I)$.
\end{proof}

\begin{cor}\label{cor-lem-Phiinftynul2}
Let $\lambda\in\cQ_\cI^+$ and $X_I\in\ga_I^{--}$ be such that $\Reel\lambda(X_I)\geq (\rho_Q+\beta_I/2)(X_I)$. Then, for $a_Z\in A_Z$ and $t\geq 0$,
$$
\begin{array}{rcl}
\Vert \Phi_{f,\lambda}(a_Z\exp(tX_I))\Vert&\leq & \displaystyle{\int_t^\infty e^{(t-s)(\rho_Q+\beta_I/2)(X_I)}\Vert E_\lambda((t-s)X_I)\Vert\Vert \Psi_{f,X_I}(a_Z\exp(sX_I))\Vert\,ds\,.}
\end{array}
$$
\end{cor}
\begin{proof}
Since $\beta_I(X_I)<0$ and $\Reel\lambda(X_I)\geq (\rho_Q+\beta_I/2)(X_I)$, one has, in particular, $\Reel\lambda(X_I)> (\rho_Q+\beta_I)(X_I)$. Then one can see, from Lemmas~\ref{lem-Phiinftynul2} and~\ref{lem-int-general}, that:
$$\Vert \Phi_{f,\lambda}(a_Z\exp(tX_I))\Vert\leq \int_t^\infty e^{(t-s)\Reel\lambda(X_I)}\Vert E_\lambda((t-s)X_I)\Vert \Vert 
\Psi_{f,X_I}(a_Z\exp(sX_I))\Vert\,ds\,.$$
Our assertion follows, since $\Reel\lambda(X_I)\geq (\rho_Q+\beta_I/2)(X_I)$ implies that $(t-s)\Reel\lambda(X_I)\leq (t-s)(\rho_Q+\beta_I/2)(X_I)$ for 
$s\geq t$.
\end{proof}

\begin{lem}\label{lem-majorPhilambda}
Let $X_I\in\ga_I^{--}$ be such that $\Reel\lambda(X_I)\leq (\rho_Q+\beta_I/2)(X_I)$. Then
$$
\begin{array}{rcl}
\Vert \Phi_{f,\lambda}(a_Z\exp(tX_I))\Vert & \leq & e^{t(\rho_Q+\beta_I/2)(X_I)}\Big(\Vert E_\lambda(tX_I)\Vert \Vert \Phi_f(a_Z)\Vert\\
&&\displaystyle{ +\int_0^\infty e^{-s(\rho_Q+\beta_I/2)(X_I)}\Vert E_\lambda((t-s)X_I)\Vert\Vert \Psi_{f,X_I}(a_Z\exp(sX_I))\Vert\,ds\Big),}\\
&&\hfill {t\geq 0}, a_Z\in A_Z\,.
\end{array}
$$
\end{lem}
\begin{proof}
We use Lemma~\ref{lem-Phi}(ii) and the inequality $(t-s)\mathrm{Re}\,\lambda(X_I)\leq (t-s)(\rho_Q+\beta_I/2)(X_I)$ for $s\leq t$ in order 
to get an analogue of the inequality of the Lemma, where $\int_0^\infty$ is replaced by $\int_0^t$. 
The Lemma follows.
\end{proof}
Like in \cite[after the proof of Lemma~22.8]{hc1}, one sees that one can choose $0<\delta\leq 1/2$ such that: 
\begin{equation}\label{eq-delta}
\mathrm{Re}\,\lambda(X_I)\leq (\rho_Q+\delta\beta_I)(X_I),\quad X_I\in\ga_I^{--},\lambda\in\cQ_\cI^-\,.
\end{equation}
\begin{lem}\label{lem-Qmoins}
Let $\lambda\in\cQ_\cI^-$ and $X_I\in\ga_I^{--}$. Then, for $a_Z\in A_Z$, $t\geq 0$,
$$
\begin{array}{rcl}
\Vert \Phi_{f,\lambda}(a_Z\exp(tX_I))\Vert & \leq & e^{t(\rho_Q+\delta\beta_I)(X_I)}\Big(\Vert E_\lambda(tX_I)\Vert 
\Vert \Phi_f(a_Z)\Vert\\
&&\displaystyle{ +\int_0^\infty e^{-s(\rho_Q+\beta_I/2)(X_I)}\Vert E_\lambda((t-s)X_I)\Vert\Vert \Psi_{f,X_I}(a_Z\exp(sX_I))\Vert\,ds\Big)\,.}
\end{array}
$$
\end{lem}
\begin{proof}
This is proved like Lemma~\ref{lem-majorPhilambda}, using that $\mathrm{Re}\,\lambda(X_I)\leq (\rho_Q+\delta\beta_I)(X_I)$ and 
$0<\delta\leq 1/2$.
\end{proof}
Notice now that, if $\lambda\in\cQ_\cI^0$, it follows from Lemma~\ref{lem-Phiinfty-constant} and the definition of $\beta_I$ 
(cf.~\eqref{eq-defbetaI}) that:
\begin{res-nn}%\label{eq-phiinfty0}
\begin{center}
For $a_Z\in A_Z$, $\Phi_{f,\lambda,\infty}(a_Z,X_I)$ is independent of $X_I\in\ga_I^{--}$\,.
\end{center}
\end{res-nn}
We will denote it by $\Phi_{f,\lambda,\infty}(a_Z)$.
\begin{lem}\label{lem-Qzero}
Assume $\lambda\in\cQ_\cI^0$ and let $X_I\in\ga_I^{--}$. Then one has, for $t\geq 0$ and $a_Z\in A_Z$,
$$
\begin{array}{cl}
&\Vert \Phi_{f,\lambda}(a_Z\exp(tX_I))-\Phi_{f,\lambda,\infty}(a_Z\exp(tX_I))\Vert \\
\leq &\displaystyle{e^{t(\rho_Q+\delta\beta_I)(X_I)}\int_0^\infty 
e^{-s(\rho_Q+\beta_I/2)(X_I)}\Vert E_\lambda((t-s)X_I)\Vert\Vert \Psi_{f,X_I}(a_Z\exp(sX_I))\Vert\,ds\,.}
\end{array}
$$
\end{lem}
\begin{proof}
From~\eqref{eq-LuPhiinfty}, one deduces: 
$$\Phi_{f,\lambda,\infty}(a_Z\exp(tX_I))=\Phi_{f,\lambda}(a_Z\exp(tX_I))+
\int_t^\infty E_\lambda e^{(t-s)\Gamma_\cI(X_I)}\Psi_{f,X_I}(a_Z\exp(sX_I))\,ds\,.$$
The Lemma now follows from the fact that $(t-s)\beta_I(X_I)\geq 0$ whenever $s\geq t$.
\end{proof}
We recall that we have defined: 
$$\Phi_{f,\lambda,\infty}(a_Z):=0,\quad a_Z\in A_Z,\lambda\in\cQ_\cI^+\cup\cQ_\cI^-\,.$$

\begin{proof}[Proof of Proposition \ref{prop-main2}]
If $\lambda\in\cQ_\cI^0\cup\cQ_\cI^-$, our assertion follows from Lemmas~\ref{lem-Qmoins} and~\ref{lem-Qzero}. 
On the other hand, if $\lambda\in\cQ_\cI^+$, we can apply Lemmas~\ref{lem-Phiinftynul2} and~\ref{lem-majorPhilambda}, 
and Corollary~\ref{cor-lem-Phiinftynul2}.
\end{proof}

\section{Definition and properties of the constant term}\label{sect-def-constant}\label{sect-5}

In this section, we define the constant term $f_I$ of a function $f\in \cA_{temp}(Z)$ in terms
of the $\Phi_{f,\lambda,\infty}$ from the previous section. At first, $f_I$ is defined as 
a smooth function on $A_Z$ but then will be extended to a smooth function on $Z_I=G/H_I$. 
 The main difficulty then is to show that the function $f_I\in C^\infty(Z_I)$
is indeed tempered. For that, we need to show certain consistency relations of $f_I$ with respect to the matching map 
${\bf m}: \cW_I\to \cW$, see Proposition \ref{prop-changeorbit}.  The consistency relations
are immediate from our strong results of rapid convergence in Proposition \ref{prop-converg-asbs-1-us}. 
As an application, 
we characterize the functions of the discrete series as those with all constant terms vanishing, see Theorem \ref{theo-ds}.

\bigskip Throughout this section, we fix a subset $I$ of $S$ and a finite codimensional ideal $\cI$ in $\bD_0(Z)$.

\subsection{Definition of the constant term}\label{sect-def-constant1}\label{sect-5}

For $f\in\cA_{temp}(Z:\cI)$ let us define $f_I$ as the function on $A_Z$ by:
\begin{equation}\label{eq:def-fI-aZ}
f_I(a_Z):=\sum_{\lambda\in\cQ_\cI^0}\la\Phi_{f,\lambda,\infty}(a_Z),1\ra,\quad a_Z\in A_Z\,,
\end{equation}
where $\Phi_{f,\lambda,\infty}$ has been defined in \eqref{def111} and \eqref{eq-phiinfty+-}. 
From Lemma~\ref{lem-actaIPhifinfty} and since the eigenvalues of $E_\lambda(\Gamma_\cI(X))$, for any $X\in\ga_I$, are 
contained in $\rho_Q(X)+i\R$ if $\lambda\in\cQ_\cI^0$, one has that:
\begin{res}\label{eq-tildefI-exp-polyn}
\begin{quote}
For any $X\in\ga_I$, the map $t\mapsto e^{-t\rho_Q(X)}f_I(\exp(tX))$ is an exponential polynomial with unitary characters.
\end{quote}
\end{res} 
We will soon extend $f_I$ to a smooth function on $G$ which is right invariant under $H_I$, i.e., $f_I$ descends
to a smooth function on $Z_I$. This will be prepared with a few estimates in the next subsection. 

\subsection{Some estimates}

In this subsection, we establish some estimates analogous to the ones given in \cite[Section~23]{hc1}.

\begin{lem}\label{lem-estimreste}
Let $N\in\N$. There exists a continuous semi-norm 
$q$ on $C_{temp,N}^\infty(Z)$ such that, for all $\lambda\in\cQ_\cI$, $a_Z\in A_Z^-$, $X_I\in\ga_I^{--}$, $t\geq 0$ and 
$f\in\cA_{temp,N}(Z:\cI)$,
$$
\begin{array}{l}
\Vert \Phi_{f,\lambda}(a_Z\exp(tX_I))-\Phi_{f,\lambda,\infty}(a_Z\exp(tX_I))\Vert\\
\\
 \leq (a_Z\exp(tX_I))^{\rho_Q}e^{t\delta\beta_I(X_I)}(1+\Vert\log a_Z\Vert)^N(1+t\Vert X_I\Vert)^{\dim U_\cI}q(f)\,.
\end{array}
$$
\end{lem}
\begin{proof} 
The assertion of the Lemma follows from Proposition~\ref{prop-main2}, 
Lemmas~\ref{lem-Phi-Psi-cinfini} and~\ref{lem-Elambda-cont}, and the fact that $a_Z^{\beta_I}\leq 1$ for $a_Z\in A_Z^-$.
\end{proof}
\begin{lem}\label{lem-actaIPhifinfty}
For all $X\in\ga_I,a_Z\in A_Z, \lambda\in\cQ_\cI$ and $f\in\cA_{temp}(Z:\cI)$
one has 
$$\Phi_{f,\lambda,\infty}(a_Z\exp X)=e^{\Gamma_\cI(X)}\Phi_{f,\lambda,\infty}(a_Z),
\,.$$
\end{lem}
\begin{proof}
According to \eqref{eq-phiinfty+-}, one may assume $\lambda\in\cQ_\cI^0$. From Lemma~\ref{lem-Phi}(ii) applied with $t=1$, one has, for $a_Z\in A_Z$, 
$X\in\ga_I$,
$$e^{-\Gamma_\cI(X)}\Phi_\lambda(a_Z\exp X)=\Phi_\lambda(a_Z)+\int_0^1E_\lambda 
e^{-s\Gamma_\cI(X)}\Psi_X(a_Z\exp(sX))\,ds\,.$$
Let $Y\in\ga_I^{--}$. Replacing $a_Z$ by $a_Z\exp(tY)$ and multiplying by 
$e^{-t\Gamma_\cI(Y)}$, one gets: 
$$
\begin{array}{rcl}
e^{-\Gamma_\cI(X+tY)}\Phi_\lambda(a_Z\exp(X+tY)) &= & e^{-\Gamma_\cI(tY)}\Phi_\lambda(a_Z\exp(tY))\\
&&\\
&&+\displaystyle{\int_0^1 E_\lambda e^{-\Gamma_\cI(sX+tY)}\Psi_X(a_Z\exp(sX+tY))\,ds\,.}
\end{array}
$$
Since $\lambda\in\cQ_\cI^0$, we obtain, from \eqref{eq-major-Elambda} and Lemma~\ref{lem-Phi-Psi-cinfini}(ii), that the integral in this equality tends to $0$ for $t\to \infty$. 
Recalling the definition of $\Phi_{f,\lambda,\infty}$ (cf.~\eqref{eq-defPhiinfty}), one gets 
$$e^{-\Gamma_\cI(X)}\Phi_{f,\lambda,\infty}(a_Z\exp X)=\Phi_{f,\lambda,\infty}(a_Z),\quad X\in\ga_I,a_Z\in A_Z\,.$$
\end{proof}
\begin{lem}\label{lem-estimPhiinfty}
Let $N\in\N$. There exists a continuous semi-norm $p$ on $C_{temp,N}^\infty(Z)$ such that, 
for all $f\in\cA_{temp,N}(Z:\cI)$, $\lambda\in\cQ_\cI^0$,
$$\Vert\Phi_{f,\lambda,\infty}(a_{Z_I})\Vert\leq a_{Z_I}^{\rho_Q}(1+\Vert\log a_{Z_I}\Vert)^{N+\dim U_\cI}p(f),\quad a_{Z_I}\in A_{Z_I}^-\,.$$
\end{lem}
\begin{proof}
We fix $X\in\ga_I^{--}$. Let $a_{Z_I}\in A_{Z_I}^-$. If $t$ is large enough, $a_{Z_I}\exp(tX)\in A_Z^-$. 
More precisely, if $a_{Z_I}=\exp Y$ with $Y\in\ga_{Z_I}^-$, $t$ has to be such that $\alpha(Y+tX)\leq 0$ 
for all $\alpha\in S\setminus I$. 
For this, it is enough that $t\geq \vert\frac{\alpha(Y)}{\alpha(X)}\vert$ for all $\alpha\in S\setminus I$. But $\vert\frac{\alpha(Y)}{\alpha(X)}\vert$ is bounded above by $C\Vert Y\Vert$ for some constant $C>0$. We will take: 
\begin{equation}\label{eq-choixT}
t=C\Vert Y\Vert
\end{equation}
and write $a_{Z_I}=a_Z\exp(-tX)$ with $a_Z=a_{Z_I}\exp(tX)\in A_Z^-$. Since $\lambda\in\cQ_\cI^0$ and $\exp(-tX)=a_Z^{-1}a_{Z_I}$, one has, from Lemma~\ref{lem-actaIPhifinfty}, 
\begin{equation}\label{eq-actaIPhiinfty2}
\Vert\Phi_{f,\lambda,\infty}(a_{Z_I})\Vert=\Vert E_\lambda e^{-t\Gamma_\cI(X)}\Phi_{f,\lambda,\infty}(a_Z)\Vert=a_{Z_I}^{\rho_Q}a_Z^{-\rho_Q}\Vert E_\lambda(-tX) \Phi_{f,\lambda,\infty}(a_Z)\Vert\,.
\end{equation}
We know from Lemma~\ref{lem-Elambda-cont} that $\Vert E_\lambda(-tX)\Vert$ is bounded by a constant times 
$(1+t\Vert X\Vert)^{N_\cI}$, where $N_\cI$ is the dimension of $U_\cI$. 
Using~\eqref{eq-choixT} and as $X$ is fixed, one concludes that there exists $C_1>0$ such that: 
$$\Vert E_\lambda(-tX)\Vert\leq C_1(1+\Vert\log a_{Z_I}\Vert)^{N_\cI}\,.$$
We remark that $\Vert\log a_Z\Vert\leq \Vert\log a_{Z_I}\Vert+\Vert tX\Vert$ is bounded by some constant times 
$\Vert\log a_{Z_I}\Vert$ because $t=C\Vert Y\Vert$ and $\Vert X\Vert$ is fixed. 
Then, using~\eqref{eq-actaIPhiinfty2}, the Lemma follows from Lemma~\ref{lem-estimreste} (applied with $t=0$) and Lemma~\ref{lem-Phi-Psi-cinfini}(i).
\end{proof}

We recall that $\Phi_{f,\lambda,\infty}=0$ for 
$\lambda\in\cQ_\cI^+\cup\cQ_\cI^-$ (cf.~\eqref{eq-phiinfty+-}). We obtain then, from Lemma \ref{lem-estimreste}, that:
\begin{lem}\label{lem-3-17bis}
Let $N\in\N$. There exists a continuous semi-norm $q$ on 
$\cA_{temp,N}(Z)$ such that, for any $f\in\cA_{temp,N}(Z:\cI)$, $a_Z\in A_Z^-$, $X_I\in\ga_I^{--}$ and $t\geq 0$,
$$
\begin{array}{l}
\vert (a_Z\exp(tX_I))^{-\rho_Q}\left[f(a_Z\exp(tX_I))-{f}_I(a_Z\exp(tX_I))\right]\vert \\
\\
\leq e^{t\delta\beta_I(X_I)}(1+\Vert \log a_Z\Vert)^N(1+t\Vert X_I\Vert)^{\dim U_\cI}q(f)\,.
\end{array}$$
\end{lem}
Note that the Lemma implies that:
\begin{equation}\label{eq-prelimit}
\lim_{t\to \infty} {(a_Z\exp(tX_I))^{-\rho_Q}}[f(a_Z\exp(tX_I))-f_I(a_Z\exp(tX_I))]=0,\quad a_Z\in A_Z^-,X_I\in\ga_I^{--}\,.
\end{equation}

\subsection{The constant term as a smooth function on $Z_I$} \label{subsection constant}
Let us first start by the following general remark:
\ber\label{eq-limexppolyn}
If an exponential polynomial function of one variable, $P(t)$, with unitary characters, satisfies: 
$$\lim_{t\to +\infty} P(t)=0\,,$$
then $P\equiv 0$.
\eer
We define some linear forms $\eta$ and $\eta_I$ on $\cA_{temp}(Z:\cI)$ by:
$$
\begin{array}{rcl}
\la\eta,f\ra & = & f(z_0),\\
\la\eta_I,f\ra & = & f_I(z_{0,I}),\quad f\in\cA_{temp}(Z:\cI)\,.
\end{array}
$$
Let us remark that $\eta$ is a continuous linear form on $\cA_{temp,N}(Z:\cI)$ for any $N\in\N$.

Note that we obtain from the definition \eqref{def mc} that:
$$m_{\eta_I,f}(a_Z)=f_I(a_Z),\quad a_Z\in A_Z\,.$$

\begin{lem}\label{lem-etaI}
Let $N\in\N$. The linear form $\eta_I$ is the unique linear form on {$\cA_{temp,N}(Z:\cI)$} such that:
\begin{enumerate}
\item[(i)] For any $f\in\cA_{temp,N}(Z:\cI)$ and $X_I\in\ga_I^{--}$, 
$$\lim_{t\to \infty} e^{-t\rho_Q(X_I)}[m_{\eta,f}(\exp(tX_I))-m_{\eta_I,f}(\exp(tX_I))]=0\,.$$
\item[(ii)] For any $f\in\cA_{temp,N}(Z:\cI)$ and $X\in\ga_I$, $t\mapsto e^{-t\rho_Q(X)}m_{\eta_I,f}(\exp(tX))$ is an exponential polynomial with unitary characters.
\item[(iii)] Moreover, $\eta_I$ is continuous on $\cA_{temp,N}(Z:\cI)$ and $H_I$-invariant.
\end{enumerate}
\end{lem}
\begin{proof}
The assertion~(i) is \eqref{eq-prelimit} and (ii) is \eqref{eq-tildefI-exp-polyn}.

\par 
To prove the unicity of such an $\eta_I$ satisfying~(i) and~(ii), we use~\eqref{eq-limexppolyn}. If $\eta_I'$ is another 
linear form satisfying~(i) and~(ii), then, for any {$f\in \cA_{temp,N}(Z:\cI)$},
$$m_{\eta_I,f}(\exp(tX_I))-m_{\eta_I',f}(\exp(tX_I))= 0,\quad  X_I\in\ga_I^{--}, t\in\R\,.$$
This equality applied to $t=0$ implies that $\eta_I=\eta_I'$.

\par
Let us show the continuity of $\eta_I$. By taking $a_Z=1$ in the inequality of Lemma~\ref{lem-3-17bis}, one gets: 
$$\vert f(z_0)-f_I(z_{0,I})\vert\leq C q(f),\textrm{ i.e., }\vert\la\eta,f\ra-\la\eta_I,f\ra\vert \leq Cq(f)\,.$$
Moreover $\eta$ is a continuous map on $\cA_{temp,N}(Z:\cI)$. 
This implies that $\eta_I$ is continuous on $\cA_{temp,N}(Z:\cI)$. \par
It remains to get that $\eta_I$ is $H_I$-invariant. From~\eqref{eq-prelimit}, for any $X_I\in\ga_I^{--}$, 
$$
\lim_{t\to \infty}e^{-t\rho_Q(X_I)}[f(\exp(tX_I))-f_I(\exp(tX_I))]=0\,.
$$
One applies this to $L_Yf$, $Y\in\gh_I$ and gets:
\begin{equation}\label{eq-prelimit1}
\lim_{t\to \infty}e^{-t\rho_Q(X_I)}\left[(L_Yf)(\exp(tX_I))-(L_Yf)_I(\exp(tX_I))\right]=0\,.
\end{equation}
On the other hand, from Lemma~\ref{lem-h-I}, one has:
\begin{equation}\label{eq-prelimit2}
\lim_{t\to \infty}e^{-t\rho_Q(X_I)}(L_Yf)(\exp(tX_I))=0\,.
\end{equation}
Hence, one gets, from~\eqref{eq-prelimit1} and~\eqref{eq-prelimit2}, that:
$$\lim_{t\to \infty}e^{-t\rho_Q(X_I)}(L_Yf)_I(\exp(tX_I))=0\,.$$
But $t\mapsto e^{-t\rho_Q(X_I)}(L_Yf)_I(\exp(tX_I))$ is an exponential polynomial with unitary characters 
(cf.~\eqref{eq-tildefI-exp-polyn}). Hence, from~\eqref{eq-limexppolyn}, it is identically equal to $0$. This implies that:
$$\eta_I(L_Yf)=0\,.$$
Then $\eta_I$ is continuous and $\gh_I$-invariant, and hence $H_I$-invariant. This completes the proof of (iii).
\end{proof}
Let $N\in\N$ be fixed. For $f\in\cA_{temp,N}(Z:\cI)$, since $\eta_I$ is continuous, we obtain with 
\begin{equation}\label{eq-fI}
g\mapsto f_I(g):=m_{\eta_I,f}(g),\quad g\in G\,,
\end{equation}
a smooth extension of $f_I$ previously defined on $A_Z$. Note that, as $\eta_I$ is $H_I$-invariant, $f_I$ defines a smooth function on $Z_I$ denoted by the same symbol. 
%n $\cA_{temp,N}(Z:\cI)$ (cf.~Lemma~\ref{lem-etaI}), $f_I$ is well-defined.
Further, note that the assignment $f\mapsto f_I$ is $G$-equivariant, in symbols:
\begin{equation}\label{eq-Lg-fI-commute}
(L_gf)_I=L_gf_I,\quad g\in G\,.
\end{equation}

\begin{rem}\label{rem-exppolyfI}
As a consequence of Lemma \ref{lem-etaI} and the above equivariance relation \eqref{eq-Lg-fI-commute}, for all $g\in G$ and $X_I\in\ga_I^{--}$, 
$$
\lim_{t\to \infty}e^{-t\rho_Q(X_I)}[f(g\exp(tX_I))-f_I(g\exp(tX_I))]=0\,.
$$
and $X\mapsto e^{-\rho_Q(X)}f_I(g\exp X)$ is an exponential polynomial on $\ga_I$ with unitary characters.
Moreover, $f_I$ is the unique smooth function on $G$  with these two properties.
\end{rem}

\subsection{Consistency relations for the constant term}

Let $w_I\in\cW_I$ and $w\in\cW$. 
Set $H_{I,w_I}=w_IH_Iw_I^{-1}$ and $H_w=wHw^{-1}$. 
Consider the real spherical spaces $Z_w=G/H_w$ and $Z_{I,w_I}=G/H_{I,w_I}$, and put $z_0^w=H_w\in Z_w$ and 
$z_{0,I}^{w_I}=H_{I,w_I}\in Z_{I,w_I}=G/H_{I,w_I}$. 
Then (cf.~\cite[Corollary~3.8]{kks}) $Q$ is $Z_w$-adapted to $P $ and $A_{Z_w}=A_Z$ with $A_{Z_w}^-=A_Z^-$.

\par 
For $f\in C^\infty(Z)$, let us define $f^w$ by:
$$
f^w(g\cdot z_0^w)=f(gw\cdot z_0),\quad g\in G\,.
$$
In the same way, one defines $\phi^{w_I}$ for $\phi\in C^\infty(Z_I)$. Then $f^w\in C^\infty(Z_w)$ and $\phi^{w_I}\in C^\infty(Z_{I,w_I})$.

\begin{prop}[Consistency relations for the constant term]\label{prop-changeorbit}
Let $w_I\in\cW_I$ and $w={\bf m}(w)\in\cW$.  Let $f\in\cA_{temp,N}(Z:\cI)$. Then $f^w\in\cA_{temp,N}(Z_w:\cI)$ and
$$(f_I)^{w_I}(a_Z)=(f^w)_I(a_Z),\quad a_Z\in A_Z\,.$$
\end{prop}
Here, $f^w\in\cA_{temp}(Z_w:\cI)$, $(f^w)_I\in C^\infty(Z_{w,I})$, $f_I\in C^\infty(Z_I)$, 
$f_I^{w_I}\in C^\infty(Z_{I,w_I})$, and, from \cite[Proposition~3.2(5) and {Corollary~3.8}]{kks}, one has:
$$
\begin{array}{c}
A_{Z_{w,I}}=A_{Z_w}=A_Z\,,\\
A_{Z_{I,w_I}}=A_{Z_I}=A_Z\,.
\end{array}
$$
Hence, both sides of the equality are well-defined on $A_Z$.

\par
The proof of Proposition~\ref{prop-changeorbit} is prepared by a simple technical lemma. 
Recall the elements $a_s=\exp(s X_I)$ for $X_I\in \ga_I^{--}$. 

\begin{lem}\label{lem-majorconvLgs}
Let $(g_s')$ be a family in $G$ which converges rapidly to $g\in G$. 
Let $f\in\cA_{temp,N}(Z)$. 
Then there exist $C>0$ and $\varepsilon>0$ such that: 
$$%\begin{equation}\label{eq-convL_gs'}
{\vert(L_{(g_s')^{-1}}f)(a_s)-(L_{g^{-1}}f)(a_s)\vert\leq Ca_s^{\rho_Q}e^{-\varepsilon s},\quad s\geq s_0}\,.
$$%\end{equation}
\end{lem}
\begin{proof}
As $(g_s')$ converges rapidly to $g$ when $s$ tends to $+\infty$, there exists 
$s_0'$, $C'$, $\varepsilon'$ strictly positive and $(X_s)\subset\gg$ such that, for all $s\geq s_0'$, 
\begin{equation}\label{eq-convrapidgs}
g_s'=g\exp{X_s}\textrm{ and }\Vert X_s\Vert\leq C'e^{-\varepsilon' s}\,.
\end{equation}
As $L_{g^{-1}}$ preserves $\cA_{temp,N}(Z)$, one is reduced to prove, for all $f\in\cA_{temp,N}(Z)$, that there exist 
$C,\varepsilon,s_0>0$ such that:
$$%\begin{equation}
\vert f(\exp(X_s)a_s)-f(a_s)\vert\leq Ca_s^{\rho_Q}e^{-\varepsilon s}\,.
$$%\end{equation}
But, by the mean value theorem, if $a\in {A_Z} $ and $X\in\gg$,
$$\vert f(\exp(X)a)-f(a)\vert\leq \sup_{t\in [0,1]}\vert (L_{-X}f)(\exp(tX)a)\vert\Vert X\Vert\,.$$
From~\eqref{eq-convrapidgs}, one then sees that it is enough to prove that, if $\Vert X\Vert$ is bounded by a constant $C''>0$, there exists a constant $C'''>0$ such that: 
\begin{equation}\label{eq-majorsup}
\sup_{t\in [0,1]}\vert (L_{-X}f)(\exp(tX)a)\vert\leq C'''a^{\rho_Q}(1+\Vert\log a\Vert)^N,\quad a\in A_Z^{-}\,.
\end{equation}
Decomposing $-X$ in a basis $(X_i)$ of $\gg$ and using the continuity of the endomorphisms $L_{X_i}$ of 
$\cA_{temp,N}(Z)$, one sees that there exists a continuous semi-norm such that:
$$\vert (L_{-X}f)(a)\vert\leq a^{\rho_Q}(1+\Vert\log a\Vert)^Nq(f),\quad a\in A_Z^-\,.$$
But $f\mapsto \sup_{\Vert X\Vert\leq C''} q(L_{\exp(-tX)}f)$ is a continuous semi-norm on $\cA_{temp,N}(Z)$. 
Hence, as $L_{-X}$ and $L_{\exp(-tX)}$ commute, \eqref{eq-majorsup} follows. This achieves to prove the Lemma.
\end{proof}
\begin{proof}[Proof of Proposition~\ref{prop-changeorbit}]
If $a\in A $, one has:
$$
[(L_af)^w]_I=[L_a(f^w)]_I\textrm{ as }(L_af)^w=L_af^w\, .$$
Hence, it is enough to prove the identity of the Proposition for $a_Z=z_0$. 
Then, using~\eqref{eq-limexppolyn} and Remark~\ref{rem-exppolyfI}, it is enough to prove that
$s\mapsto (f_I)^{w_I}(a_s)$ is an exponential polynomial with unitary characters satisfying: 
\begin{equation}\label{eq-cond(i')}
\lim_{s\to +\infty}a_s^{-\rho_Q}[f^w(a_s)-(f_I)^{w_I}(a_s)]=0\,.
\end{equation}
But, from~\eqref{eq-decompwIan}, 
$$\tilde{a}_sw\cdot z_0=(\tilde{a}_s\tilde{b}_s^{-1}m_s^{-1}u_s^{-1})(u_sm_s\tilde{b}_sw)\cdot z_0=g_sw_I\tilde{a}_s\cdot z_0\,, $$ 
for $s\geq s_0$, where $g_s=\tilde{a}_s\tilde{b}_s^{-1}m_s^{-1}u_s^{-1}$. Then one has: 
$$f^w(a_s)=L_{w_I^{-1}g_s^{-1}}f(a_s)\,.$$
On the other hand, from \cite[Lemma~{3.5}]{kks} for $Z=Z_I$, as $A_{Z_I,E}=A_I$ (cf.~loc.cit., equation~(3.13)), one has:
\begin{equation}\label{eq-commut-a-wI}
\tilde{a}_sw_I\cdot z_{0,I} = w_I \tilde{a}_s\cdot z_{0,I}\,,
\end{equation}
which implies that: 
\begin{equation}\label{eq-LmwIfI}
(L_{w_I^{-1}}f_I)(\tilde{a}_s\cdot z_{0,I})=(f_I)^{w_I}(a_s)\,.
\end{equation}
Now, according to Proposition \ref{prop-converg-asbs-1-us} -- this is the key ingredient! --, 
the sequence $(g_sw_I)$ converges rapidly to $w_I$. Hence, we can apply Lemma~\ref{lem-majorconvLgs} with 
$g_s'=g_sw_I$ and find $C',\varepsilon',s_0'>0$ such that: 
\begin{equation}\label{eq-majorconvwIgs}
a_s^{-\rho_Q}\vert (L_{w_I^{-1}g_s^{-1}}f)(a_s)-(L_{w_I^{-1}}f)(a_s)\vert\leq C'e^{-\varepsilon' s},\quad s\geq s_0'\,.
\end{equation}
Using Lemma~\ref{lem-3-17bis}, one has, for some $C'',\varepsilon'>0$, 
$$a_s^{-\rho_Q}\vert (L_{w_I^{-1}}f)(a_s)-(L_{w_I^{-1}}f_I)(a_s)\vert\leq C''e^{-\varepsilon' s},
\quad s\geq s_0'\,.$$
Hence, from~\eqref{eq-LmwIfI} and~\eqref{eq-majorconvwIgs}, one deduces~\eqref{eq-cond(i')}. It remains to prove that: 
$$s\mapsto (f_I)^{w_I}({a_s})=f_I({a_s}w_I\cdot z_{0,I})$$
is an exponential polynomial with unitary characters. But, from \cite[Lemma~{3.5}]{kks} applied to $Z_I$,
$$(f_I)^{w_I}({a_s})=f_I(w_I{a_s})\,.$$
Hence, our claim follows from~\eqref{eq-commut-a-wI}. This achieves the proof of the Proposition.
\end{proof} 

\subsection{Constant term approximation}
Now we turn to the main Theorem of this section.

\begin{theo}[Constant term approximation]\label{theo-constterm}
Let $I\subset S$ and $\cI$ be a finite codimensional ideal of $\bD_0(Z)$.
\begin{enumerate}
\item[(i)] For all $N\in\N$, the map $f\mapsto f_I$ is a continuous linear map from $\cA_{temp,N}(Z:\cI)$ to $\cA_{temp,N+\dim U_\cI}(Z_I:\mu_I(\cI))$. 
\item[(ii)] Let $N\in\N$ and $\cC_I$ be a compact subset of $\ga_I^{--}$. For 
$w_I\in\cW_I$ let $w={\bf m}(w_I)\in\cW$. Then there exist $\varepsilon>0$ and a continuous semi-norm $p$ on $C_{temp,N}^\infty(Z)$ such that, for all $f\in\cA_{temp,N}(Z:\cI)$,  
$$
\begin{array}{l}
 \vert (a_{Z}\exp(tX))^{-\rho_Q}\left( f(g a_{Z}\exp(tX_I)w\cdot z_0)
 -f_I(g a_{Z}\exp(tX_I)w_I\cdot z_{0,I})\right)\vert\\
\\
\leq  e^{-\varepsilon t}(1+\Vert\log a_{Z}\Vert)^Np(f),\qquad a_{Z}\in A_{Z}^-, X_I\in \cC_I, g\in\Omega,w_I \in \cW_I, t\geq 0\,.
\end{array}
$$
\end{enumerate}
\end{theo}
%%%%%%%%%%
\begin{proof} We first show (i). In view of~\eqref{eq-equiv-norms},  it suffices  to prove that, for any $w_I\in\cW_I$, there exists a continuous semi-norm $p$ on $\cA_{temp,N}(Z:\cI)$ such that:
$$\sup_{g\in\Omega,a_{Z_I}\in A_{Z_I}^-}\vert a_{Z_I}^{-\rho_Q}(1+\Vert \log a_{Z_I}\Vert)^{-(N+\dim U_\cI)}f_I(g a_{Z_I}w_I)\vert\leq p(f),
\quad f\in\cA_{temp,N}(Z:\cI)\,.$$
For $w_I=1$, one  has $w={\bf m}(w_I)=1$. 
Our claim then follows from \eqref{eq-Lg-fI-commute}, \eqref{eq:def-fI-aZ} and Lemma~\ref{lem-estimPhiinfty} and the continuity of the left regular representation of $G$ on $C_{temp,N}^\infty(Z)$ (see \eqref{eq-EnGstable}).

\par
For general $w_I$, one uses Proposition~\ref{prop-changeorbit} to get 
$f_I(a_{Z_I}w_I)=(f^w)_I(a_{Z_I})$ and the above inequality for $H^w$ instead of $H$. This shows (i).

\par 
Using Proposition~\ref{prop-changeorbit}, one is reduced to prove~(ii) for 
$w_I=w=1$, by changing $H$ 
into $H_w$. Moreover, from \eqref{eq-Lg-fI-commute} and \eqref{eq-EnGstable}, one is reduced to show (ii) for $g=1$. 
In that case, (ii) follows from \eqref{eq-Phif-wI} (applied with $u=1$), \eqref{eq:def-fI-aZ} and 
Lemma~\ref{lem-estimreste} by choosing $\varepsilon >0$ and $p$ in the following way.

\par
Set $N_\cI:=\dim U_\cI$. Let us consider the continuous function $\varphi:(X_I,t)\mapsto e^{t\delta\beta_I(X_I)/2}(1+t\Vert X_I\Vert)^{N_\cI}$ on $\ga_I\times\R$, which is smooth on the second variable and positive on $\ga_I\times\R_{\geq 0}$. Recall that $\delta\beta_I(X_I)<0$ for any $X_I\in\ga_I^{--}$. Since $\cC_I$ is a compact subset of $\ga_I^{--}$, by continuity, $C:=\max_{X_I\in\cC_I}\varphi(X_I,-2N_\cI/\delta\beta_I(X_I)-1/\Vert X_I\Vert)$ and $\varepsilon:=-\delta/2[\max_{X_I\in\cC_I}(\beta_I(X_I))]$ exist and $\varepsilon>0$.
Moreover, $\varphi$ has values $\leq C$ on $\cC_I\times \R$. 
Hence $C>0$ and, by Lemma~\ref{lem-estimreste}, $\varepsilon$ yields the inequality in (ii) for $t\geq 0$ by setting $p:=C q$. 
\end{proof}

\begin{rem}[Statement for $H_0$ connected]
Theorem \ref{theo-constterm} remains valid for $H$ replaced by $H_0$:
exchange the expression $f_I(g a_{Z}\exp(tX)w_I\cdot z_{0,I})$ by $f_I(g  m_{w_I}^{-1}a_{Z}\exp(tX)w_I\cdot z_{0,I})$
for certain $m_{w_I}\in M$, see Remark \ref{rem-cc}(b).  Likewise, this will hold for Theorem \ref{theo-uniform} below,
which generalizes Theorem \ref{theo-constterm}. 
\end{rem}

\begin{rem}[Reformulation of Theorem \ref{theo-constterm} in terms of representation theory]
Let $(\pi,V^\infty)$ be an $SAF$-representation of $G$, for example $V^\infty=\cA_{temp}(Z: \cI)$ (see Proposition \ref{prop-SAF-rep}).
Then Theorem \ref{theo-constterm}(i) gives rise to a linear map 
$$(V^{-\infty})_{temp}^H\longrightarrow (V^{-\infty})_{temp}^{H_I}, \ \ \eta\mapsto \eta_I$$
and correspondingly, for every $v\in V^\infty$, an approximation of the matrix coefficient $g\mapsto f(g\cdot z_0)=m_{\eta,v}(g)$ by $g\mapsto f_I(g\cdot z_{0,I})=m_{\eta_I,v}(g)$ as in Theorem \ref{theo-constterm}(ii).  In this language, the consistency relations
from Proposition \ref{prop-changeorbit} then translate into 
$$ (w \cdot \eta)_I  = w_I\cdot  \eta_I \qquad w_I\in \cW_I, w ={\bf m}(w_I)\,,$$
where, for an element $\xi\in V^{-\infty}$ and $g\in G$, we use the notation $g\cdot \xi=\xi(g^{-1}\cdot)$ for the dual action.

\end{rem}

\subsection{Application to the relative discrete series for ${Z}$}

Let $\chi$ be a normalized unitary character of $A_{Z,E}=\exp(\ga_{Z,E})$, i.e., $d\chi_{\vert\ga_{Z,E}}=\rho_Q\vert_{\ga_{Z,E}}$.

\par 
We recall that, if $a\in A_{Z,E}$ and $w\in\cW$, 
\begin{equation}\label{eq awH}
\tilde{a}wH=waH
\end{equation} 
(cf.~\cite[Lemma~3.5]{kks}).

\par 
As $\widetilde A_{Z,E}$ normalizes $H$, there is a right action 
$(a,z)\mapsto z\cdot a$ of $A_{Z,E}$ on $Z$. 
Let $C^\infty(Z,\chi)$ be the space of $C^\infty$ functions on $Z$ such that:
$$f(z\cdot a)=\chi(a)f(z),\quad a\in A_{Z,E},z\in Z$$
and observe that 
\begin{equation}\label{eq welldef}
|a^{-\rho_Q} f(z\cdot a)|=|f(z)|,\quad a\in A_{Z,E},z\in Z\,,
\end{equation}
as $\chi$ was requested to be normalized unitary.

If $f\in C^\infty(Z,\chi)$, $u\in \cU(\gg)$ and $N\in\N$, then \eqref{eq awH} and \eqref{eq welldef} allow us to define
$$r_{N,u}(f)=\sup_{g\in\Omega,a\in A_Z^-/A_{Z,E},w\in\cW}\vert a^{-\rho_Q}(1+\Vert\log a\Vert)^N (L_uf)(g aw\cdot z_0)\vert\,,$$
with $\Vert\cdot\Vert$ refering to the quotient norm on $\ga_Z/\ga_{Z,E}$. Moreover, we set
$$\cC(Z,\chi)=\{ f\in C^\infty(Z,\chi)\mid r_{N,u}(f)<\infty,N\in\N,u\in \cU(\gg)\}\,.$$

Since $\widetilde A_{Z,E}$ normalizes $H$, we obtain a closed subgroup $\widehat{H}:=H\widetilde A_{Z,E}$ (not depending on the section 
${\bf s}$) and a real spherical space $\widehat{Z}=G/\widehat{H}$. 
We extend $\chi$ trivially to $H$ and then define a character of $\widehat{H}$ still denoted $\chi$. 
Let us define $L^2(\widehat{Z};\chi)$ as in \cite[Section~8.1]{kks}.

\par 
Let $w\in\cW$. We recall that $H_w=wHw^{-1}$ and $Z_w=G/H_w$. Let $f$ be in $C^\infty(Z,\chi)$. 
Recall that $f_w$ defined by $f_w(g)=f(gwH)$, $g\in G$, is right $H_w$-invariant and defines an element of $C^\infty(Z_w)$ and even of $C^\infty(Z_w,\chi)$ by using the relation \eqref{eq awH}. 
This element will still be denoted $f_w$. Moreover, by ``transport of structure'', if $f$ is $Z$-tempered, $f_w$ is $Z_w$-tempered.

\par
Let $\eta$ be a $Z$-tempered $H$-fixed linear form on $V^\infty$. Let $w\in\cW$. Then $\ga_{Z_w}=\ga_Z$ and $w\cdot \eta$ is $H_w$-invariant and $Z_w$-tempered by ``transport of structure''. By \cite[Corollary 3.8]{kks}, $Q$ is $Z_w$-adapted to $P$. Moreover, the set of spherical roots for $Z_w$ is equal to $S$ (see \cite[equation~(3.2), definition of $S$ in the beginning of Section~3.2 and Lemma~3.7]{kks}). Hence, one can define $(w\cdot \eta)_I$, $w\in\cW$.

\begin{theo}\label{theo-ds}
Let $(\pi,V^\infty)$ be an $SAF$-representation of $G$, with $V$ its associated Harish-Chandra module, and $\eta$ be a $Z$-tempered continuous linear form on 
$V^\infty$  
which transforms under a unitary character $\chi$ of $A_{Z,E}$. 
Then the following assertions are equivalent:
\begin{enumerate}
\item[(i)] For all $v\in V$, $m_{\eta,v}\in L^2(\widehat{Z};\chi)$.
\item[(ii)] For all proper subsets $I$ of $S$ and $w\in\cW$, $(w\cdot \eta)_I=0$.
\item[(iii)] For all $v\in V^\infty$, $m_{\eta,v}\in\cC(Z,\chi)$.
\end{enumerate}
\end{theo}
\begin{proof}
Let us assume~(i).  We may assume that $V\subset L^2(\widehat{Z};\chi)$ via the embedding 
$v \mapsto m_{\eta, v}$. Let $\cH$ be the unitary completion of $V$ in $L^2(\widehat{Z};\chi)$. Then 
$V^\infty = \cH^\infty$ and this implies that the statement is independent of the particular 
choice of the maximal compact subgroup $K$. We choose now $K$ as in \cite[Sect. 6]{kks}. 
In particular we obtain an open neighborhood $U_A$ of $1$ in $A$ such that all 
$m_{\eta, v}$, $v\in V$  admit absolutely convergent power series expansions \cite[(6.2)]{kks} 

\begin{equation} \label{power expand}  m_{\eta, v} (aw) =  \sum_{j=1}^l \sum_{\alpha\in \N_0[S]} a^{\Lambda_j +\alpha} q_{\alpha, j, w}(\log a) \qquad (a \in U_A\cdot A_Z, w\in \cW)\end{equation} 
where $\Lambda_j \in \ga_{Z,\C}^*$, $1\leq j \leq l$,  and $q_{\alpha, j,w}$ are polynomials on $\ga_Z$. 
Let $S=\{\sigma_1,\dots,\sigma_s\}$ and $\omega_1,\dots,\omega_s\in\ga_Z$ be such that:
$$
\begin{array}{ll}
\sigma_i(\omega_j)=\delta_{i,j},& i,j=1,\dots,s\\
\omega_i\perp\ga_{Z,E},& i=1,\dots,s\,.
\end{array}
$$
Here we use the scalar product on $\ga_Z$ defined before~\eqref{eq-defXab}. 
According to \cite[Theorem~8.5]{kks} the condition that all $m_{\eta, v} \in L^2(\widehat{Z};\chi)$ implies that 
\begin{equation}\label{eq-LambdaVeta}
\operatorname{Re}(\Lambda_k-\rho_Q)(\omega_j)>0,\quad j=1,\dots,s, \ k=1,\dots, l\, .
\end{equation}
Now \eqref{power expand} in combination with \eqref{eq-LambdaVeta} imply 
the existence of an $\varepsilon>0$ such that for all $v\in V$ there exists a constant $C_v>0$ such that 
\begin{equation}  \label{rapid bound} |m_{\eta, v} (aw)|\leq C_v a^{(1+\varepsilon) \rho_Q} \qquad (a \in A_Z^-)\, .\end{equation} 

On the other hand by the constant term approximation (Lemma \ref{lem-etaI}(i) applied to $Z_w$ and $\eta$ replaced by 
$w\cdot \eta$) 
we obtain that 
$$\lim_{t\to \infty} e^{t\rho_Q(X_I)} \left (m_{w\cdot \eta, v} (\exp(tX_I)) - m_{(w\cdot \eta)_I, v}(\exp (tX_I) )\right )=0$$
for all $X\in \ga_I^{--} \subset \ga_Z^-$. Moreover $t\mapsto m_{(w\cdot \eta)_I, v} (\exp(tX)) $ is the unique 
exponential polynomial with unitary characters having this approximation property by Lemma \ref{lem-etaI}(ii). 
 Hence \eqref{rapid bound} implies that this exponential polynomial is zero. In particular, 
 $\la (w\cdot \eta)_I,v\ra=m_{(w\cdot\eta)_I,v}(1)=0$ for all $v\in V$, and hence, by density of $V$ in $V^\infty$, $(w\cdot \eta)_I=0$, that is
(ii). 

\par
Let us assume that~(ii) holds. Let $\cI$ be an ideal of $\cZ(\gg)$ which annihilates $V$ or $V^\infty$. It is of finite codimension. 
Since $\eta$ is $Z$-tempered, there exists $N_0\in\N$ such that, for all $v\in V^\infty$, $m_{\eta,v}\in\cA_{temp,N_0}(Z:\cI)$ (cf.~\eqref{eq-def-tempform}). Let $v\in V^\infty$ and set $f=m_{\eta,v}$. Then one can apply 
Theorem~\ref{theo-constterm} to $Z_w$ and $f_w$ for $w_I$ equal to $1$: Let $I\varsubsetneq S$, $\cC$ be a compact subset of $\ga_I^{--}$, 
$\Omega_1$ be a compact subset 
of $G$ and $u\in \cU(\gg)$. Then there exists a continuous semi-norm $p$ on $C_{temp,N_0}^\infty(Z)$, $\varepsilon>0$ such that:
\begin{equation}\label{eq-modgrowth-f}
\begin{array}{l}
\vert(a_Z\exp(tX))^{-\rho_Q}(L_uf)(g a_Z\exp(tX)w\cdot z_0)\vert\\
\\
\leq e^{-\varepsilon t}(1+\Vert\log a_Z\Vert)^{N_0}p(f),\quad a_Z\in A_Z^-/A_{Z,E},X\in\cC,g\in\Omega_1,w\in\cW,t\geq 0\,.
\end{array}
\end{equation}
Note, as $\eta$ transforms under a unitary character for $A_{Z,E}$ (see \eqref{eq awH}), the left hand side 
in \eqref{eq-modgrowth-f}  depends only on  $a_Z\exp(tX)$ mod $A_{Z,E}$.
\par Let ${\bf S}_1$ be the unit sphere in $\ga_Z/\ga_{Z,E}$ and 
let $\widetilde X_0\in {\bf S}_1\cap\ga_Z^-/\ga_{Z,E}$. 
Let $\Omega_0$ be an open neighborhood of $\widetilde X_0$ in $S_1\cap\ga_Z^-/\ga_{Z,E}$ such that, for all $\widetilde X\in\Omega_0$, 
$\alpha(\widetilde X)\leq \alpha(\widetilde X_0)/2$, $\alpha\in S$. Let $I$ be the set of $\alpha\in S$ such that $\alpha(\widetilde X_0)=0$. 
One has $I\neq S$ as we may assume that $S\neq \emptyset$.  Let $X_0\in \ga_I$ be a lift of $\widetilde X_0$ and note 
that $X_0\in\ga_I^{--}$. 
Let $Y\in\Omega_0$ and $t\geq 0$. Then $t(Y-X_0/2)\in\ga_Z^-/ \ga_{Z,E}$ and 
$\exp(tY)=\exp t(Y-\widetilde X_0/2)\exp(t\widetilde X_0/2)\in A_Z/ A_{Z,E} $. Using~\eqref{eq-modgrowth-f} for $X=X_0/2$ and $a_Z=\exp t(Y-\widetilde X_0/2)\in A_Z^-/ A_{Z.E} $ one gets:  For any $N\in \N$ there exists a $c>0$, depending on 
$N, \varepsilon, f, \Omega_0$ and $\Omega_1$ 
such that 
$$
\begin{array}{l}
\vert(\exp(tY))^{-\rho_Q}(L_uf)(g \exp(tY)w\cdot z_0)\vert\\
\\
\leq e^{-\varepsilon t}(1+t\Vert Y-\widetilde X_0/2\Vert)^{N_0}p(f)\leq c(1+t)^{-N},\quad Y\in\Omega_0,g\in\Omega_1,w\in\cW,t\geq 0\,,
\end{array}
$$
One deduces easily from this that, for any $u\in\cU(\gg)$ and $N\in\N$:
$$\sup_{g\in\Omega_1,w\in\cW,a\in\exp(\R^+\Omega_0)}a^{-\rho_Q}(1+\Vert\log a\Vert)^{N}
\vert(L_uf)(g a w\cdot z_0)\vert<+\infty\,.$$
Using a finite covering of the compact set ${\bf S}_1\cap\ga_Z^-/\ga_{Z,E}$, one deduces from this that $f\in\cC(Z,\chi)$. 
This achieves to prove that~(ii) implies~(iii).

\par
To prove that~(iii) implies~(i), one proceeds as in the proof that~(ii) implies~(i) in \cite[Theorem~8.5]{kks}.
\end{proof}

\section{Transitivity of the constant term}
Recall that $\cA_{temp}(Z)$ consists of $\cZ(\gg)$-finite functions. In particular, for each $f\in \cA_{temp}(Z)$ there exists 
a co-finite ideal $\cJ\subset \cZ(\gg)$ such that $f \in \cA_{temp}(Z: \cJ)$. Hence constant terms 
$f_I$ are defined for all $f\in \cA_{temp}(Z)$. 
\begin{prop}[Transitivity of the constant term]\label{trans prop}
Let $I\subset J$ be two subsets of $S$. Then, if $f\in\cA_{temp}(Z)$,
$$f_I=(f_J)_{I}\,.$$
\end{prop}
\begin{proof}
%%%%%
By $G$-equivariance of the maps: 
$$
\begin{array}{rcl}
%\begin{minipage}{}
\begin{array}{rcl}
\cA_{temp}(Z) & \to & \cA_{temp}(Z_I)\\
f & \mapsto & f_I 
\end{array}
%\end{minipage}
& \textrm{and} &
%\begin{minipage}{1}
\begin{array}{rcl}
\cA_{temp}(Z_J) & \to & \cA_{temp}(Z_I)\\
f & \mapsto & f_I 
\end{array},
%\end{minipage}
\end{array}
$$
it is enough to show that, if $f\in\cA_{temp}(Z)$, $f_I(z_{0,I})=(f_J)_I(z_{0,I})$. Recall that $\ga_{Z_J}=\ga_Z$ and
$$\ga_I^{--}=\{X\in\ga_I:\, \alpha(X)<0,\alpha\in S\setminus I\},\quad 
\ga_{I,J}^{--}=\{X \in\ga_I:\, \alpha(X)<0,\alpha\in J\setminus I\}\,.$$ 
As $\ga_I=\{X\in\ga_Z:\, \alpha(X)=0,\alpha\in I\}$ and $\ga_J=\{X\in\ga_Z:\, \alpha(X)=0,\alpha\in J\}$, one has: 
$$
\ga_J\subset\ga_I,\quad \ga_I^{--}\subset\ga_Z^-,\quad \ga_{I,J}^{--}\subset\ga_Z^-\,.
$$
One remarks that $\ga_{I}^{--}\subset\ga_{I,J}^{--}$. Let $X\in\ga_J^{--}$ and $Y\in\ga_{I}^{--}$. Then $X+Y\in\ga_{I}^{--}$.

\par
Using Theorem~\ref{theo-constterm}(ii) applied successively to $(Z,I,f,X+Y,1)$, $(Z,J,f,X,\exp(tY))$ and $(Z_J, I,f_J, Y,\exp(tX))$ instead of $(Z,I,f,X,a_Z)$, one gets that there exist $C>0$ and $\varepsilon>0$ such that, for all $t\geq 0$,
$$
\begin{array}{l}
\alpha_t\vert f(\exp(t(X+Y)))-f_I(\exp(t(X+Y)))\vert\leq Ce^{-\varepsilon t}\,,\\
\alpha_t\vert f(\exp(tY)\exp(tX))-f_J(\exp(tY)\exp(tX))\vert\leq Ce^{-\varepsilon t}\,,\\
\alpha_t\vert f_J(\exp(tX)\exp(tY))-(f_J)_I(\exp(tX)\exp(tY))\vert\leq Ce^{-\varepsilon t}\,,
\end{array}
$$
where $\alpha_t=e^{-t\rho_Q(X+Y)}$. Hence, one concludes from the three inequalities above that:
$$
\alpha_t\vert f_I(\exp(t(X+Y)))-(f_J)_I(\exp(t(X+Y)))\vert\leq 3Ce^{-\varepsilon t},
\quad t\geq 0\,.
$$
Hence, $\alpha_t [f_I(\exp(t(X+Y)))-(f_J)_I(\exp(t(X+Y)))]$ tends to zero when $t$ goes to $+\infty$. 
But, each term of this difference is an exponential polynomial in $t$ with unitary characters. 
Hence, according to~\eqref{eq-limexppolyn}, the difference of the two occurring exponential polynomials is identically zero. 
It implies, taking $t=0$, that $f_I(z_{0,I})=(f_J)_I(z_{0,I})$.
\end{proof}

\subsection{Application: Tempered embedding theorem} 

From the constant term approximation in Th. \ref{theo-constterm}, the consistency relations (Prop. \ref{prop-changeorbit}) and the  transitivity of the constant term (Prop. \ref{trans prop}) one can quite easily derive an extension 
of the tempered embedding theorem \cite[Th. 9.11]{kks} to all real spherical spaces. The details are 
carried out in \cite[Th.~11.12]{planch-sph} and we record for later reference: 

\begin{theo}[Tempered embedding theorem] \label{temp embed}Let $V$ be an irreducible Harish-Chandra module contained in $C^\infty_{temp}(Z)$. Then there exists
$I\subset S$, $w\in \cW$  and a unitary character $\chi$ of $A_I$ such that there is an $(\gg, K)$-embedding
$$V\hookrightarrow L^2(\hat {(Z_w)_I}, \chi)\, .$$
\end{theo}

\section{Uniform estimates}\label{section uniform}

The goal of this section is to obtain a parameter independent version of the main result 
Theorem~\ref{theo-constterm}: the bounds become uniform if we restrict ourselves 
to ideals $\cI$ of $\bD_0(Z)$ of codimension one. 
The crucial ingredient is a recent result that infinitesimal characters 
of tempered representations have integral real parts (see \cite{kkos} and summarized in Lemma \ref{lem-lattice} below).

\par Recall the Cartan subalgebra $\gj=\ga\oplus\gt\subset \gg$ with real form $\gj_\R=\ga\oplus i\gt\subset\gj_\C$, associated Weyl group $W_\gj$ and half sum of roots $\rho_\gj$.   Note that $\rho_\gj\big|_{\ga_H}= \rho\big|_{\ga_H}=\rho_Q\big|_{\ga_H}=0$ as 
$Z$ was requested to be unimodular. In particular $\rho_\gj\big|_{\ga}$ factors through 
$\ga_Z$ and coincides with $\rho_Q$.
\par
If $\Lambda\in\gj_\C^*/W_\gj$, let $\chi_\Lambda$ be the character of $\cZ(\gg)$ corresponding to 
$\Lambda$ via the Harish-Chandra isomorphism $\gamma: \cZ(\gg)\to S(\gj)^{W_\gj} $. 
More precisely,
$$\chi_\Lambda(u)=(\gamma(u))(\Lambda),\quad u\in \cZ(\gg)\,.$$
Further, we set $\cJ_\Lambda:=\ker  \chi_\Lambda$. 
We also recall the untwisted Harish-Chandra homomorphism $\gamma_0: \cZ(\gg)\to S(\gj)$
and set $\cJ_{\Lambda, 0}:= \gamma_0(\cJ_\Lambda)$. 

\par 
According to Chevalley's theorem, 
$S(\gj)$ is a free module of finite rank over $S(\gj)^{W_\gj}\simeq \gamma_0(\cZ(\gg))$. Hence, 
we obtain a subspace $U_0\subset S(\gj)$ such that the natural map: 
$$ \gamma_0(\cZ(\gg)) \otimes U_0\to S(\gj), \quad v\otimes  u \mapsto vu$$
is an isomorphism. Thus, for any $\Lambda\in \gj_\C^*/W_\gj$, we obtain with $\gamma_0(\cZ(\gg)) = \cJ_{\Lambda,0} +\C 1$
that $S(\gj)/ S(\gj) \cJ_{\Lambda,0}\simeq U_0$ as vector spaces. The natural representation of $S(\gj)$ on 
$S(\gj)/ S(\gj) \cJ_{\Lambda,0}\simeq U_0$ gives then rise to 
a $S(\gj)$-representation:
$$\sigma_\Lambda: S(\gj) \to \End(U_0)\, .$$
For $\Lambda\in\gj_\C^*/W_\gj$, let us fix a representative $\lambda\in\gj_\C^*$ such that $\Lambda=W_\gj\cdot \lambda$.

\begin{lem} \label{lem: j-spec}
The following assertions hold: 
\begin{enumerate}[(i)]
\item The representation $(\sigma_\Lambda, U_0)$ is polynomial in $\Lambda$, i.e., for all $v\in S(\gj)$, the assignment 
$\Lambda\mapsto \sigma_\Lambda(v)$ is polynomial. 
 \item One has $\Spec (\sigma_\Lambda)=\rho_\gj+W_\gj\cdot \lambda$. 
\end{enumerate}
\end{lem}

\begin{proof}  
We prove both assertions together. Consider the auxiliary $S(\gj)$-module 
$S(\gj)/ S(\gj) \cJ_\Lambda$ and call the corresponding representation 
of $S(\gj)$ by $\sigma_\Lambda'$.  We have $\Spec(\sigma_\Lambda')= W_\gj\cdot \lambda$. 
Recall the complement $U_0\subset S(\gj)$ and let $U_1= U_0(\cdot +\rho_\gj)\subset S(\gj)$ obtained from 
$\rho_\gj$-shift. We model $\sigma_\Lambda'$ on $U_1$ 
and claim that $v\mapsto \sigma_\Lambda'(v)$ is polynomial in $\Lambda$. 
It suffices to verify the assertion for $v$ of the form 
$v=\gamma(z)u$ with $z\in \cZ(\gg)$ and $u \in U_1$. 
Now 
$$v=u(\gamma(z) - \chi_\Lambda(z))  + \chi_\Lambda(z)u $$
with the first sum in 
the ideal $S(\gj)\cJ_\Lambda$. The claim follows. 
It remains to relate the representation $\sigma_\Lambda$ to $\sigma_\Lambda'$, which is given by 
$\sigma_\Lambda(v)  = \sigma_\Lambda'( v(\cdot +\rho_\gj))$ via the algebra
automorphism $S(\gj)\to S(\gj), \ v\mapsto v(\cdot +\rho_\gj)$ obtained by the $\rho_\gj$-shift upon identification $S(\gj)\simeq\C[\gj_\C^*]$.
\end{proof}

Recall that there is a surjective algebra morphism $p:\cZ(\gg)\to\bD_0(Z)$. Given a codimension one ideal $\cI$ in $\bD_0(Z)$, its preimage $\cJ=p^{-1}(\cI)$ is of codimension one in $\cZ(\gg)$, hence, of the form $\cJ_\Lambda$, for some $\Lambda\in\gj_\C^*/W_\gj$.

\par 
Denote by $\X\subset\gj_\C^*/W_\gj$ the set of $\Lambda$'s obtained that way. For $\Lambda\in\X$, we set $\cI_\Lambda:=p(\cJ_\Lambda)$. 

\par Next we wish to describe the set $\X$ more closely. 
Since $\bD_0(Z)$ is a finitely generated $\C$-algebra without nilpotent elements, its maximal spectrum $\operatorname{specmax}(\bD_0(Z))$ is an affine variety
and naturally identifies with $\X$. The surjective algebra morphism $p:\cZ(\gg)\to \bD_0(Z)$ gives rise to the closed embedding:  
$$p_*: \X=\operatorname{specmax}(\bD_0(Z)) \hookrightarrow \gj_\C^*/W_\gj=\operatorname{specmax}(\cZ(\gg))\,.$$
We recall our choice of $\gt$ and $\gt_H$ before Lemma \ref{lem: fin gen}.

\begin{lem} \label{lemma X} The affine subvariety $\X\subset \gj_\C^*/W_\gj$ is given by 
\begin{equation} \label{decr X} \X=\{\Lambda\in \gj_\C^*/ W_\gj\mid  \exists \mu \in \Lambda=W_\gj\cdot \lambda\ \text{such that}\  (\rho_\gj+\mu)\big|_{\ga_H +\gt_H}=0\}\end{equation} 
\end{lem}

To prepare the proof of this Lemma, we need to develop a little bit of general theory which is 
used later on as well.  
\par We recall that Lemma   \ref{lem: fin gen}  implies that 
$\bD(Z_I)$ is a finitely generated $\C$-algebra without nilpotent elements and thus corresponds to an affine variety 
${\mathbb Y}_I=\operatorname{specmax}(\bD(Z_I))$. 
It follows from Lemma \ref{lem-defmu_I2} that the
algebra morphism $\mu_I: \bD_0(Z) \to \bD(Z_I)$ is injective, hence 
$\mu_{I,*}: {\mathbb Y}_I\to \X$ is a dominant morphism of affine algebraic varieties. 
Moreover, since $\bD(Z_I)$ is a module of finite type 
over $\bD_0(Z)$, it follows in addition 
that $\mu_{I,*}$   is a finite surjective morphism with uniformly bounded finite fibers (by the going up property
in ring theory, see \cite[Theorem 5.10]{AM} or \cite[Proposition 3.2.4]{hpk}).

Define $\gamma_{00}: \cZ(\gg)\to S(\gj)/ S(\gj)(\ga_H + \gt_H)$, obtained 
from the composition of $\gamma_0$ and the projection $S(\gj)\to S(\gj)/S(\gj)(\ga_H + \gt_H)$. 
We recall from \eqref{algebra j} the injective algebra morphism
$$j_0: \gZ(Z_\emptyset)\to S(\gj)/ S(\gj)(\ga_H + \gt_H)\, .$$
Now $j_0$ composed with the natural inclusion $\bD(Z_\emptyset) \hookrightarrow \gZ(Z_\emptyset)$ gives rise 
the injective morphism 
$$\iota_\emptyset: \bD(Z_\emptyset)\to S(\gj)/ S(\gj)(\ga_H + \gt_H)\, .$$
Next, we recall that $\bD(Z_I)$ is naturally a subalgebra of $\bD(Z_\emptyset)$
via the monomorphism $\D(Z_I)\hookrightarrow \D(Z_\emptyset)$ of Lemma \ref{lem-defmu_I2} applied to $Z=Z_I$.
Composing this injection with $\iota_\emptyset$ we obtain a monomorphism 
$$\iota_I: \bD(Z_I)\to S(\gj)/ S(\gj)(\ga_H + \gt_H)\, .$$
With \eqref{diagram1}, we thus arrive at the following commutative diagram of finite module extensions
\begin{equation} \label{IJ-diagram}
    \xymatrix{
    S(\gj) \ar@{>}[r] & S(\gj)/ S(\gj)(\ga_H + \gt_H)  \ar@{=}[r]&S(\gj)/ S(\gj)(\ga_H + \gt_H)\\
    \cZ(\gg) \ar@{>}[u]^{\gamma_0} \ar@{>}[r]^{p} &  \bD_0(Z) \ar@{>}[u]^{\iota_0} \ar@{>}[r]^{\mu_I}&\bD(Z_I)\ar@{>}[u]^{\iota_I}\, ,}
        \end{equation}
with the middle vertical arrow $\iota_0$ uniquely determined by the injectivity of $\mu_I$. In particular, 
$\iota_0$ is injective. 
On the level of affine varieties, this corresponds to the commutative diagram 
\begin{equation} \label{diagram2}
    \xymatrix{
    \gj_\C^*\ar@{>}[d]_{\gamma_{0,*}}& \ar@{>}[l] (\ga_H +\gt_H)^\perp
     \ar@{=}[r]\ar@{>}[d]_{\iota_{0,*}} &(\ga_H +\gt_H)^\perp \ar@{>}[d]_{\iota_{I,*}}\\
      \gj_\C^*/W_\gj &\ar@{>}[l]_{p_*}\X&\ar@{>}[l]_{\mu_{I,*}}{\mathbb Y}_I\,,}
        \end{equation}
where $(\ga_H +\gt_H)^\perp\subset \gj_\C^*$.  Since all vertical arrows in \eqref{IJ-diagram} are injective 
and represent finite module extensions, it follows that all vertical arrows
in \eqref{diagram2} are surjective  (by application of 
the going down property  as above).

\begin{proof}[Proof of Lemma \ref{lemma X}] Immediate from the surjectivity of the vertical maps
in the commutative diagram \eqref{diagram2}.\end{proof}

Recall the decomposition  
$$
\bD(Z_I)=U_\Lambda\oplus\bD(Z_I)\cI_\Lambda'\,,
$$ 
from \eqref{eq: bdZ3}, with $U_\Lambda:=U_{\cI_\Lambda}$ and $\cI_\Lambda':=\mu_I(\cI_\Lambda)$. 

\par
Recall that $U_\Lambda\subset U$ (where $U$ is the finite dimensional subspace of $\bD(Z_I)$ independent of $\Lambda$ satisfying \eqref{eq: defU}) and thus $n:= \max_{\Lambda\in\X}  \dim U_\Lambda\leq \dim U<\infty$. 
For every $0\leq j\leq n$, we now set 
$$\X^j:=\{ \Lambda\in \X\mid \dim U_\Lambda=j\}$$
and get $\X=\coprod_{j=0}^n \X^j$. Now, for every $\Lambda\in\X^j$, the set 
$$\X_\Lambda:=\{x\in\X \mid U_\Lambda\oplus\bD(Z_I)\cI_x'=\bD(Z_I)\}$$
is a subset of $\X^j$. 

\begin{lem}\label{lem-zar1} 
The following assertions hold:
\begin{enumerate}[(i)]
\item For any $0\leq j\leq n$, the set $\bigcup_{k\leq j}\X^k$ is Zariski-open in $\X$. In particular, $\X^j$ is locally closed in $\X$.
\item For $\Lambda\in \X^j$, the set $\X_\Lambda$ is Zariski-open in $\X^j$. 
\end{enumerate}
\end{lem}

\begin{proof} 
Note that $\bD_0(Z)=\cO(\X)$ is the coordinate ring of the affine variety $\X$.  For any $x\in \X$,
we denote by $\gm_x\subset \cO(\X)$ the corresponding maximal ideal. 
Since $\cO(\Y_I)=\bD(Z_I)$ is a finite 
module of $\cO(\X)$, we find a finite dimensional subspace $U_x\subset \cO(\Y_I)$ such that $\cO(\Y_I)= U_x \oplus \cO(\Y_I)\gm_x$. 
The Nakayama lemma implies that there exists an $f\in \cO(\X)$ with $f(x)\neq 0$ such that 
$\cO(\Y_I)_f = \cO(\X)_f U_x$. 
In particular, we have, for all $z\in \X$ with $f(z)\neq 0$, that 
$\cO(\Y_I)=  U_x + \cO(\Y_I)\gm_z$. This implies that: 
$$\X \to \N_0, \ \ x\mapsto \dim \cO(\Y_I)/\cO(\Y_I)\gm_x$$
is upper semi-continuous and, in particular, for any $1\leq j\leq n$, we have that 
$\bigcup_{k\leq j } \X^k$ is Zariski-open in $\X$ and (i) follows.
\par 
For (ii), we just saw that, for $\Lambda\in\X^j$, we have, for $z\in\X_\Lambda$, that there exists $f\in\cO(\X)$ such that $f(z)\neq 0$ and $\{y\in\X^j\mid f(y)\neq 0\}\subset \X_\Lambda$. Hence, $\X_\Lambda$ is Zariski-open in $\X^j$.
\end{proof}

As quasi-affine varieties are quasi-compact for the Zariski topology, 
it follows that there exists finitely many $\Lambda\in\X$, say $\Lambda_1,\ldots,\Lambda_s$, such that:
$$\X=\bigcup_{j=1}^s\X_{\Lambda_j}\,.$$
For any $1\leq j\leq s$, we define a fixed finite dimensional vector space $U_j:=U_{\Lambda_j}$ as above. 
 This gives us a direct sum decomposition
 \begin{equation}\label{eq-decompDZI-Lambda}
\bD(Z_I)=U_j\oplus\bD(Z_I)\cI_\Lambda',\quad \Lambda\in \X_{\Lambda_j}\,,
\end{equation}
and, upon the identification $U_j\simeq\bD(Z_I)/\cI_\Lambda'$, a representation 
$$\rho_\Lambda:\bD(Z_I)\to\End(U_j)\,.$$

\begin{lem}\label{lem-rhoLambda}
The following assertions hold:
\begin{enumerate}[(i)]
\item Fix $1\leq j\leq s$. For any $v\in\bD(Z_I)$, the map 
$$
\X_{\Lambda_j}\to \End(U_j),\quad \Lambda\mapsto \rho_\Lambda(v)
$$
is regular, i.e., locally the restriction to $\X_{\Lambda_j}$ of a rational function on $\X$. In particular, there exists 
an open covering $\X=\bigcup_{j=1}^s  \X_j $ with $\X_j\subset \X_{\Lambda_j}$ such that, 
for all $v\in \bD(Z_I)$, there exists a constant $C_v>0$ such that 
\begin{equation}\label{eq-nothing}
\|\rho_\Lambda(v)\|\leq C_v ( 1+\|\Lambda\|)^N \qquad (\Lambda \in \X_j)\,,\end{equation}
for some $N\in \N$ independent of $v$. Here, $\|\cdot\|$ on the left hand side of \eqref{eq-nothing} refers to the operator norm of $\End(U_j)$. 
\item With $S(\ga_I)\subset\bD(Z_I)$, one has: 
$$\Spec_{\ga_I}(\rho_\Lambda)\subset (\rho_Q+W_\gj\cdot \Lambda)\vert_{\ga_I}\,.$$
\end{enumerate}
\end{lem}

\begin{proof}
Recall the terminology we introduced in the proof of Lemma \ref{lem-zar1}. Since the assertion is local, we may assume that $\X=\X_{\Lambda_j}$, for some $j$ and $U=U_j$, is such that $\cO(\Y_I)=\cO(\X)U=U\oplus \cO(\Y_I)\gm_x$ for all $x\in\X$. This decomposition defines a projection $p_x:\cO(\Y_I)\to U$ for any $x\in\X$. Moreover, note that the natural map
$$\cO(\X)\otimes U\to \cO(\Y_I),\quad g\otimes u\mapsto gu$$
is an isomorphism. Accordingly, every $f\in\cO(\Y_I)$ can be expressed uniquely as $f=\sum_i g_i\otimes u_i$ for a fixed basis $(u_i)$ of $U$. Then 
$$p_x(f)=\sum_ig_i(x)u_i$$
is regular in $x\in\X$. This proves the first assertion of (i) and the second assertion in (i) is an immediate consequence thereof. 
\par (ii) Geometrically, it might happen that the fiber of the morphism $\mu_{I,*}: {\mathbb Y}_I\to \X$ over $\Lambda\in\X$ is 
not reduced, i.e., $\bD(Z_I)\cI_\Lambda'$ is not a radical ideal in $\bD(Z_I)$. However, the set of $\Lambda$'s, with reduced fibers, is open dense in $\X$. 
In view of the continuity showed in (i), it suffices to show that $\Spec_{\ga_I}(\rho_\Lambda)\subset(\rho_\gj+W\cdot\lambda)\big|_{\ga_I}$ for generic $\Lambda$, i.e., $\Lambda$ reduced.

\par Next, we recall the diagram \eqref{diagram2} with all vertical arrows surjective 
and all fibers being finite. 
Now,  as the fiber $\mu_{I,*}^{-1}(\Lambda)$ was assumed to be reduced, it has 
$\dim(U_\Lambda)$ elements as the  corresponding affine algebra to this finite variety is just the $\ga_I$-module $\bD(Z_I)/ \cI_\Lambda'$. 
In particular, $\mu_{I,*}^{-1}(\Lambda)$  consists of the $\ga_I$-weights of $U_\Lambda\simeq \bD(Z_I)/ \cI_\Lambda'$.

\par From \eqref{diagram2}, we obtain the the fiber diagram: 
\begin{equation}\label{eq-diagr4} 
    \xymatrix{
    \rho_{\gj}+W_\gj\cdot\lambda\ar@{>}[d]& \ar@{_{(}->}[l] (\rho_{\gj}+W_\gj\cdot\lambda)\cap(\ga_H +\gt_H)^\perp
     \ar@{>}[d] &\iota_{I,*}^{-1}(\mu_{I,*}^{-1}(\Lambda)) \ar@{_{(}->}[l]\ar@{->>}[d]\\
     W_\gj\cdot\lambda= \Lambda &\ar@{=}[l]\Lambda &\ar@{>>}[l]\mu_{I,*}^{-1}(\Lambda)\, .}
        \end{equation}
Hence (ii) follows from the $\ga_I$-equivariance of $\iota_{I,*}$. 
\end{proof}
 
The section $\bs$ we use in the sequel is the one where we identify $\ga_Z$ with the subspace $\ga_H^{\perp_{\ga_L}}\subset \ga_L$, the orthogonal being taken with respect to the form $\kappa$ introduced at the beginning of Subsection~\ref{sect-1.2}. 
Let $\sJ(\C)\subset \sG(\C)$ be the Cartan subgroup
with Lie algebra $\gj_\C$ and $\cL:=\Hom(\sJ(\C), \C^*)$ be its character group. In the sequel, we 
identify $\cL$ with a lattice in $\gj^*$. We call a subspace $V\subset \gj^*$ {\em rational} 
provided that $V= \R ( V \cap \cL)$.  Likewise, we call a discrete subgroup 
$\Gamma\subset (\gj^*, +)$ rational if $\Gamma = \Gamma \cap \Q \cL$. Using the 
dual lattice $\cL^\vee\subset \gj$, we obtain a notion of rationality for subspaces and 
discrete subgroups of $\gj$ as well.

\par
Finally, we may and will request that $\kappa\big|_{\gj\times \gj}$ is rational, 
i.e., with respect to a basis of $\gj$ which lies in $\cL^\vee$, its matrix entries are rational.

\begin{lem} \label{lem-rational} 
The following subspaces of $\gj$ are all rational: $\ga_H, \ga_Z$ and $\ga_I$ for $I\subset S$.
\end{lem}
\begin{proof} 
The subspace $\ga_H$ is rational as it corresponds to the Lie algebra 
of the subtorus $(\sA_L \cap \sH)_0 \subset \sJ$.  Since the form $\kappa\big|_{\gj\times \gj}$ is rational, 
we obtain that $\ga_Z\subset \ga\subset \gj$ is rational as well. Finally, we obtain from \eqref{eq-cM} that 
$S\subset \Q \cL$ and this gives us the rationality of $\ga_I$ for any $I\subset S$.  
\end{proof}  
 
We recall that $\cQ_\Lambda$ denotes the set of $\ga_I$-weights of $\rho_\Lambda$ and (cf.~Lemma \ref{lem-rhoLambda}(ii))
 \begin{equation} \label{eq-Q-spectrum} 
 \cQ_{\Lambda}\subset\{(\rho_Q+w\Lambda)\big|_{\ga_I} \mid w\in W_\gj\}\,,
 \end{equation}
 where we identify $\ga_I$ as a subspace of $\ga$ as above. 
 For $\lambda\in \cQ_\Lambda$, we recall the projectors $E_\lambda: U_\Lambda^*\to  U_{\Lambda, \lambda}^*$ to the generalized 
 common eigenspace along the supplementary generalized eigenspaces. 
 
\par In the sequel, we abbreviate and write $\cA_{temp}(Z:\Lambda)$ 
instead of $\cA_{temp}(Z:\cJ_\Lambda)$.

\bigskip The key to obtain uniform estimates for the constant term approximation is at the core related 
to polynomial bounds for the truncating spectral projections $E_\lambda$.  

\begin{prop}\label{prop-projection-bound}  
Let $1\leq j\leq s$. There exist constants $C, N >0$ such that, for all 
$\Lambda \in \X_j$ with $\cA_{temp}(Z:\Lambda)\neq \{0\}$, one has 
$$ \| E_\lambda \|\leq C ( 1 + \|\Lambda\|)^N, \qquad \lambda\in \cQ_\Lambda \,, $$
with $\Vert E_\lambda\Vert$ the operator norm on the fixed finite dimensional vector space $\End(U_j)$.
\end{prop} 

The proof of the Proposition is preceded by two lemmas:

\begin{lem} \label{matrix lemma}
Let $0<\nu\leq 1$, $N\in \N$ and $\mathsf{A}\in \operatorname{Mat}_N(\C)$ with $\operatorname{Spec} (\mathsf{A})=\{ \lambda_1,\ldots, \lambda_r\}$
such that $\operatorname{Re} \lambda_1\leq \ldots\leq \operatorname{Re} \lambda_r$.  For every $1\leq j\leq r$, let $V_j\subset \C^n$ be the generalized 
eigenspace of $\mathsf{A}$ associated to the eigenvalue $\lambda_j$.  For every $1\leq k\leq r$, we let $E_k=\bigoplus_{j=1}^k  V_j$
and $\mathsf{P}_k : \C^N \to E_k$ be the projection along $\bigoplus_{j=k+1}^r V_j$.  Suppose, for some 
$1\leq k\leq r-1$, that $\operatorname{Re} \lambda_{k+1}- \operatorname{Re} \lambda_k\geq \nu$.  
Then there exists a constant $C=C(\nu, N)>0$ such that 
$$\|\mathsf{P}_k\|\leq C  (1 +\|\mathsf{A}\|)^N\, .$$
\end{lem}

\begin{proof} \cite[Lemma 6.4]{geoc}.\end{proof}

\begin{lem} \label{lem-lattice} 
There exists a $W_\gj$-stable rational lattice $\Xi_Z$ in the vector space $\gj^*$ such that 
\begin{equation} \label{eq-tc} 
\operatorname{Re} \Lambda \in \Xi_Z 
\end{equation}
for all $\Lambda \in \gj_\C^*$ with $\cA_{temp}(Z:\Lambda)\neq \{0\}$. 
\end{lem} 

\begin{proof} 
Let $0\neq f\in \cA_{temp}(Z:\Lambda)$ be a $K$-finite element which generates an irreducible 
Harish-Chandra module, say $V$. According to Theorem \ref{temp embed} $V$ embeds into a twisted discrete series of some 
$(Z_w)_I$.  Now, we apply \cite[Theorem~{1.1}]{kkos} and obtain a $W_\gj$-invariant lattice $\Xi_{(Z_w)_I}$, 
called $\Lambda_{(Z_w)_I}$ in \cite{kkos}, with property \eqref{eq-tc}. The lattice is indeed rational by 
\cite[Theorem~8.3]{kkos} combined with \cite[Lemma 3.4]{kkos}.  The asserted lattice is then obtained by taking the 
rational lattice generated by the rational lattices $\Xi_{{Z_w}_I}$, i.e. 
$$\Xi_Z=\la v\in \Xi_{{Z,w}_I}: I\subset S, w\in \cW\ra_{\Z-\mathrm{mod}} $$
 \end{proof} 

\begin{proof} [Proof of Proposition~\ref{prop-projection-bound}] 
According to  Lemma \ref{lem-rational}, $\ga_I$ is a rational subspace of $\ga\subset\gj$. 
Now, we keep in mind the 
following general fact:  if $U\subset \gj$ is a rational subspace and $\Xi\subset \gj^*$ is a rational lattice, then 
$\Xi\big|_{U}$ is a rational lattice in $U^*$.
In particular, it follows that  $\Xi_{Z,I}:=\Xi_Z\big|_{\ga_I}$ is a rational lattice 
in $\ga_I^*$.  Next, observe that Lemma \ref{lem-lattice} combined with (\ref{eq-Q-spectrum}) implies that 
$\operatorname{Re} \cQ_\Lambda \subset \rho_Q\big|_{\ga_I}+ \Xi_{Z,I}$ for all tempered infinitesimal 
characters $\Lambda$. 
Denote by $\Xi_{Z,I}^\vee\subset \ga_I$ the dual lattice of $\rho_Q\big|_{\ga_I}+ \Xi_{Z,I}$.  Since $\ga_I^-$ is a rational cone, 
we find elements 
$X_1, \ldots, X_k$  of $\ga_I^-\cap \Xi_{Z,I}^\vee$ such that: 
$$ \ga_I^-=\sum_{j=1}^k \R_{\geq 0} X_j\, .$$

We identify $U_j$ with $\C^N$ and define matrices $\mathsf{A}_i:=\Gamma_\Lambda(X_i)=\,^{t}\!\rho_\Lambda(X_i)$. Let $\lambda\in\cQ_\Lambda$. Write $E_{\lambda, i}$ 
for the spectral projection to the generalized eigenspace of $\mathsf{A}_i$ with eigenvalue 
$\lambda(X_i)$. 
Since the matrices $\mathsf{A}_i$ commute with each other and the $X_i$ span $\ga_I$, we obtain that:
\begin{equation} \label{eq-Eproj}
E_\lambda= E_{\lambda, 1} \circ  \ldots \circ E_{\lambda, k}\, .
\end{equation}
Hence, we are reduced to prove a polynomial bound for each $E_{\lambda, i}$. 
As 
$$\Spec({\mathsf A}_i) \subset(\rho_Q+W_\gj \cdot\Lambda) (X_i)\,,$$
we get $\operatorname{Re} \operatorname{Spec} (\mathsf{A_i}) \subset \Z$. Hence, 
we can apply Lemma \ref{matrix lemma} to the matrices $\mathsf{A}_i$, 
with $\nu=1$, and obtain 
$\| E_{\lambda, i}\| \leq C ( 1 + \| \mathsf{A_i}\|)^N$.  Now, we recall 
from \eqref{eq-nothing} that 
$$\|\Gamma_\Lambda(X)\|\leq  C \|X\| ( 1+ \|\Lambda\|)^N\,, $$
after possible enlargement of $C$ and $N$. 
This gives the asserted norm bound for $\| E_{\Lambda,i}\|$ and then for $E_\lambda$ via  \eqref{eq-Eproj}. 
\end{proof}

For $\lambda\in\cQ_\Lambda$, we recall the notation 
$$E_\lambda(X)=e^{-\lambda(X)}E_\lambda(e^{\Gamma_\Lambda(X)}),\quad X\in\ga_I\,,$$
and recall, from Lemma~\ref{lem-Phi}(ii), the starting identity: 

$$\begin{array}{rcl}
\displaystyle{\Phi_{f,\lambda}(a_Z\exp(tX_I))} & = & \displaystyle{e^{t\Gamma_\Lambda(X_I)}
\Phi_{f,\lambda}(a_Z)}\\
&&+\displaystyle{\int_0^t E_\lambda e^{(t-s)\Gamma_\Lambda(X_I)}\Psi_{f,X_I}(a_Z\exp(sX_I))\,ds\,,}\\
&&\hfill a_Z\in A_{Z},X_I\in\ga_I,t\in\R\,.
\end{array}
$$
\begin{lem}\label{lem-estimPsifX}
Let $N\in\N$. There exist a continuous semi-norm 
$q$ on $C_{temp,N}^\infty(Z)$  and $m\in\N$ such that, for all $\Lambda\in\gj_\C^*/W_{\gj}$ and $f\in\cA_{temp,N}(Z:\Lambda)$,
$$
\begin{array}{l}
\Vert \Phi_{f,\lambda}(a_Z\exp(tX_I))-\Phi_{f,\lambda,\infty}(a_Z\exp(tX_I))\Vert\\
\\
 \leq (a_Z\exp(tX_I))^{\rho_Q}e^{t\delta\beta_I(X_I)}(1+\Vert\log a_Z\Vert)^N(1+t\Vert X_I\Vert)^{\dim U}(1+\Vert\Lambda\Vert)^mq(f),\\
 \\
 \hfill \lambda\in\cQ_\cI, a_Z\in A_Z^-, X_I\in\ga_I^{--}, t\geq 0\,.
\end{array}
$$
\end{lem}
\begin{proof} The statement is a uniform version of Lemma~\ref{lem-estimreste} which rested on 
Lemma \ref{lem-Phi-Psi-cinfini}, Lemma  \ref{lem-Elambda-cont} and Proposition \ref{prop-main2}. 
Now Proposition  \ref{prop-projection-bound} makes the bound in Lemma \ref{lem-Phi-Psi-cinfini}
for the norm of the spectral projections $E_\lambda$ uniform at the cost of an additional 
polynomial factor, a power of $(1+\|\Lambda\|)$. This takes care of the uniform estimates 
for the $E_\lambda$ in  Proposition \ref{prop-main2}. It remains to obtain uniform estimates 
for $\Phi_f$ and $\Psi_{f,X}$ in Lemma \ref{lem-Phi-Psi-cinfini}. This Lemma  was obtained 
for a fixed ideal $\cI$ and fixed complement $U_\cI$. Now,  by Lemma \ref{lem-rhoLambda} we can in fact get by 
with finitely many choices of complements $U_1, \ldots, U_s$ at the cost of another polynomial factor of a power of 
$(1+\|\Lambda\|)$. As a result the estimate in Lemma \ref{lem-estimreste} 
becomes uniform at the cost of a polynomial factor of the type $(1+\|\Lambda\|)^m$ which is recored at the 
right hand side of the asserted estimate. 
\end{proof}

Having said all that, it is now clear that all bounds from Sections \ref{sect-4} and \ref{sect-5} become uniform at the 
cost of an extra polynomial factor in $\|\Lambda\|$.   Polynomial behavior in $\|\Lambda\|$ can be subsumed 
in raising the Sobolev order of  the corresponding semi-norms.  In more detail, 
if $p$ is a continuous semi-norm on an $SF$-module $V^\infty$ 
with infinitesimal character $\Lambda$, then we claim that there exists $C>0, k\in \N$ independent 
of $p$, $V$ and $\Lambda$ such that 
\begin{equation} \label{claim eins} (1 +\|\Lambda\|) p(v) \leq  C p_k (v) \qquad (v\in V)\, ,\end{equation} 
where $p_k$ denotes the $k$-th Sobolev
norm of $p$ with respect to a fixed basis of $\gg$. For that we first note  that 
\begin{equation}\label{claim zwei}  |\chi_\Lambda(z)| p(v) = p(zv) \leq C_z p_{\deg z}(v),  \qquad v\in V^\infty\,,\end{equation} 
for all $z\in \cZ(\gg)$ and a constant $C_z>0$. 
Now for any $X\in \gj_\C$ we define a $W_\gj$-invariant polynomial function 
on $\gj_\C^*$ by $f_X(\Lambda):= \prod_{w\in W_\gj} \Lambda(w\cdot X)$. Note that for any $\Lambda\neq 0$
we find an $X\in \gj_\C$ such that $f_X(\Lambda)\neq 0$, i.e. choose $X\in \gj_\C\setminus \bigcup_{w\in W_\gj} \ker \Lambda \circ w$.  By the homogeneity of the $f_X$ and the compactness of the unit sphere 
in $\gj_\C^*$ we thus find finitely many $X_1, \ldots, X_m$ such that 
\begin{equation} \label{claim drei} \max_{1\leq j\leq m} |f_{X_j}(\Lambda)| \geq c \|\Lambda\|^{|W_\gj|}\qquad (\Lambda \in \gj_\C^*)\, .\end{equation} 
Let now $z_1, \dots, z_m \in \cZ(\gg_\C)$ be such that $\chi_\Lambda(z_j)= f_{X_j}(\Lambda)$ for all 
$\Lambda\in \gj_\C^*$. Thus, combining \eqref{claim zwei} and \eqref{claim drei} we obtain the claim \eqref{claim eins} for $k=|W_\gj|$. 

\par The preceding reasoning now implies the following parameter independent version of Theorem~\ref{theo-constterm}:

\begin{theo}[Uniform constant term approximation]\label{theo-uniform}
Let $N\in\N$, $I\subset S$ and $\cC_I$ be a compact subset of $\ga_I^{--}$. Let $w_I\in\cW_I$ and 
$w={\bf m}(w_I)\in \cW$. Then there exist $\varepsilon>0$ and a continuous semi-norm $p$ on $C_{temp,N}^\infty(Z)$ such that, for all $f\in\cA_{temp,N}(Z:\Lambda)$, $\Lambda\in \gj_\C^*/W_\gj$: 
$$
\begin{array}{l}
(a_Z\exp(tX))^{-\rho_Q}\vert f(g a_{Z}\exp(tX)w\cdot z_0)-f_I(g  a_Z\exp(tX)w_I\cdot z_{0,I})\vert\\
\\
\leq e^{-\varepsilon t}(1+\Vert\log a_Z\Vert)^Np(f),
\qquad a_Z\in A_Z^-,X\in \cC_I, g\in\Omega, t\geq 0\,.
\end{array}
$$
Moreover, let $N_1:=\max_{\Lambda} \dim(\bD(Z_I)/\bD(Z_I)\mu_I(\cI_\Lambda))\in \N$ and  $q$ be a continuous semi-norm on $C_{temp,N+N_1}^\infty(Z_I)$. Then there exists a continuous semi-norm $p$ on $C_{temp,N}^\infty(Z)$ such that: 
$$q(f_I)\leq p(f), \qquad f\in\cA_{temp,N}(Z:\Lambda),  \Lambda\in \gj_\C^*/W_\gj\,.$$
\end{theo}

\appendix

\section{Rapid convergence}\label{sect-appA}
\begin{defi}\label{def-convrapidRd}
Let $a\geq 0$ and $(x_s)$ be a family of elements of a normed vector space with $s\in [a,+\infty[$. One says that $(x_s)$ converges rapidly to $l$ if 
\begin{res-nn}
\begin{center}
there exist $\varepsilon>0, C>0, s_0\in [a,+\infty[$ such that, for any $s\geq s_0$,\\
$\Vert x_s-l\Vert\leq Ce^{-\varepsilon s}$.
\end{center}
\end{res-nn}
To shorten, we will write $x_s\xrightarrow[s\to\infty]{rapid} l$.
\end{defi}
\begin{lem}\label{lem-rapid-diffeoRd}
Let $a\geq 0$, $E$, $F$ be two Euclidean spaces and $l\in E$. Let $\phi$ be an $F$-valued map of class $C^1$ 
on a neighborhood $U$ of $l$ and such that the differential $d\phi(l)$ of $\phi$ at $l$ is injective. If $(x_s)_{s\in [a,+\infty[}$ 
is a family of elements of $E$ such that $\phi(x_s)\xrightarrow[s\to\infty]{rapid} \phi(l)$ and $(x_s)$ converges to $l$ when 
$s$ tends to $+\infty$, then 
$$x_s\xrightarrow[s\to\infty]{rapid} l\,.$$
\end{lem}
\begin{proof} Choose a left inverse $A\in \Hom (F,E)$ to $d\phi(l)$ and replace $\phi$ by $A\circ \phi$. In this 
way we reduce to the case where $E=F$ with $d\phi(l)$ an isomorphism. By the inverse function theorem we may, after 
shrinking $U$, assume further that $\phi: U \to E$ is diffeomorphic onto its open image $\phi(U)\subset E$. 
 Applying the Taylor expansion of $\phi^{-1}$ at $\phi(l)$, one has for $s$ large enough such that $x_s\in V$:
$$
\begin{array}{rcl}
\Vert x_s-l\Vert & = & \Vert \phi^{-1}(\phi(x_s))-\phi^{-1}(\phi(l))\Vert\\
&\leq & \Vert d\phi^{-1}(\phi(l))\Vert\,
\Vert \phi(x_s)-\phi(l)\Vert + o(\Vert \phi(x_s)-\phi(l)\Vert)\,.
\end{array}$$
Our claim follows from the rapid convergence of $(\phi(x_s))$.
\end{proof}

\begin{defi}\label{def-convrapidX}
Let $a\geq 0$, $X$ be a $d$-dimensional smooth manifold and $(x_s)_{s\in [a,+\infty[}$ be a family of elements of $X$. One says that $(x_s)$ converges rapidly in $X$ if there exist $l\in X$ and a chart $(U,\phi)$ around $l$ such that:
\begin{res-nn}
\begin{center}
$(\phi(x_s))$ converges rapidly to $\phi(l)$\,.
\end{center}
\end{res-nn}
\end{defi}

\begin{rem}\label{rem-indep-chart}
This notion is independent of the choice of the chart $(U,\phi)$. Indeed, let $(\tilde{U},\tilde{\phi})$ 
be another chart around $l$. Then, from Lemma~\ref{lem-rapid-diffeoRd}, $((\phi\circ\tilde{\phi}^{-1})^{-1}(\phi(x_s)))$ converges 
rapidly to $\tilde{\phi}(l)$ which means that $(\tilde{\phi}(x_s))$ converges rapidly to $\tilde{\phi}(l)$.
Also if $\Psi:X\to Y$ is a differentiable map between $C^\infty$ manifolds and $(x_s)$ converges rapidly to $x$ in $X$, 
then $\Psi((x_s))$ converges rapidly to $\Psi(x)$ in $Y$.
\end{rem}

\section{Real points of elementary group actions}

We assume that $\sG$ is a reductive group defined over $\R$ and let $\sH$
be an $\R$-algebraic subgroup of $\sG$.  We form the homogeneous space $\sZ=\sG/\sH$ and our concern is 
to what extent $\sZ(\R)$ coincides with $\sG(\R)/\sH(\R)$.   

\par We say that $\sG$ is anisotropic provided $\sG(\R)$ is compact and recall from \cite[Proposition~13.1]{kk}
the following fact:

\begin{lem}\label{lem-B1}  If $\sG$ is anisotropic, then $\sZ(\R)= \sG(\R)/\sH(\R)$. \end{lem}

\par In the sequel, we assume that $\sG$ is a connected elementary group (defined over $\R$), that is: 
\begin{itemize}
\item  $\sG = \sM \sA$ for normal $\R$-subgroups $\sA$ and $\sM$, 
\item $\sM$ is anisotropic, 
\item $\sA$ is a split torus, i.e., $\sA(\R)\simeq (\R^\times)^n$. 
\end{itemize}

Consider now $\sZ=\sG/\sH$,
with $\sG$ elementary. We set $\sM_{\sH}:= \sM \cap \sH$ and, likewise, $\sA_{\sH}:=\sA\cap \sH$. 
Furthermore, we set $\sA_{\sZ}:= \sA/ \sA_H$ and $\sM_{\sZ}:= \sM/\sM_H$, which we view as subvarieties 
of $\sZ$.  From Lemma \ref{lem-B1}, we already know that $\sM_{\sZ}(\R)= \sM(\R)/\sM_H(\R)$. 
Consider now the fiber bundle 
$$ \sA_{\sZ} \to \sZ \to \sG/ \sH \sA$$
and take real points
\begin{equation}\label{eq-fiber} 
\sA_{\sZ}(\R) \to \sZ(\R) \to (\sG/ \sH \sA)(\R)\, .
\end{equation} 
We claim that the natural map 
\begin{equation} \label{B-mult} 
\sM_{\sZ}(\R) \times \sA_{\sZ}(\R) \to {\sZ}(\R)
\end{equation}
is surjective. In fact, observe that $\sG/ \sH \sA\simeq \sM/ (\sM \cap (\sH \sA))$ is homogeneous for the 
anisotropic group $\sM$.  Hence, $(\sG/ \sH \sA)(\R) \simeq \sM(\R) / (\sM \cap (\sH\sA))(\R)$
and our claim follows from (\ref{eq-fiber}).

\par 
We remain with the determination of the fiber of the map \eqref{B-mult}.  
Since $\sM$ and $\sA$ commute, we obtain with 
$$\hat \sM_{\sH}:=\{ m \in \sM \mid  m \sH \in \sA_{\sZ}\subset \sZ\}$$
a closed $\R$-subgroup of $\sM$, which acts on $\sA_{\sZ}$ by morphisms (translations). 
The kernel of this 
action is $\sM_{\sH}$ and this identifies $\sM_{\sH}$ as a normal subgroup of $\hat \sM_{\sH}$. In particular,
we obtain an embedding 
$\hat \sM_{\sH}/ \sM_{\sH}\to \sA_{\sZ}$ and, taking real points, we obtain, as $\sM$ is anisotropic and $\hat \sM_{\sH}$ is closed in $\sM$, a closed embedding  
$$ F_{\sM(\R)}:=\hat \sM_{\sH}(\R)/\sM_{\sH}(\R)\to \sA_{\sZ}(\R)\, .$$
The image of $F_{\sM(\R)}$ is compact, hence, a $2$-group of $\sA_{\sZ}(\R)\simeq (\R^\times)^k$. 
In summary, we have shown: 
\begin{prop}\label{prop-B} 
Let $\sZ=\sG/\sH$ be a homogeneous space for an elementary group $\sG=\sM \sA$ with respect to an $\R$-algebraic subgroup $\sH$. 
Then $F_{\sM(\R)}$ is a finite $2$-group and the map 
$$ [\sM(\R)/ \sM_H(\R)]\times^{F_{\sM(\R)}} \sA_{\sZ}(\R) \to \sZ(\R), \ \ [m\sM_{\sH}(\R), a_Z]\mapsto m a_Z$$
is an isomorphism of real manifolds. 
\end{prop}

\begin{cor}\label{cor-B} 
Under the assumptions of Proposition \ref{prop-B}, 
the $\sG(\R)$-orbits in $\sZ(\R)$ 
are in bijection with $\sA_{\sZ}(\R)_2/ F_{\sM(\R)}$, where $\sA_{\sZ}(\R)_2$ is the group of $2$-torsion points
in $\sA_{\sZ}(\R)$.  
The isomorphism is given explicitly by: 
$$ \sA_{\sZ}(\R)_2/ F_{\sM(\R)}  \to \sG(\R)\backslash \sZ(\R), \ \ F_{\sM(\R)} a_Z \mapsto \sG(\R)  a_Z\, .$$ 
\end{cor}

\section{Invariant differential operators on $Z$ and $Z_I$ (by Rapha\"el Beuzart-Plessis)}\label{app-RBP}

In the beginning we let  $Z=G/H$ be a general homogeneous space attached to a Lie group $G$ 
and a closed subgroup $H\subset G$. A bit later we specialize to real spherical spaces as in the main body of the text. 
Our concern is with the algebra of $G$-invariant differential operators
$\D(Z)$ and we start with a recall of the standard description of $\D(Z)$ in terms of the universal enveloping algebra
$\cU(\gg)$ of $\gg_\C$.  As usual, we denote the right regular representation of $G$ on $C^\infty(G)$ by $R$ and, by slight abuse of notation, the induced action of the enveloping algebra $\cU(\gg)$ by the same letter; in symbols: 
$$R: \cU(\gg) \to \End (C^\infty(G))\, .$$
Now, for an element $u \in \cU(\gg)$, the operator $R(u)$ descends to a differential operator on $Z$ if and only if 
$u\in \cU_H(\gg)$, where 
$$\cU_H(\gg):=\{ u \in \cU(\gg)\mid \Ad(h)u - u \in \cU(\gg)\gh,\quad h\in H\}\, .$$
Notice that  $\cU_H(\gg)\subset \cU(\gg)$ is a subalgebra of $\cU(\gg)$ which features
$\cU(\gg)\gh\subset \cU_H(\gg)$ as a two-sided ideal. 
The following Lemma goes back to Helgason in case there exists an $\Ad(H)$-stable vector complement 
to $\gh$ in $\gg$. The general case is an easy adaption and probably known to a larger part in the community. 
Since we could not find a reference we include a proof. 

\begin{lem} The right regular action induces a natural isomorphism
\begin{equation} \label{diff1} \D(Z)\simeq \cU_H(\gg)/ \cU(\gg)\gh\,,\end{equation} 
\end{lem}

\begin{proof}
(Compare \cite[Proof of Lemma 16]{Hel}) 	
Let $\pi: G\to Z$ be the natural projection. For every function $f\in C^\infty(Z)$, we set $\pi^*f=f\circ \pi\in C^\infty(G)$. The map $f\mapsto \pi^*f$ induces an isomorphism $C^\infty(Z)\simeq C^\infty(G)^H$ with the space of $H$-right-invariant functions in $C^\infty(G)$. Let $u\in \cU_H(\gg)$. Then, $R(u)$ preserves $C^\infty(G)^H$ and therefore induces an endomorphism of $C^\infty(Z)$ obviously given by a $G$-invariant differential operator. Thus, we have an algebra homomorphism
\begin{equation}\label{eqDO 1}
u\in \cU_H(\gg)\mapsto D_u\in \D(Z)
\end{equation}
characterized by the property that
\begin{equation}\label{eqDO 1bis}
\displaystyle R(u)\pi^*f=\pi^*(D_u f)
\end{equation}
for every $u\in \cU_H(\gg)$, $f\in C^\infty(Z)$. It remains to show that this morphism is onto with kernel $\cU(\gg) \gh$.

First we show that 
\begin{equation}\label{eqDO 3}
\displaystyle \{u\in \cU(\gg)\mid R(u)\pi^*C^\infty(Z)=0 \}=\cU(\gg)\gh.
\end{equation}
Note that this fact immediately implies that the kernel of \eqref{eqDO 1} is $\cU(\gg)\gh$.

Choose a complementary subspace $\mathfrak{m}$ of $\gh$ in $\gg$ and let
$$\displaystyle \Symm: S(\gg)\to \cU(\gg)$$
be the symmetrization map. By Poincar\'e-Birkhoff-Witt, we have
\begin{equation*}
\displaystyle \cU(\gg)=\Symm(S(\gm))\oplus \cU(\gg)\gh.
\end{equation*} 
Hence, we just need to show that if $u\in \Symm(S(\gm))$ is such that $R(u)\pi^*C^\infty(Z)=0$ then $u=0$. Let $u$ be such an element. It can be written as $u=\Symm(v)$ for some $v\in S(\gm)$. If $U$ is a sufficiently small open neighborhood of $0$ in $\gm$, the map
$$\displaystyle \psi: X\in U\mapsto \pi(\exp(X))\in Z$$
is an open embedding. Therefore, 
\begin{equation}\label{eqDO 4}
\psi^* C^\infty(Z)=C^\infty(U).
\end{equation}
On the other hand, by a well-known characterization of the symmetrization map (see \cite[eq. (3.9)]{Hel}), we have
$$\displaystyle (R(u)\pi^*f)(1)=(\partial(v)\psi^* f)(0)$$
for every $f\in C^\infty(Z)$ where $\partial(v)$ is the differential operator with constant coefficients on $\gm$ associated to $v$. By \eqref{eqDO 4}, this last equality implies $v=0$ hence $u=0$ and this ends the proof of \eqref{eqDO 3}.

It only remains to prove that \eqref{eqDO 1} is surjective. Let $D\in \D(Z)$. As $\psi$ is an open embedding, there exists $v\in S(\gm)$ such that
$$\displaystyle (Df)(z_0)=(\partial(v)\psi^* f)(0)$$
for every $f\in C^\infty(Z)$ where $z_0=\pi(1)$ is the natural base-point of $Z$. Set $u=\Symm(v)$. As before, the above identity can be rewritten as
$$\displaystyle (Df)(z_0)=(R(u)\pi^* f)(1).$$
Since $D$ is $G$-invariant, it follows that
\[\begin{aligned}
\displaystyle (Df)(g z_0)=(DL(g^{-1})f)(z_0)=(R(u) \pi^* L(g^{-1})f )(1)=(L(g^{-1})R(u)\pi^* f)(1)=(R(u)\pi^*f)(g)
\end{aligned}\]
for all $f\in C^\infty(Z)$ and $g\in G$. Otherwise said, we have
\begin{equation}\label{eqDO 5}
\displaystyle \pi^*(Df)=R(u)\pi^*(f),\;\; f\in C^\infty(Z).
\end{equation}
Since $R(h)\pi^*(f)=\pi^*(f)$ for every $h\in H$ and $f\in C^\infty(Z)$, we deduce that $R(\Ad(h)u-u)\pi^*C^\infty(Z)=0$, hence $\Ad(h)u-u\in \cU(\gg)\gh$ by \eqref{eqDO 3}. This shows that $u\in \cU_H(\gg)$ and comparing \eqref{eqDO 1bis} with \eqref{eqDO 5} we have $D=D_u$. Therefore, the map \eqref{eqDO 1} is surjective.

\end{proof}

  For $u \in \cU_H(\gg)/ \cU(\gg)\gh$, we denote by 
$R_H(u)\in \D(Z)$ the correponding invariant differential operator.
Suppose furthermore that there is a subalgebra $\gb\subset\gg$ such that $\gg = \gb +\gh$ (not necessarily direct).
Then Poincar\'e-Birkhoff-Witt (PBW) implies that $\cU(\gg)= \cU(\gb) + \cU(\gg)\gh$ and 
setting $\cU_H(\gb)= \cU(\gb)\cap \cU_H(\gg)$, we obtain 
from \eqref{diff1} an isomorphism

\begin{equation} \label{diff2} \D(Z)\simeq \cU_H(\gb)/ \cU(\gb)(\gh\cap \gb)\, .\end{equation} 

\begin{rem}  \label{rem appD}
(a) Recall that we expressed by $H_0$ the identity component of $H$. It is 
then clear that $\cU_H(\gg)\subset \cU_{H_0}(\gg)$.  Hence, we obtain 
from $\Lie H= \Lie H_0$ and \eqref{diff1} that 
$$\D(Z)\subset\D(G/H_0)$$  
naturally. Moreover, we record that
\begin{eqnarray*}\cU_{H_0}(\gg)&=&\{u\in\cU(\gg)\mid [X,u]\in\cU(\gg)\gh,\quad X\in \gh\}\\
&=& \{u\in\cU(\gg)\mid Xu\in\cU(\gg)\gh,\quad X\in \gh\}\, .\end{eqnarray*}

(b) Assume that $G=\underline{G}(\R)$ is the group of $\R$-points of a linear algebraic group $\underline{G}$ over $\R$. Let $H_{alg}$ be the Zariski closure of $H$ in $G$ and assume that $H_{alg}$ and $H$ have the same Lie algebra (this happens, e.g., if $H$ has finite index in the group of $\R$-points of an algebraic subgroup of $\underline{G}$). Then, by \eqref{diff2},
$$\D(Z)=\D(G/H_{alg})\,.$$
\end{rem}

\par We now return to the spherical setup and request from now that $Z=G/H$ is real spherical and unimodular where as in the main body of the text, $G=\underline{G}(\R)$ is the group of $\R$-points of a connected real reductive group and $H$ is open in the $\R$-points of an algebraic subgroup of $\underline{G}$. 
The topic of this section is then to study the relationship of $\D(Z)$ to $\D(Z_I)$ for $I\subset S$. 
Recall that the authors of this paper have defined $H_I$ to be connected.  We abbreviate notation and write $R_I$ for $R_{H_I}$ 
and $\cU_I(\gg)=\cU_{H_I}(\gg)$ etc.  With 
$$\gb:=\ga +\gm +\gu\,,$$
 we obtain
a subalgebra of $\gg$ such that $\gg=\gb +\gh_I$ for all $I\subset S$. Further, we have $\gb\cap \gh_I =
\ga_H + \gm_H=:\gb_H$.  In particular, we obtain an algebra isomorphism 
\begin{equation} \label{diff3}p_I: \cU_I(\gb)/ \cU(\gb)\gb_H\to \D(Z_I), \ \ u\mapsto R_I(u)\,.
\end{equation}

Via this algebra isomorphism, we identify from now on $\D(Z_I)$ with $\cU_I(\gb)/ \cU(\gb)\gb_H$.
Remark that, as $A_I$ normalizes $H_I$, we obtain a natural inclusion $S(\ga_I)\hookrightarrow \D(Z_I)$
induced from the right action of $A_I$ on $Z_I$. Note that $Z_S=G/H_0$ and $\ga_S=\ga_{Z,E}$.

\begin{lem}\label{appC lem 1}
The symmetric algebra $S(\ga_S)$ embeds in the center of $\D(Z_S)$.
\end{lem}

\begin{proof} 
By slight abuse of notation, let us denote by $\sH$ the algebraic closure of $H_0$ in $\sG$ and let $H=\sH(\R)$. In view of Remark \ref{rem appD}(b), we may replace $H_0$ by $H$ in the following. 

\par Let $\sZ(\C)=\sG(\C)/\sH(\C)$. Since $Z$ is unimodular, $\sZ(\C)$ is a quasi-affine algebraic variety (see Lemma \ref{lem unim-quasiaff}) and there is a natural embedding
$$\displaystyle \D(Z)\hookrightarrow \End_G(\C[\sZ(\C)])\simeq \bigoplus_{V} \End(V^{H})\,,$$
where the direct sum runs over all isomorphism classes of algebraic finite dimensional irreducible $G$-modules. 
Moreover, the image of $S(\ga_S)$ in $\End(V^{H})$ by this morphism corresponds to the natural action of $\ga_S$ on $V^H$. 
Therefore, we only need to show that this action is scalar for every finite dimensional irreducible $G$-module $V$. Set $V(U)=\gu V$. 
Then $V(U)$ is a proper $Q$-submodule of $V$ and the quotient $V/V(U)$ is an irreducible $L$-module on which the split center $\ga_L$ acts by a certain weight $\mu\in \ga_L^*$. 
Identify $\ga_S$ with a subspace of $\ga_L$ through the choice of a splitting of $\ga_Z$ in $\ga_L$. 
Then the claim would follow if we can show that the only weight of $\ga_S$ in $V^H$ is the restriction of $\mu$. 
We have
\begin{equation}\label{eq 2}
V^H\cap V(U)=0.
\end{equation}
Indeed, if $v\in V^H\cap V(U)$ then $\sQ(\C)\sH(\C).v\subset V(U)$ and, as $\sQ(\C) \sH(\C)$ is Zariski dense in $\sG(\C)$, this implies that the $\sG(\C)$-invariant subspace generated by $v$ is included in $V(U)$ hence $v=0$ since $V$ is irreducible and $V(U)\neq V$. By \eqref{eq 2}, the restriction of the projection $V\to V/V(U)$ yields an injective $\ga_S$-equivariant morphism $V^H\hookrightarrow V/V(U)$. The result follows. 
\end{proof}

\par In the sequel, we view $ \cU_I(\gb)/ \cU(\gb)\gb_H$ as a subspace 
of $\cU(\gb)/ \cU(\gb)\gb_H$, which is naturally a module for $A/A_H$, hence for $A_Z$. 
In particular, we can speak of the $\ga_Z$-weights of an element in  $ \cU_I(\gb)/ \cU(\gb)\gb_H$. 
Recall that $\ga_S=\ga_{Z,E}\subset \ga_I$  for all $I\subset S$. 

\par
Let $I\subset S$ and $(\ga_I^*)^+$ be the cone of elements $\lambda\in \ga_I^*$ such that 
$$\displaystyle \lambda(X)\leqslant 0, \quad  X\in \ga_I^-\,.$$
Let $u_S\in \cU_S(\gb)/\cU(\gb)\gb_H$ and $\displaystyle u_S=\sum_{\lambda\in \ga_I^*}u_{S,\lambda}$ be its decomposition (in $\cU(\gb)/\cU(\gb)\gb_H$) into $\ga_I$-eigenvectors. Let $\bW_I(u_S)$ be the set of $\lambda\in \ga_I^*$ such that $u_{S,\lambda}\neq 0$. Then there exists a unique minimal subset $\bW_I(u_S)_{max}$ of $\bW_I(u_S)$ such that
$$\displaystyle \conv(\bW_I(u_S)_{max}+(\ga_I^*)^+)=\conv(\bW_I(u_S)+(\ga_I^*)^+)\,,$$
where $\conv(D)$ denotes the convex hull of a subset $D\subset \ga_I^*$ (indeed, $\bW_I(u_S)_{max}$ is just the set of extremal points of $\conv(\bW_I(u_S)+(\ga_I^*)^+)$; this follows from a version of the Krein--Milman theorem for convex subsets invariant by a cone, see, e.g., \cite{gl}).

\begin{lem}\label{appC lem 2}
Let $\lambda_{max}\in \bW_I(u_S)_{max}$. Then $u_{S,\lambda_{max}}\in \cU_I(\gb)/\cU(\gb)\gb_H$.
\end{lem}

\begin{proof} 
Choose, for every $\lambda\in \ga_I^*$, a lift $\widetilde{u}_{S,\lambda}\in \cU_S(\gb)$ of $u_{S,\lambda}$, which is again an $\ga_I$-eigenvector of weight $\lambda$ and with $\widetilde{u}_{S,\lambda}=0$ if $u_{S,\lambda}=0$. Set
$$\displaystyle \widetilde{u}_S=\sum_{\lambda\in \ga_I^*}\widetilde{u}_{S,\lambda}$$
(a lift of $u_S$). Then we want to show that $\widetilde{u}_{S,\lambda_{max}}\in \cU_I(\gb)$. By the choice of $\lambda_{max}$, there exists $X\in \ga_I^{--}$ such that $\lambda(X)<\lambda_{max}(X)$ for every $\lambda\in \ga_I^*$ with $\widetilde{u}_{S,\lambda}\neq 0$ and $\lambda\neq \lambda_{max}$. Therefore, we have
$$\displaystyle \lim\limits_{t\to \infty} e^{-t\lambda_{max}(X)}e^{t\ad X}\widetilde{u}_{S}=\widetilde{u}_{S,\lambda_{max}}\,.$$
Since $\lim\limits_{t\to \infty}e^{t\ad X}\gh=\gh_I$ in the Grassmannian $\Gr(\gg)$, we easily check that for every $n\geqslant 0$ the limit $\lim\limits_{t\to \infty} e^{t\ad X}\cU_S(\gg)_{\leqslant n}$ in the Grassmannian $\Gr(\cU(\gg)_{\leqslant n})$ (which always exists) is a subspace of $\cU_I(\gg)_{\leqslant n}$. Since $\widetilde{u}_{S}\in \cU_S(\gg)$, this shows that $\widetilde{u}_{S,\lambda_{max}}\in \cU_I(\gg)\cap \cU(\gb)=\cU_I(\gb)$. 
\end{proof}

Notice that, for every $I\subset S$, we have a morphism $\cZ(\gg)\to \D(Z_I)$ induced by the ``right" action of $\cZ(\gg)$ on smooth functions on $Z_I$. We can now state the main theorem of this appendix.

\begin{theo}\label{lem-defmu_0}
For every $u_S\in \cU_S(\gb)/\cU(\gb)\gb_H$, the limit
\begin{equation}\label{eq limitmuI}
\displaystyle u_I=\lim\limits_{t\to \infty}e^{t\ad X}u_S
\end{equation}
exists in $\cU(\gb)/\cU(\gb)\gb_H$ for every $X\in \ga_I^{--}$ and is independent of $X$. The map $u\mapsto u_I$ induces an injective morphism of algebras
$$
\begin{array}{rcl}
\displaystyle \mu_I:\cU_S(\gb)/\cU(\gb)\gb_H=\D(Z_S)&\longrightarrow& \D(Z_I)=\cU_I(\gb)/\cU(\gb)\gb_H.
\end{array}$$
Moreover, the following assertions hold:
\begin{enumerate}[(i)]
\item \begin{enumerate}[(a)]
\item the $\ga_Z$-weights of $u_S$ are  non-positive on $\ga_Z^-$,
\item the $\ga_Z$-weights of $u_I-u_S$ are negative on $\ga_I^{--}$.
\end{enumerate}
\item The morphism $\mu_I$ fits into commutative squares
$$\displaystyle \xymatrix{ \cZ(\gg) \ar@{=}[r] \ar[d] & \cZ(\gg) \ar[d] \\ \D(Z_S) \ar[r] & \D(Z_I)}\quad \textrm{and}\quad\displaystyle \xymatrix{ S(\ga_S) \ar[r] \ar[d] & S(\ga_I) \ar[d] \\  \D(Z_S) \ar[r] & \D(Z_I)}\,,$$
where the vertical arrows in the first and second diagrams are the natural ones.
\end{enumerate}
\end{theo}

\begin{proof}
By Lemma \ref{appC lem 1} (applied to $Z_I$ instead of $Z$) and Lemma \ref{appC lem 2}, we see that, for any nonzero $u_S\in \cU_S(\gb)/\cU(\gb)\gb_H$, we have $\bW_I(u_S)_{max}=\{ 0\}$ (in particular, $u_{S,0}\neq 0$). This implies that the limit in \eqref{eq limitmuI} exists, is independent of $X$ and is nonzero if $u_S\neq 0$. This readily implies that $\mu_I$ is a monomorphism of algebras. Moving on to (i), we deduce (a) and (b)  from the fact that the limit \eqref{eq limitmuI} exists. 

\par The second square of assertion (ii) is commutative since the image of $S(\ga_S)$ in $\cU(\gb)/\cU(\gb)\gb_H$ is obviously in the $0$-weight space of $\ga_I$. It only remains to show that the first square of (ii) is commutative. Let $z\in \cZ(\gg)$. Let $\widetilde{z}_S\in \cU(\gb)$ and $\widetilde{z}^S\in \cU(\gg)\gh$ be such that $z=\widetilde{z}_S+\widetilde{z}^S$. Then $\widetilde{z}_S\in \cU_S(\gb)$ and through our identification $\D(Z_S)\simeq \cU_S(\gb)/\cU(\gb)\gb_H$, $z$ gets mapped to the image $z_S$ of $\widetilde{z}_S$ in $\cU_S(\gb)/\cU(\gb)\gb_H$. By (i), up to translating $\widetilde{z}_S$ by an element of $\cU(\gb)\gb_H$, we may assume that the limit
$$\displaystyle \widetilde{z}_I=\lim\limits_{t\to \infty}e^{t\ad X}\widetilde{z}_S$$
exists in $\cU(\gb)$ for every $X\in \ga_I^{--}$ and that it is independent of $X$. Moreover, $\widetilde{z}_I\in \cU_I(\gb)$ and its image $z_I$ in $\cU_I(\gb)/\cU(\gb)\gb_H$ coincides with the image of $z_S$ by $\mu_I$. As $z$ is fixed by any inner automorphism, the limit
$$\displaystyle \widetilde{z}^I=\lim\limits_{t\to \infty}e^{t\ad X}\widetilde{z}^S$$
also exists in $\cU(\gg)$ for every $X\in \ga_I^{--}$, is independent of $X$ and $z=\widetilde{z}_I+\widetilde{z}^I$. Since $\widetilde{z}^S\in \cU(\gg)\gh$ and $\lim\limits_{t\to \infty} e^{t \ad X} \gh=\gh_I$ in the Grassmannian $\Gr(\gg)$, we have $\widetilde{z}^I\in \cU(\gg)\gh_I$. Therefore, by definition of the identification $\D(Z_I)\simeq \cU_I(\gb)/\cU(\gb)\gb_H$, $z_I$ is also the image of $z$ in $\D(Z_I)$. The commutativity of the first square follows.
\end{proof}

\noindent P. Delorme, Universit\'e d'Aix-Marseille, I2M, UMR 7373, 13453
Marseille, France, {\tt patrick.delorme@univ-amu.fr}\\
\\
\noindent B. Kr\"otz, Universit\"at Paderborn, Institut f\"ur Mathematik, Warburger Str. 100, 33098 Paderborn, Germany, 
{\tt bkroetz@gmx.de}\\
\\ 
\noindent S. Souaifi, Universit\'e de Strasbourg, IRMA, UMR 7501, 7 rue Ren\'e Descartes, 67084 Strasbourg Cedex, France, {\tt souaifi@math.unistra.fr} \\
\\
\noindent R. Beuzart-Plessis, Aix Marseille Univ, CNRS, Centrale Marseille, I2M, Marseille, France
{\tt rbeuzart@gmail.com}


\begin{thebibliography}{10}

\bibitem{AM}  M.F.~Atiyah and I.G.~ Macdonald.
\newblock Introduction to Commutative algebra.
\newblock Perseus Books Publishing, 1969.

\bibitem{vdBS} E.P.~ van den Ban and H.~ Schlichtkrull. 
\newblock  The Plancherel decomposition for
a reductive symmetric space. I-II.
\newblock{Invent. Math.} {\bf 161(3)}, 453--566 and 567--628, 2005.


\bibitem{bernstein-planch}
J.~N. Bernstein.
\newblock On the support of {P}lancherel measure.
\newblock {\em J. Geom. Phys.} {\bf 5(4)}, 663--710, 1988.

\bibitem{bk}
J.~N. Bernstein and B.~Kr{\"o}tz.
\newblock Smooth {F}r\'echet globalizations of {H}arish-{C}handra modules.
\newblock {\em Israel J. Math.} {\bf 199(1)}, 45--111, 2014.

\bibitem{br} J.~N. Bernstein and A.~Reznikov. 
\newblock{\em Estimates of automorphic functions.} 
\newblock Moscow Math. J. {\bf 4}(1), 19--37, 2004.


\bibitem{brion-luna}
M.~Brion and D.~Luna.
\newblock Espaces homog\`enes sph\'eriques.
\newblock {\em Invent. Math.} {\bf 84(3)}, 617--632, 1986.

\bibitem{carmona}
J.~Carmona.
\newblock Terme constant des fonctions temp\'er\'ees sur un espace sym\'etrique
  r\'eductif.
\newblock {\em J. Reine Angew. Math.} {\bf 491}, 17--63, 1997.

\bibitem{cas}
W.~Casselman.
\newblock Canonical extensions of {H}arish-{C}handra modules to representations
  of {$G$}.
\newblock {\em Canad. J. Math.} {\bf 41(3)}, 385--438, 1989.


\bibitem{cls} D.~ Cox, J.~ Little and H.~ Schenck.
\newblock{\em Toric varieties.}
\newblock Graduate Studies in Mathematics {\bf 124}. American Mathematical Society, Providence, RI, 2011.

\bibitem{dks} T.~Danielson, B. ~Kr\"otz and H.~Schlichtkrull
\newblock Decomposition theorems for triple spaces.
\newblock {\em Geom. Dedicata} {\bf 174}, 145--154, 2015. 

\bibitem{Delorme} P.~ Delorme.
\newblock Formule de Plancherel pour les espaces sym\'etriques r\'eductifs.
\newblock {\em Ann. of Math.} {\bf 147}, 417--452, 1998.



\bibitem{planch-sph}
P.~Delorme, F.~Knop, B.~Kr{\"o}tz, and H.~Schlichtkrull.
\newblock Plancherel theory for real spherical spaces: Construction of the
  {B}ernstein morphisms.
\newblock \href{https://arxiv.org/pdf/1807.07541.pdf}{arXiv1807.07541}, 2018.


\bibitem{gl}
J. Goubault-Larrecq.
\newblock A Cone-Theoretic Krein-Milman Theorem.
\newblock Research Report LSV-08-18, Laboratoire Sp\'ecification et V\'erification, ENS Cachan, France, June 2008. 8 pages.

\bibitem{hc56} Harish-{C}handra. 
\newblock{The characters of semisimple Lie groups.}
\newblock{\em Trans. Amer. Math. Soc.} {\bf 83}, 98--163, 1956.

\bibitem{hcd1} Harish-{C}handra.
\newblock Discrete series for semisimple Lie groups. I. Construction of invariant eigendistributions. 
\newblock {\em Acta Math.} {\bf 113}, 241--318, 1965.

\bibitem{hcd2} Harish-{C}handra.
\newblock Discrete series for semisimple Lie groups. II. Explicit determination of the characters. 
\newblock {\em Acta Math.} {\bf 116}, 1-111, 1966. 


\bibitem{hc1}
Harish-{C}handra.
\newblock Harmonic analysis on real reductive groups. {I}. {T}he theory of the
  constant term.
\newblock {\em J. Functional Analysis} {\bf 19}, 104--204, 1975.

\bibitem{hc2}
Harish-{C}handra.
\newblock Harmonic analysis on real reductive groups. {II}. {W}avepackets in
  the {S}chwartz space.
\newblock {\em Invent. Math.} {\bf 36}, 1--55, 1976.

\bibitem{hc3}Harish-{C}handra.
\newblock Harmonic analysis on real reductive groups. III. The Maass-Selberg relations and the Plancherel formula.
\newblock {\em Ann. of Math.} {\bf 104(1)}, 117--201, 1976.

\bibitem{Hel} S.~ Helgason. 
\newblock{\em Differential operators on homogeneous spaces.} 
\newblock Acta Math. {\bf 102}, 239--299, 1959.

\bibitem{humphreys}
J. E. Humphreys.
\newblock{\em Linear Algebraic Groups.}
\newblock Springer, Graduate Texts in Mathematics {\bf 21}.

%\bibitem{kkms}
%G. Kempf, F.~Knudsen, F.~Mumford, and B.~Saint-Donat.
%\newblock {\em Toroidal embeddings. {I}}.
%\newblock Lecture Notes in Mathematics, Vol. {\bf 339}. Springer-Verlag, Berlin-New
  %York, 1973.

\bibitem{KV}
A.~W.~Knapp, D.~A.~Vogan.
\newblock {\em Cohomological induction and Unitary Representations}, 
\newblock Princeton University Press.


\bibitem{knop}
F.~Knop.
\newblock The asymptotic behavior of invariant collective motion.
\newblock {\em Invent. Math.} {\bf 16(1-3)}, 309--328, 1994.

\bibitem{knop-hc}
F.~Knop.
\newblock A {H}arish-{C}handra homomorphism for reductive group actions.
\newblock {\em Ann. of Math.} {\bf 140}, 253--288, 1994.

\bibitem{kk}
F.~Knop and B.~Kr{\"o}tz.
\newblock Reductive group actions.
\newblock \href{https://arxiv.org/abs/1604.01005v2}{arXiv:1604.01005v2}, 2016.

\bibitem{kkps}
F.~Knop, B.~Kr{\"o}tz, T.~Pecher and H. Schlichtkrull.
\newblock Classification of reductive real spherical pairs I. The simple case.
\newblock {\em Transformation Groups} {\bf 24(1)}, 67--114, 2019. 


\bibitem{kkps2}
F.~Knop, B.~Kr{\"o}tz, T.~Pecher and H. Schlichtkrull.
\newblock Classification of reductive real spherical pairs II. The semisimple case.
\newblock {\em Transformation Groups}, {\bf 24(2)}, 467 -- 510, 2019. 





\bibitem{kkss}
F.~Knop, B.~Kr{\"o}tz, E.~Sayag, and H.~Schlichtkrull.
\newblock Simple compactifications and polar decomposition of homogeneous real
  spherical spaces.
\newblock {\em Selecta Math. (N.S.)} {\bf 21(3)}, 1071--1097, 2015.

\bibitem{kkss-volgrowth}
F.~Knop, B.~Kr{\"o}tz, E.~Sayag, and H.~Schlichtkrull.
\newblock Volume growth, temperedness and integrability of matrix coefficients
  on a real spherical space.
\newblock {\em J. Funct. Anal.} {\bf 271(1)}, 12--36, 2016.

\bibitem{kks0}
F.~Knop, B.~Kr{\"o}tz, and H.~Schlichtkrull.
\newblock The local structure theorem for real spherical varieties.
\newblock {\em Compos. Math.} {\bf 151(11)}, 2145--2159, 2015.

\bibitem{kks}
F.~Knop, B.~Kr{\"o}tz, and H.~Schlichtkrull.
\newblock The tempered spectrum of a real spherical space.
\newblock {\em Acta Math.} {\bf 218(2)}, 319--383, 2017.

\bibitem{hpk} H. ~Kraft.
\newblock Basics from algebraic geometry. 
\href{http://docs.wixstatic.com/ugd/68dc5d_6dd3660bf09447168bafec1c2c348a95.pdf}{Online lecture notes}.

\bibitem{kkos}
B.~Kr{\"o}tz, J.~Kuit, E.~Opdam, and H.~Schlichtkrull.
\newblock The infinitesimal characters of discrete series for real spherical
  spaces.
\newblock To appear in GAFA,  2020.

\bibitem{geoc}
B.~Kr{\"o}tz, E.~Sayag, and H.~Schlichtkrull.
\newblock Geometric {C}ounting on {W}avefront {R}eal {S}pherical {S}paces.
\newblock {\em Acta Math. Sin. (Engl. Ser.)} {\bf 34(3)}, 488--531, 2018.

\bibitem{ks}
B.~Kr{\"o}tz and H.~Schlichtkrull.
\newblock Multiplicity bounds and the subrepresentation theorem for real
  spherical spaces.
\newblock {\em Trans. Amer. Math. Soc.} {\bf 368(4)}, 2749--2762, 2016.

\bibitem{ks-sanya}
B.~Kr{\"o}tz and H.~Schlichtkrull.
\newblock Harmonic {A}nalysis for {R}eal {S}pherical {S}paces.
\newblock {\em Acta Math. Sin. (Engl. Ser.)} {\bf 34(3)}, 341--370, 2018.


\bibitem{OM} T.~Oshima and T.~Matsuki.
\newblock A description of discrete series for semisimple symmetric spaces. 
\newblock {\em Adv. Stud. Pure Math.} {\bf 4}, 331--390, 1984. 

\bibitem{repka} J.~ Repka
\newblock{\em Tensor Products of Unitary Representations of $\SL_2(\R)$.}
\newblock Amer. J. Math. {\bf 100(4)}, 747--774, 1978.

\bibitem{sak-venk}
Y.~Sakellaridis and A.~Venkatesh.
\newblock Periods and harmonic analysis on spherical varieties.
\newblock  {\em Ast\'erisque} {\bf 396}, 2017.

\bibitem{wald} J.-L.~ Waldspurger.
\newblock La formule de Plancherel pour les groupes p-adiques (d'apr\'es Harish-Chandra).
\newblock {\em J. Inst. Math. Jussieu} {\bf 2(3)}, 235--333, 2003. 

\bibitem{wallach1}
N.~R. Wallach
\newblock {\em Real reductive groups. {I}}, volume {\bf 132} of {\em Pure and Applied
  Mathematics}.
\newblock Academic Press, Inc., Boston, MA, 1988.

\bibitem{wallach2}
N.~R. Wallach.
\newblock {\em Real reductive groups. {II}}, volume {\bf 132 -II} of {\em Pure and
  Applied Mathematics}.
\newblock Academic Press Inc., Boston, MA, 1992.

\end{thebibliography}
\end{document}